\documentclass[english,12pt,a4paper]{amsart}
\usepackage[utf8]{inputenc}
\usepackage{a4wide}
\usepackage{hyperref}
\usepackage{amssymb,amsmath,amsthm}
\usepackage{mathrsfs}
\usepackage{tikz}
\usetikzlibrary{positioning}
\usetikzlibrary{matrix,arrows}

\numberwithin{equation}{section}
\numberwithin{figure}{section}

\theoremstyle{plain}
\newtheorem{thm}{Theorem}
  \theoremstyle{plain}
  \numberwithin{thm}{section}
  \newtheorem{cor}[thm]{Corollary}
  \theoremstyle{plain}
  \newtheorem{lem}[thm]{Lemma}
  \theoremstyle{remark}
  \newtheorem{rem}[thm]{Remark}
    \theoremstyle{remark}

   \theoremstyle{plain}
  \newtheorem{prop}[thm]{Proposition}
  \newtheorem{assumption}[thm]{Assumption}
  \newtheorem{notat}[thm]{Notation}
  
  \theoremstyle{definition}
\newtheorem{definition}[thm]{Definition}

\newcommand{\norm}[1]{\left\| #1 \right\|}
\newcommand{\mklm}[1]{\left\{ #1 \right\}}
\newcommand{\eklm}[1]{\left\langle #1 \right\rangle}

\newcommand{\imult}{\mathbin{\lrcorner}}
\renewcommand{\d}{\,d}

\newcommand{\N}{{\mathbb N}}
\newcommand{\Z}{{\mathbb Z}}
\newcommand{\C}{{\mathbb C}}

\newcommand{\R}{{\mathbb R}}

\newcommand{\D}{{\mathcal D}}

\newcommand{\F}{{\mathcal F}}

\newcommand{\J}{{ \mathcal J}}

\newcommand{\M}{{\mathscr M}}

\renewcommand{\O}{{\mathcal O}}

\newcommand{\W}{{\mathcal W}}

\newcommand{\0}{{\rm 0}}

\renewcommand{\epsilon}{\varepsilon}

\newcommand \< {\langle}
\renewcommand \> {\rangle}

\DeclareMathOperator{\End}{End}

\DeclareMathOperator{\Tr}{Tr}
\DeclareMathOperator{\Td}{Td}

\DeclareMathOperator{\Ker}{Ker}

\DeclareMathOperator{\ch}{ch}
\DeclareMathOperator{\inc}{\mathrm{inc}}
\DeclareMathOperator{\ret}{\mathrm{ret}}

\newcommand{\bdm}{\begin{displaymath}}
\newcommand{\edm}{\end{displaymath}}
\newcommand{\bq}{\begin{equation}}
\newcommand{\eq}{\end{equation}}
\newcommand{\bqn}{\begin{equation*}}
\newcommand{\eqn}{\end{equation*}}
\newcommand{\Cinft}{{\rm C^{\infty}}}
\newcommand{\CT}{{\rm C^{\infty}_c}}

\renewcommand{\L}{{\rm L}}

\newcommand{\Scal}{{\mathcal S}}

\newcommand{\Spin}{\mathrm{Spin}}
\newcommand{\g}{{\bf \mathfrak g}}

\newcommand{\eps}{\varepsilon}

\renewcommand{\det}{\mathrm{det}\,}

\newcommand{\Crit}{\mathrm{Crit}}
\newcommand{\pr}{\mathrm{pr}}

\DeclareMathOperator{\supp}{supp\,}

\DeclareMathOperator{\Int}{Int}

\DeclareMathOperator{\gd}{\partial}

\begin{document}

\title{A Riemann-Roch formula for singular reductions by circle actions}

\author{Benjamin Delarue}
\email{bdelarue@math.upb.de}
\address{Institut f\"ur Mathematik, Universit\"at Paderborn, 33098 Paderborn, Germany}

\author{Louis Ioos}
\email{louis.ioos@cyu.fr}
\address{CY Cergy Paris Université,
Site de Saint-Martin,
2 avenue Adolphe Chauvin,
95300 Cergy-Pontoise, France}

\author{Pablo Ramacher}
\email{ramacher@mathematik.uni-marburg.de}
\address{Philipps-Universit\"at Marburg, Fachbereich Mathematik und Informatik, 35032 Marburg, Germany}

\date{\today}
\begin{abstract}
We compute a Hirzebruch-Riemann-Roch type  formula for the invariant Riemann-Roch number of a quantizable Hamiltonian $S^1$-manifold $(M,\omega,\J)$, allowing $0$ to be a singular value of the moment map $\J:M\to\R$.   Our formula represents an instance of the Guillemin-Sternberg principle, which states that quantization should commute with reduction. The conceptual novelty of our result is that the involved reduced system only depends on the symplectic data of $M$. To establish this, we derive a complete singular stationary phase expansion of the Witten integral without appealing to any kind of desingularization.  As a consequence, our formula expresses the invariant Riemann-Roch number purely in terms of  symplectic invariants of the singular symplectic quotient.  In particular, it involves a new explicit symplectic invariant of the singularities.
 \end{abstract}

\maketitle


\section{Introduction}

Let $(M,\omega,\J)$ be a compact connected symplectic manifold
equipped with a Hamiltonian action of a compact connected Lie group $G$
with moment map $\J:M \rightarrow \g^\ast$.
Then the associated
\emph{Marsden-Weinstein reduced space},
or simply the \emph{symplectic quotient}, is given by
\begin{equation}\label{M0}
\M_0:= \J^{-1}(\mklm{\0})/G\,.
\end{equation}
If $0$ is a regular value of the moment map, the symplectic quotient naturally
inherits the structure of a symplectic orbifold $(\M_0,\omega_0)$,
but in general, by a classical result of Sjamaar and Lerman \cite{SL91}, it is only a \emph{stratified symplectic space},
each smooth stratum being  naturally equipped with a symplectic structure.

This paper is devoted to the computation of the \emph{invariant Riemann-Roch
number} of $(M,\omega,\J)$  in terms of the geometry of the associated symplectic 
quotient \eqref{M0}. To introduce it, one has to impose the condition that
the cohomology class $[\omega]\in H^2(M,\Z)$ is integral, in which case
$(M,\omega)$ is called \emph{quantizable}. This condition is
equivalent to the existence of a Hermitian line bundle $(L,h^L)$
over $M$ equipped with a Hermitian connection $\nabla^L$ with
curvature $R^L$
satisfying the \emph{prequantization condition}
\begin{equation}\label{preq}
\omega=\frac{i}{2\pi} R^L\,.
\end{equation}
The Hamiltonian $G$-manifold $(M,\omega,\J)$ is \emph{prequantized} if the
action of $G$ lifts to an action on
$(L,h^L,\nabla^L)$ and the moment map
$\J:M \rightarrow \g^\ast$ is given in terms of the lifted action by the 
\emph{Kostant formula} \eqref{Kostant}. 
Upon choosing a $G$-invariant
almost complex structure $J\in\End(TM)$ over $M$
compatible with $\omega$, one gets a Riemannian metric on $M$ defined by
\begin{equation}
\label{gTX}
g^{TM}:=\omega(\cdot,J\cdot)\,.
\end{equation}
On can then consider the associated
\emph{Spin$^c$-Dirac operators} $D^\pm:\Omega^{0,\pm}(M,L)\to\Omega^{0,\mp}(M,L)$
defined in Section \ref{sec:2.3}, which are elliptic first order differential
operators, and thus have finite dimensional kernels. 
The action of $G$ on $(L,h^L)$ induces an action on these kernels,
and one defines the associated
\emph{invariant Riemann-Roch number} as
\begin{equation}\label{RRGfla}
\text{RR}^G(M,L):=\dim(\Ker D^+)^G-\dim(\Ker D^-)^G\,,
\end{equation}
where $(\Ker D^\pm)^G\subset \Ker D^\pm$ denotes the subspace of $G$-invariant vectors. 
By the classical invariance of the index of Fredholm operators, the invariant
Riemann-Roch number \eqref{RRGfla} does not depend on the choice of the
compatible almost complex structure $J\in\End(TM)$ over $(M,\omega)$,
nor on the choice of $h^L$ and $\nabla^L$ satisfying the
prequantization condition \eqref{preq}. 
In case  that $G$ is trivial, it reduces to the classical Riemann-Roch number $\text{RR}(M,L)$, which
is computed by the celebrated \emph{Hirzebruch-Riemann-Roch formula}
\begin{equation}\label{HRRfla}
\text{RR}(M,L)=\int_M{ e^{\omega}}\,\Td(M)\,,
\end{equation}
where the closed form $\Td(M)\in\Omega^*(M,\C)$ of mixed degrees
is the \emph{Todd form} of $(M,J,g^{TM})$, whose cohomology class
does not depend on $J$ and hence is a symplectic invariant of $(M,\omega)$.

In case of a general $G$, if $0$ does not belong to the image of the moment map, then $\M_0=\emptyset$ and $\text{RR}^G(M,L)=0$. If $0$ is a regular value of the moment map $\J$, the symplectic quotient $(\M_0,\omega_0)$ is an orbifold
prequantized by a line bundle $(L_0,h^{L_0})$ and by a general result of Meinrenken \cite{Mei98} one has
\begin{equation}\label{QR=0}
\text{RR}^G(M,L)=\int_{\M_0} e^{\omega_{0}}\Td(\M_0)=\text{RR}(\M_0,L_0)\,.
\end{equation}
In the special case $G=S^1$,
this was previously established independently
by Meinrenken in \cite[Theorem 2.1,\,(16)]{meinrenken96} and by Vergne in
\cite{Ver96}. As explained in \cite{Ver96}, these results
rely on localization formulas for
certain oscillatory integrals
of equivariant differential forms which were first
studied by Witten in \cite{witten92}, based on previous work of Duistermaat and
Heckman in \cite{duistermaat-heckman83}. To obtain \eqref{QR=0} within this approach, one expresses $\text{RR}^G(M,L)$ by means of the equivariant Hirzebruch-Riemann-Roch formula in a guise due to Berline and Vergne \cite{berline-vergne82}, called the \emph{Kirillov formula}. This  
 formula involves  the \emph{equivariant Todd form} $\Td_\g(M)\in\Omega^*_G(M)$ defined in Section \ref{sec:RRflas}, whose cohomology class in the equivariant
cohomology
with analytic coefficients $H^*_G(M)$ is an invariant of the Hamiltonian
$G$-action on $(M,\omega)$. The resulting expression for $\text{RR}^G(M,L)$
is given in terms of a Witten integral,
which can then be treated using the stationary phase
principle.
In doing so, one passes from the equivariant cohomology $H^*_G(M)$
 to the usual cohomology $H^*(\M_0)$ of the symplectic quotient
through the well-known \emph{Kirwan map} $\kappa:H^*_G(M)\to H^*(\M_0)$
described in Proposition \ref{Kirwanth}.
Since  in the case $G=S^1$ one has  $\kappa(\Td_\g(M))=\Td(\M_0)$, this gives
the first equality in \eqref{QR=0}, while the second follows from \eqref{HRRfla}.

In this paper, we are mainly interested in the considerably more involved case when  $0$ is a singular value of the moment map and restrict ourselves to the case $G=S^1$, which already encompasses essential features of the singular case. Denoting by $M^{S^1}\subset M$ 
the set of fixed points of the $S^1$-action  
and setting $\J^{-1}(\mklm{\0})^\text{reg}:=\J^{-1}(\mklm{\0})\setminus \big ( M^{S^1} \cap\J^{-1}(\{0\})\big)$ 
we  have a stratification of $\J^{-1}(\{0\})$ according to
\begin{equation}\label{stratiffla}
\J^{-1}(\{0\})=  \big(M^{S^1}\cap\J^{-1}(\{0\})\big) \sqcup \J^{-1}(\{0\})^\text{reg}\,,
\end{equation}
where the connected components of  each stratum are  smooth submanifolds of $M$.
To avoid any artificial complications, we will assume for simplicity
that $S^1$ acts freely on $\J^{-1}(\mklm{\0})^\text{reg}$,
so that the \emph{regular stratum} of the symplectic quotient, defined by 
\begin{equation*}\label{topstratdef}
\M_0^\text{reg}:=\J^{-1}(\mklm{\0})^\text{reg}/S^1\,,
\end{equation*}
has no orbifold singularities. We still write 
$\omega_0$ for the symplectic form over the smooth stratum $\M_0^\text{reg}$
characterized by the formula $\text{inc}^*_0\,\omega=\pi_0^*\,\omega_0$,
where $\text{inc}_0:\J^{-1}(\mklm{\0})^\text{reg}\hookrightarrow M$ is the inclusion
map
and $\pi_0:\J^{-1}(\mklm{\0})^\text{reg}\to\M_0^\text{reg}$ the quotient map. Similarly, the \emph{singular stratum} $M^{S^1}\cap\J^{-1}(\{0\})$ of the symplectic quotient carries a symplectic structure by restriction of $\omega$. 

By a classical result of Kirwan \cite{Kir84b}, we know that the fibers
of the moment map are connected, so that if $0$
is a local extremal value of $\J:M\to\R$, one infers with the local normal form theorem of Proposition \ref{prop:localnormform} that 
$\J^{-1}(\mklm{\0})=M^{S^1}\cap\J^{-1}(\mklm{\0})$,
which is thus a symplectic submanifold of $(M,\omega)$.
Since symplectic submanifolds are of codimension $2$,  the mean value theorem implies that 
$0$ lies on the boundary $\partial\,\J(M)$ of the image of
$\J$ and is consequently a global extremal value of $\J$. In this case the quotient \eqref{M0} satisfies
$\M_0=\J^{-1}(\mklm{\0})$ and inherits
a natural prequantizing line bundle $(L_0,h^{L_0})$ by restriction. The invariant
Hirzebruch-Riemann-Roch formula \eqref{QR=0} was then
established by Duistermaat, Guillemin,
Meinrenken and Wu in \cite[\S\,2]{DGMW95}.
Therefore, the main case of interest in this paper is when
$0$ lies in the interior $\text{Int}\,\J(M)$ of the image of $\J$, in which case $0$ is not a local
extremal value of $\J$.

Following Notation \ref{def:F}, let us write $\F^0$ for the set of connected
components of $M^{S^1}\cap\J^{-1}(\{0\})$.
For any $F\in\mathcal{F}^0$,
let $\omega_F:=\mathrm{inc}_F^\ast\,\omega\in\Omega^2(F,\R)$ be the symplectic form on $F\subset M$
induced by the inclusion $\mathrm{inc}_F:F \hookrightarrow M$, and let $\nu_{\Sigma_F}:\Sigma_F\to F$  be the associated symplectic normal bundle.
We write
\begin{equation*}
\Sigma_F=:\bigoplus_{k\in W}\Sigma_F^{(k)}
\end{equation*}
for its decomposition
into isotypic components with respect to the induced linear $S^1$-action,
where $W\subset\Z$
denotes the finite subset of weights as described in Section \ref{sec:stratif}.
The compatible almost complex structure $J\in\End(TM)$ over $(M,\omega)$
and the associated
Riemannian metric $g^{TM}$ given by \eqref{gTX} induce for each
weight $k\in W\subset\Z$ 
a complex structure and Hermitian norm $\|\cdot\|_F$ on $\Sigma_F^{(k)}$
by restriction,
and we write $R^{\Sigma_F^{(k)}}\in\Omega^2(F,\End(\Sigma_F^{(k)}))$
for the curvature of the connection on $\Sigma_F^{(k)}$ induced
by the Levi-Civita connection of $g^{TM}$ for each $k\in W$.

Let us now consider the case $0\in\Int\J(M)$, and 
 for any $F\in\F^0$, consider the fiberwise product 
\begin{equation}\label{SFintro}
\nu_{S_F}:S_F:=S_F^+\times_F S_F^-\longrightarrow F\,,
\end{equation}
where $S_F^\pm\to F$ are the unit sphere bundles
of the subbundles $\Sigma_F^\pm\subset\Sigma_F$
of positive and negative weights, respectively.
By the local normal form theorem of Proposition \ref{prop:localnormform}, the total space of $S_F$ is
naturally identified with
the boundary of a neighborhood of $F$ inside $\J^{-1}(\mklm{\0})$. 
We write 
\begin{equation}\label{S1SFS1}
\nu_{\mathfrak{S}_F}: \mathfrak{S}_F:=(S_F^+/S^1)\times_F(S_F^-/S^1) \longrightarrow F
\end{equation}
for the orbifold bundle obtained by taking the fiberwise product of the quotient of each sphere by the
induced locally free $S^1$-action.
Its de Rham cohomology ring will be denoted by $H^*(\mathfrak{S}_F)$.  

Our main result is the following Hirzebruch-Riemann-Roch type formula for  the invariant
Riemann-Roch number \eqref{RRGfla}. 

\begin{thm}\label{mainth}
Let $(M,\omega,\J)$ be a compact connected prequantized Hamiltonian $S^1$-manifold
such that $0 \in \J(M)$ and  the $S^1$-action is free on
$\J^{-1}(\mklm{\0})^\text{reg}$.
Then, the $S^1$-invariant Riemann-Roch number \eqref{RRGfla} is given by
\begin{multline}\label{mainthfla}
\text{\emph{RR}}^{S^1}(M,L)=\int_{\M_0^{\mathrm{reg}}} e^{\omega_{0}}\kappa(\Td_\g(M))
+\sum_{F\in\mathcal{F}^0_\text{\emph{exc}}}
\int_{\mathfrak{S}_F}e^{\nu_{\mathfrak{S}_F}^\ast \omega_{F}}
\kappa_{F}(\Td_{\g}(M))\\
+\sum_{F\in\mathcal{F}^0}{\textup{Res}}\,
\bigg(z^{-1}\int_{F}\frac{e^{\omega_F}\Td(F)}
{\prod_{k\in W}\det_{\Sigma_F^{(k)}}(1-z^k
\exp(R^{\Sigma_F^{(k)}}/2\pi i))}\bigg)\,,
\end{multline}
where  $\kappa:\Omega_{S^1}^*(M)\longrightarrow
\Omega^*(\M_0^{\mathrm{reg}},\C)$  denotes
the \emph{regular Kirwan map} of Definition \ref{Kirwantildedef} and
$\kappa_F:H_{S^1}^*(M)\to H^*(\mathfrak{S}_F)$ the \emph{exceptional Kirwan map}
\eqref{excKirdef}, 
and 
\begin{equation}\label{F0exc}
\F^0_\text{\emph{exc}}:=\begin{cases}\F_0  &\text{if} \quad 0 \in \Int\J(M), \\ \emptyset  &\text{if} \quad
0 \in\partial\,\J(M), \end{cases}
\end{equation}
while 
$\textup{Res}$ stands for the residue at $z=0$ if $0$ is
a minimal value of $\J$,
the residue at $z=\infty$ if $0$ is a maximal value of $\J$, 
or the average of the two residues otherwise.

Furthermore, every term on 
the right-hand side of
Formula \eqref{mainthfla} is independent of the choice of a
compatible almost complex
structure $J\in\End(TM)$ over $(M,\omega)$ and of
the choice of $(L,h^L,\nabla^L)$ satisfying the prequantization condition
\eqref{preq}.
\end{thm}
The first term of formula \eqref{mainthfla} is called the \emph{regular term}, and
the convergence of this integral over the non-compact manifold $\M_0^{\mathrm{reg}}$
will be established in Lemma \ref{lem:existenceintegral}.  
As we explain in more detail below, Theorem \ref{mainth} represents an instance of the
Guillemin-Sternberg principle in the singular case and is characterized by the novel 
conceptual feature that it expresses the invariant Riemann-Roch number purely in
terms of the symplectic invariants of the singular symplectic quotient $\M_0$. In particular, replacing $L$ by the tensor power $L^m:=L^{\otimes m}$ for any $m\in\N$, so that the symplectic form $\omega$ is replaced by $m\omega$,
one sees that each term of the right-hand side of
Formula \eqref{mainthfla} for $RR^{S^1}(M,L^m)$ is
polynomial in $m\in\N$, and Theorem \ref{mainth} thus gives explicit
formulas to compute the coefficients of $RR^{S^1}(M,L^m)$ as a
polynomial in $m\in\N$ in terms of symplectic invariants
of $\M_0$.
Let us also point out that Theorem \ref{mainth} is already
relevant in the complex case, giving an explicit formula for
the coordinate ring of singular projective varieties $\M_0$
obtained as GIT quotients by a $\C^*$-action of a smooth
projective variety $M$.

In case that $0$ is a regular value of the moment map $\J$ one has 
 $M^{S^1}\cap\J^{-1}(\mklm{\0})=\emptyset$, the last two terms of Formula \eqref{mainthfla} vanish and
Theorem \ref{mainth} reduces to the invariant Riemann-Roch formula
\eqref{QR=0}, as already explained there. At the other extreme end, in the case when $0\in\partial\,\J$, we have $\J^{-1}(\mklm{\0})=M^{S^1}\cap\J^{-1}(\mklm{\0})$ and the first two terms of
Formula \eqref{mainthfla} vanish. Theorem \ref{mainth} then yields again \eqref{QR=0} by taking into account
that all weights $k\in W$ of the $S^1$-action on
$\Sigma_F$ are of the same sign, its proof reducing to the original proof of Duistermaat, Guillemin,
Meinrenken and Wu in \cite[\S\,2]{DGMW95}.
Theorem \ref{mainth} can thus be seen
as an interpolation between these two extreme cases which covers also the genuinely
singular case. In particular, the middle term in Formula \eqref{mainthfla} appears to
be a previously unknown object involving a new explicit symplectic invariant of the
singularities.

To describe this term more closely, let $\Theta^\pm_{S_F}\in\Omega^1(S_F,\R)$ be the pullbacks
to \eqref{SFintro} of connections
for the $S^1$-actions on $S_F^\pm$ in the sense of \eqref{connfla} for any $F\in\F^0$. 
At the level of $S^1$-equivariant differential forms
as in Definition \eqref{OmS1} and using
the same notation as for the Kirwan
map \eqref{Kirwandef}, the image of an
equivariantly closed form $\varrho\in\Omega^*_{S^1}(M)$
by the \emph{exceptional Kirwan map} is the element in $\Omega^*(\mathfrak{S}_F,\C)$ defined by
\begin{equation}\label{excKirdef}
\kappa_F(\varrho):=\frac{\frac{1}{2}\big(\varrho_{S_F}(\frac{i}{2\pi}d \Theta_{S_F }^+)
    +\varrho_{S_F}(\frac{i}{2\pi}d \Theta_{S_F }^-)\big)-\varrho_{S_F}\big(\frac{i}{4\pi}(d \Theta_{S_F }^++d \Theta_{S_F }^-)\big)}{d\Theta^+_{S_F }-d\Theta^-_{S_F }}\,,
\end{equation}
where we wrote 
$\varrho_{S_F}:=\nu_{S_F}^\ast\mathrm{inc}_F^\ast \varrho\in \Omega^*_{S^1}(S_F)$ for the pullback to $S_F$ of the restriction of
$\varrho$ to $F$.  The numerator on the right-hand side of \eqref{excKirdef} is
a multiple of $d \Theta_{S_F}^+-d \Theta_{S_F }^-$ inside the ring
$\Omega^*(S_F,\C)$, which gives an obvious sense to the fraction.
Hence, as it is also closed and
basic for the action of $S^1$
on both sphere bundles $S_F^\pm$, by Proposition \ref{pi*iso},
Formula \eqref{excKirdef} induces a well-defined map 
$\kappa_F:H_{S^1}^*(M)\to H^*(\mathfrak{S}_F)$
in cohomology. In fact, the numerator in \eqref{excKirdef}
is even a multiple of $(d \Theta_{S_F}^+-d \Theta_{S_F }^-)^2$,
so that the right-hand side of \eqref{excKirdef} is itself a multiple of
$d \Theta_{S_F}^+-d \Theta_{S_F }^-$. This implies
in particular that $\kappa_F$, and hence the middle term in Formula 
\eqref{mainthfla}, vanishes when $\dim M<6$. 

The invariant Riemann-Roch formula \eqref{QR=0} is the content of
the celebrated \emph{Quantization commutes with Reduction} principle
of Guillemin and Sternberg, which they formulated in
\cite{guillemin-sternberg84} in the case of K\"ahler manifolds, always under the assumption that  $0$ a  regular value of the moment map. 
In that case, the kernel of the Spin$^c$-Dirac operator $D^\pm$ reduces to the kernel of the Dolbeault
$\overline\partial$-operator acting on $\Omega^{0,\pm}(M,L)$.
In particular, the kernel of the $\overline\partial$-operator restricted to
$\Omega^{0,0}(M,L)$ coincides with the space $H^0(M,L)$
of holomorphic sections of $L$
over $M$, which is interpreted in \cite{guillemin-sternberg84} as the
quantization of the classical phase space $(M,\omega)$.  Furthermore, $0$ being a regular value, the symplectic reduction
$(\M_0,\omega_0)$ inherits a natural structure of a K\"ahler orbifold, and
it was shown in \cite{guillemin-sternberg84} that there is
a natural isomorphism
\begin{equation}\label{GSQR=0}
H^0(M,L)^G\simeq H^0(\M_0,L_0)\,.
\end{equation}
The identity \eqref{GSQR=0} was generalized to the kernel of the
$\overline\partial$-operator restricted to
$\Omega^{0,j}(M,L)$ for each $j>0$ by Teleman \cite{Tel00}
and Zhang \cite{Zha99}. Taking the alternating sum of dimensions of these spaces,
this precisely leads to the invariant Riemann-Roch formula \eqref{QR=0} in the
context of K\"ahler manifolds. As we explain in Section \ref{sec:2.3},
Formula \eqref{QR=0} naturally extends to general symplectic manifolds,
so that it constitutes
the appropriate extension of Quantization commutes with Reduction in the
general symplectic case,
since the isomorphism \eqref{GSQR=0} does not make sense
in general.  In case that $L$ is replaced by the tensor power $L^m:=L^{\otimes m}$ for $m\in\N$ large enough, Formula \eqref{QR=0} was  established by Meinrenken in \cite{meinrenken96} by the approach already described above, then by Jeffrey-Kirwan in \cite{jeffrey-kirwan97}
relying on the so-called \emph{Witten non-abelian localization formula} 
stated in \cite{witten92},
which was
established rigorously by Jeffrey and Kirwan in
\cite{jeffrey-kirwan95}. In general, Formula \eqref{QR=0} was first established by 
Meinrenken in \cite{Mei98} using the symplectic cutting techniques of Lerman
\cite{Ler95},  
then by Tian and Zhang \cite{TZ98}
using the analytic localisation techniques of Bismut-Lebeau \cite{BL91}.
Formula \eqref{QR=0} was  generalized to the case of manifolds with
boundary by Tian and Zhang
in \cite{TZ99}, then to the case of non-compact $(M,\omega)$ 
by Ma and Zhang in \cite{MZ14} and Paradan in \cite{Par03}. A
generalization to compact CR-manifolds was recently established by Ma, Marinescu, 
and Hsiao in \cite{MMC19}.

In case that $0$ is a singular value of the moment
map, there is an immediate difficulty coming from the fact that there is no natural
definition of the Riemann-Roch number $\text{RR}(\M_0,L_0)\in\Z$ of a
singular symplectic quotient $\M_0$, and one is tempted to define it directly as
$\text{RR}(\M_0,L_0):=\text{RR}^G(M,L)$, making the result tautological.
More substantially, one can consider the Riemann-Roch number of
various notions of symplectic desingularization of stratified symplectic spaces,
such as Kirwan's \emph{partial desingularization} or the
\emph{shift desingularization}. The
invariant Riemann-Roch number $\text{RR}^G(M,L)$ was shown to coincide
with the Riemann-Roch number of Kirwan's partial desingularization
by Meinrenken and Sjamaar in \cite{meinrenken-sjamaar}, and
of the shift desingularization by Meinrenken and Sjamaar in
\cite{meinrenken-sjamaar}, Zhang in \cite{Zha99}
and Paradan in \cite{Par01}. In the case $G=S^1$,
Tian and Zhang established in \cite[Theorem 6.4]{TZ99} an identity in terms of
a Riemann-Roch number for shift desingularizations containing
a term similar to the last term in Formula \eqref{mainthfla}.

A main drawback of the approaches described above is that
the desingularization depends on a number of choices and is in no way unique,
while the symplectic structure also
depends on the choice of an auxiliary parameter $\epsilon>0$.
In particular, these results do not allow to compute 
$\text{RR}^G(M,L)$ in terms of the symplectic
invariants of the symplectic quotient $\M_0$ itself and do not provide a canonical 
Hirzebruch-Riemann-Roch type formula such as \eqref{QR=0}.
In contrast, Theorem \ref{mainth} does not rely on any
desingularization process and provides
an explicit formula for $\text{RR}^G(M,L)$  in which every term  is a well-defined
invariant of the Hamiltonian
action of $S^1$ on $(M,\omega)$. Thus, our formula satisfies the key requirement of
the Guillemin-Sternberg principle that the involved reduced  system should only depend
on the symplectic data of $M$, answering an almost 30-year-old  question of Sjamaar
in \cite[pp.\ 124-126]{Sja95}.

In order to compare Theorem \ref{mainth} more closely with
the mentioned previous results, we provide in Section \ref{Kirwanressec}
natural topological conditions on the $S^1$-action under which
the first term in \eqref{mainthfla} can be interpreted topologically as
\begin{equation}\label{eq:09.04.2023}
\int_{\M_0^{\mathrm{reg}}}e^{\omega_{0}}\kappa(\Td_\g(M))=
\int_{\widetilde \M_0}
e^{\widetilde \omega_0}\widetilde \kappa(\Td_\g(M))\,,
\end{equation} 
where $\pi:\widetilde \M_0\to\M_0$ denotes the partial resolution of $\M_0$
and $\widetilde \omega_0$ denotes the degenerate $2$-form over
$\widetilde \M_0$ coinciding with $\omega_0$ on the dense open set
$\M_0^{\text{reg}}$, while 
$\widetilde\kappa:H^*_{S^1}(M)\longrightarrow H^*(\widetilde\M_0)$
is the usual Kirwan map of the resolution. By contrast, the Riemann-Roch formula
obtained via the method of Meinrenken and Sjamaar in \cite{meinrenken-sjamaar}
reads
\begin{equation}\label{QR=eps}
\text{RR}^G(M,L)=\int_{\widetilde{\M}_0}
e^{\widetilde\omega_{\epsilon}}\Td_\epsilon(\widetilde\M_0)\,,
\end{equation}
where the symplectic form
$\widetilde\omega_{\epsilon}\in\Omega^2(\widetilde\M_0,\R)$
depends on the choice of a parameter $\epsilon>0$ such
that $\widetilde\omega_{\epsilon}\xrightarrow{\epsilon\to 0}\widetilde\omega_0$,
while
$\Td_\epsilon(\widetilde\M_0)\in\Omega^\ast(\widetilde\M_0,\R)$
denotes the induced Todd form.
In particular, as $\widetilde\omega_0$ is
degenerate, the Todd form $\Td_\epsilon(\widetilde\M_0)$ does not admit
a limit as $\epsilon\to 0$, so that the right-hand side is not well defined
a priori for $\epsilon=0$. 
Let us also point out that Jeffrey, Kiem, Kirwan, and Woolf in 
\cite{JKKW03}, and Lerman and Tolman in \cite{lerman-tolman00} in the case of $G=S^1$, 
studied Kirwan maps to the intersection cohomology of the singular
quotient. The relation of their Kirwan maps to our resolution  Kirwan map is explained in Remark \ref{rem:27.04.2023}.

Our main tool for the proof of Theorem \ref{mainth} is the already mentioned
\emph{Witten integral}, which we introduce in Definition \ref{def:mainint}.
In fact,  Theorem \ref{mainth} is established as a consequence of our
second main result,
Theorem \ref{thm:wittenasymp}, where we compute an asymptotic expansion of the Witten
integral in terms of geometric invariants associated with the singular
symplectic quotient. This allows us to  directly compute the
invariant Riemann-Roch number $\text{RR}^{S^1}(M,L)$ by adapting the
method of Meinrenken in \cite{meinrenken96} to the singular case.
The asymptotic parameter in the Witten integral is  given by the power $m\to \infty $ of 
the tensor power $L^m:=L^{\otimes m}$ of the prequantizing line bundle.
It can therefore be regarded as a semiclassical limit, so that from this viewpoint Theorem \ref{mainth} becomes an instance of the correspondence principle of quantum mechanics.
Asymptotics of the Witten integral for general Hamiltonian $G$-manifolds
and when $0$ is a singular value of the moment map were first obtained
by Paradan in \cite[Theorem 5.1]{paradan00} relying on partitions of unity in equivariant cohomology with generalized coefficients. There, the  coefficients in the expansion are given by integrals over the shifted
Marsden-Weinstein reduced spaces $\M_\eps= \J^{-1}(\mklm{\epsilon})/G$ for some $\eps \in \g^\ast$ close to but different from $0$.   In certain algebraic settings, an asymptotic expansion of the Witten integral was also derived
by Jeffrey, Kiem, Kirwan, and Woolf
in \cite[Section 9]{JKKW03} using the localization principle. Accordingly, the coefficients in the asymptotics are given in terms of residues. In contrast, the coefficients in our asymptotic expansion of the Witten integral are given in terms of integrals on the symplectic strata of $\M_0$, which is crucial in proving  Theorem \ref{mainth}. Our approach was preceded  by work of Delarue and Ramacher in \cite{kuester-ramacher20} within a purely analytic context, motivated by the original attempt of Ramacher \cite{ramacher13} of proving residue formulae via singular equivariant asymptotics. 

Let us finally note that the assumption of $S^1$
acting freely on $\J^{-1}(\mklm{\0})^\text{reg}$, meaning that there are
no orbifold singularities over $\M_0^{\textup{reg}}$, is not essential.
In fact, the result of Meinrenken in
\cite{meinrenken96}, which we extend in this paper to the case when $0$ is a singular value
of the moment map, holds for orbifolds as well, and our contribution mainly focuses
on the most singular stratum $M^{S^1}\cap\J^{-1}(\{0\})$. Following the general method
of Meinrenken, one can certainly extend our method to orbifolds, obtaining a
Kawasaki-Riemann-Roch type formula which extends \eqref{mainthfla}. In a future work,
we intend to generalize our results to general compact group actions.

We begin our exposition  in Section \ref{sec:2} by giving a detailed account on the background and setup of our paper. In Section \ref{sec:sympquot} we study the geometry of the zero level set of the moment map around its singularities, which will be crucial for the ensuing analysis, and introduce the relevant Kirwan maps in Section \ref{sec:Kirwan}.  Based on these results, we derive in Section \ref{sec:I} a complete asymptotic expansion of the Witten integral. Finally, we establish in Section \ref{sec:6}
the proof of our main result Theorem \ref{mainth}.\newline

{\bf Acknowledgments.} We would like to thank Mich\`{e}le Vergne for her continuous interest in our work and her encouragement, and Paul-Emile Paradan for stimulating discussions.  The first author has received funding from the Deutsche Forschungsgemeinschaft (DFG, German Research Foundation) – Project-ID 491392403 – TRR 358 and through the Priority Program (SPP) 2026 ``Geometry at Infinity'', while the second author was partially supported by the Alexander von Humboldt Stiftung (Alexander von Humboldt Foundation) and by ANR-23-CE40-0021-01 JCJC project QCM. Finally, we thank an anonymous referee for suggesting several improvements.

\section{Background and setup}\label{sec:prodcohoms2}
\label{sec:2}

Before we introduce our setup, let us first fix some global notation. For any vector bundle $E$ over a smooth manifold $M$ and for all $k\in\N$,
we will write $\Omega^k(M,E)$ for the space of differential forms of degree
$k$ with values in $E$, and
$\Omega^*(M,E):=\bigoplus_{k=0}^{\textup{dim} M}\Omega^k(M,E)$ for the
space of differential forms of mixed degrees with values in $E$.
We will use the same notation for any vector space $V$,
which we regard as a trivial vector bundle over $M$. We will denote the inclusion $A\hookrightarrow B$ of a subset $A\subset B$ into a set $B$  by $\mathrm{inc}_A$, the target set $B$ being clear from the context. For a finite Cartesian product $X=X_1\times\cdots X_N$, we will write $\mathrm{pr}_{X_j}:X\to X_j$ for the canonical projection onto the factor $X_j$.

\subsection{Equivariant cohomology}
\label{sec:eqcoh} 
Let $M$ be a smooth manifold equipped with a smooth action of
a compact Lie group $G$. 
Write $S^\omega(\g^*)$ for the complex vector
space of analytic series on $\g$ with complex coefficients
converging in a neighborhood of $0\in\g$.
The coadjoint action of
$G$ on $\g^*$  induces a natural action on $S^\omega(\g^*)$.

\begin{definition}\label{OmG}
The \emph{complex of analytic $G$-equivariant differential forms}
is the complex vector space of $G$-invariants
\begin{equation*}
\Omega^*_G(M):=\left(\Omega^*(M,\C)\otimes S^\omega(\g^*)\right)^G\,
\end{equation*}
for the natural $G$-action on the tensor product,
equipped
with the \emph{equivariant differential}
\begin{equation}
\begin{split}
d_\g:\Omega_G^*(M)&\longrightarrow\Omega_G^*(M),\\
\sigma(X)&\longmapsto d\sigma(X)+2\pi i\,\widetilde{X}\imult\sigma(X)\,.\label{eq:dg}
\end{split}
\end{equation}
The cohomology of the complex
$(\Omega^*_G(M),d_\g)$ is denoted by $H^*_G(M)$ and called
 \emph{equivariant cohomology} of $M$ with analytic coefficients.
\end{definition}

In the above definition and in the sequel, we will often write equivariant
differential forms
$\sigma(X)\in\Omega^*_G(M)$ with explicit dependence on the variable
$X\in\g$.
We point out that there are several competing conventions for the definition of the equivariant differential in the literature, with the factor $2\pi i$ in \eqref{eq:dg} replaced by other constants. Similarly, the sign in \eqref{eq:sign24} is a convention. We follow here Meinrenken's conventions in \cite{meinrenken96}.  The following version of the Stokes Lemma for equivariant differential forms
can be found for instance in \cite[\S 4]{bott-tu}.

\begin{lem}\label{eqStokes}
{\bf (Equivariant Stokes' lemma)}
Let $U\subset M$ be an open set such that its closure
$\overline U\subset M$ is a compact submanifold with boundary. Then, for any
$\sigma\in\Omega^*_G(M)$ we have
\begin{equation*}
\int_U d_{\g}\sigma(X)=\int_{\partial\overline U}\sigma(X)\,.
\end{equation*}
\end{lem}

For a submanifold $N\subset M$ we write
$\nu_{\Sigma_N}:\Sigma_N\to N$ for its normal bundle inside $TM$  
and  identify $N$ canonically with the zero section in $\Sigma_N$. The following version of the classical homotopy lemma for equivariant differential forms
follows for instance from \cite[Theorem  6.1]{goertscheszoller}.

\begin{lem}\label{eqHom}{\bf (Equivariant homotopy lemma)}  
Let $N\subset M$ be a submanifold preserved by the action of $G$, and let $\Phi_N:V_N\to U_N$
be a $G$-equivariant diffeomorphism between a tubular neighborhood
$V_N\subset\Sigma_N$ of the zero section of $\Sigma_N$ and a
tubular neighborhood $U_N\subset M$ of $N$. Then, for any equivariant cohomology class
$[\varrho] \in H^*_G(M)$  there exists an equivariant form $\beta_N \in  \Omega_{G}^\ast(V_N)$ such that
\begin{equation}\label{eq:eqHom}
\Phi_{N}^\ast(\varrho|_{U_N})
= \nu_{\Sigma_N}^\ast  \mathrm{inc}_N^*\varrho|_{V_N}+d_\g \beta_N.
\end{equation}
Furthermore, the equivariant form $\beta_N \in  \Omega_{G}^\ast(V_N)$ can be chosen
such that $\inc_N^*\beta_N\equiv 0$ and such that for any open set $S\subset V_N$ on which  we already have $\Phi_{N}^\ast(\varrho|_{U_N})|_{S}
= \nu_{\Sigma_N}^\ast \mathrm{inc}_N^*\varrho|_{S}$, one has $\beta_N|_{S}\equiv 0$.
\end{lem}

The second part of Lemma \ref{eqHom}, which will be used in
Section \ref{sec:I} to
establish Lemma \ref{assumption}, follows from the fact that $N$ is a strong deformation retract of $U_N$ and the explicit form
of the homotopy operator used for instance in \cite[Theorem  6.1]{goertscheszoller}
to construct the equivariant form $\beta_N \in  \Omega_{G}^\ast(V_N)$
of Formula \eqref{eq:eqHom}.

Consider now the important special case when the $G$-action is
locally free, so that the quotient map
\begin{equation}\label{quomap}
\pi:M\longrightarrow M/G
\end{equation}
is a $G$-principal bundle over the orbifold $M/G$.
We then have the following basic result.

\begin{prop}\label{pi*iso}
The pullback by the quotient map \eqref{quomap} induces an isomorphism
of complexes
\begin{equation}
\pi^*:(\Omega(M/G,\R),d)\xrightarrow{~\sim~}(\Omega(M)^G_{\mathrm{bas}},d)\,,\label{eq:basiso}
\end{equation}
where the subcomplex of \emph{basic differential forms}
$\Omega(M)^G_{\mathrm{bas}}\subset\Omega(M,\R)$ is defined by
\begin{equation}\label{basdef}
\Omega(M)^G_{\mathrm{bas}}:=\{\sigma\in\Omega^*(M,\R)^G~|~\widetilde{X}\imult\sigma=0\}\,,
\end{equation}
\end{prop}
The cohomology of the complex \eqref{basdef} is called the
\emph{basic cohomology} of $M$, and Proposition \ref{pi*iso}
shows that it is isomorphic to the de Rham cohomology $H^*(M/G)$
of the orbifold $M/G$. On the other hand, one has the
following fundamental notion in Chern-Weil theory,
which will be of crucial importance in this paper.

\begin{definition}\label{conndef}
A \emph{connection} for a locally free $G$-action on $M$ is a $\g$-valued
$1$-form $\Theta\in\Omega^1(M,\g)$ satisfying for all $X\in\g$
\begin{equation*}
\tilde X \imult \Theta=X\,.
\end{equation*}
\end{definition}

In what follows, we will mainly be interested in the case where $G=S^1$ is the circle group, which we realize as the subgroup of complex numbers of modulus $1$, inducing an identification
\begin{equation}\label{identfla}
\begin{split}
\g&\xrightarrow{\,\sim\,}\R\\
X&\longmapsto x\,,
\end{split}
\end{equation}
in such a way that $X\in\g$ exponentiates to $e^{2\pi i x}\in S^1\subset\C$.
This induces in turn an identification $\g^*\simeq\R$ of the dual of the Lie algebra with $\R$ and of the Lebesgue measure on the interval $[0,1]$ with the normalized Haar measure on $S^1$. Under this identification, Definition \ref{OmG} becomes
\begin{equation}\label{OmS1}
\Omega^*_{S^1}(M)=\Omega(M,S^\omega(\R))^{S^1}\,,
\end{equation}
the complex vector space of $S^1$-invariant differential forms
with values in entire analytic series of the variable $x\in\R$.
We will write $S^1$-equivariant differential forms
$\sigma(x)\in\Omega^*_{S^1}(M)$ with explicit dependence in the variable
$x\in\R$, when they are understood in the identification \eqref{OmS1},
so that they can actually be seen as functions of $x\in\R$ with values in
$S^1$-invariant differential forms. Furthermore, Definition \ref{conndef} of
a connection for
the $S^1$-action on $M$ becomes under this identification a $1$-form
$\Theta\in\Omega^1(M,\R)$ such that for any $X\in\g$ with image $x\in\R$
by \eqref{identfla}, we have
\begin{equation}\label{connfla}
\tilde X \imult \Theta=x\,.
\end{equation}
The following basic lemma is then a straightforward consequence of 
Proposition \ref{pi*iso} and Definition \ref{conndef} via the
identification \eqref{identfla}.
\begin{lem}\label{intXi=int}
Assume that the $S^1$-action on $M$ is locally free,
let $\Theta\in\Omega^1(M,\R)$ be an associated connection in the sense
of \eqref{connfla}, and
let $\sigma\in\Omega(M/S^1)$ be closed.
Then we have
\begin{equation}\label{intXi}
\int_M\pi^*\sigma\wedge\Theta=\int_{M/S^1}\sigma
\end{equation}
in terms of  the isomorphism
of complexes \eqref{eq:basiso}. In particular, the integral
on the left-hand side of \eqref{intXi}
only depends on the basic cohomology class of
$\sigma$.
\end{lem}

\subsection{Dirac operators and invariant Riemann-Roch numbers}
\label{sec:2.3}
To introduce our proper setup, let us now focus on the case of a Hamiltonian $G$-action on
a compact connected symplectic manifold $(M,\omega)$ of dimension $2n$. Recall that a $G$-equivariant map
 $\J: M \to \g^\ast$ is called a \emph{moment map}
for such an  action if for all $X\in\g$,
the function $\J(X)\in\Cinft(M,\R)$ satisfies
\bq
d\J(X) =-\widetilde X \imult \omega\,,\label{eq:sign24}
\eq
where $d$ is the de Rham differential and $\imult$ denotes the contraction. By definition, a Hamiltonian $G$-action on a symplectic manifold
$(M,\omega)$ induces a locally free action on the level
set $\J^{-1}(\{0\})$ if and only if 
$0$ is a regular value of $\J:M\to\g^*$.  For $G=S^1$, under the identification \eqref{identfla}  the moment map
$\J:M\to\g^*$ corresponds to a function  $\J\in\Cinft(M,\R)$ which we
still call moment map,  such that for any $X\in\g$ with image $x\in\R$
by \eqref{identfla}, we have
\begin{equation}\label{JX=xJ}
\J(X)=x\J\in\Cinft(M,\R)\,.
\end{equation}

Next, let $(M,\omega)$ be endowed with
a so-called \emph{prequantizing} line bundle $(L,h^L,\nabla^L)$, so that $(L,h^L)$ is a Hermitian line bundle over $M$
with a Hermitian connection $\nabla^L$ whose curvature
$R^L\in\Omega^2(M,\C)$
satisfies the prequantization condition \eqref{preq}.
We denote by $\Cinft(M,L)$ the space of smooth sections of $L$. Let further $G$ be a connected compact
Lie group such that $G$ acts on $L$ over $M$, preserving
the Hermitian metric $h^L$ and the connection $\nabla^L$.
Such an action is called \emph{prequantized}, and the induced action
of $G$ on $(M,\omega)$ then preserves the symplectic form $\omega$. Following for instance \cite[(25)]{meinrenken96} adapted to our conventions,
there is a canonical
choice of moment map $\J: M \to \g^\ast$ for a prequantized action of $G$
on $(M,\omega)$ defined by the \emph{Kostant formula} 
\begin{equation}\label{Kostant}
\J(X)s:=\frac{i}{2\pi}\left(L_Xs-\nabla^{L}_{\widetilde X}s\right)
\end{equation}
for all $s\in\Cinft(M,L)$ and $X\in\g$, where $\widetilde X\in\Cinft(M,TM)$ denotes the fundamental vector field on $M$ associated with $X$ and $L_X$ denotes the Lie
derivative with respect to $X$ induced by the $G$-action.
This is the moment map which will underlie all our considerations from now on.

Next, choose a $G$-invariant compatible almost complex structure $J\in\End(TM)$
over $(M,\omega)$,
inducing a splitting
\begin{equation}\label{splitJ}
TM\otimes\C=T^{(1,0)}M\oplus T^{(0,1)}M
\end{equation}
on the complexification $TM\otimes\C$ of $TM$ into the
eigenspaces of $J$ corresponding to the eigenvalues $i$ and $-i$, respectively.
Consider the total exterior product
\begin{equation}\label{spin}
\Lambda( T^{*(0,1)}M):=\bigoplus_{j=0}^n \Lambda^j( T^{*(0,1)}M)\,,
\end{equation}
where $T^{*(0,1)}M$ denotes the dual bundle of $T^{(0,1)}M$. 
For any $v\in TM$ with decomposition $v=v^{1,0}+v^{0,1}$ according to \eqref{splitJ}, we define its Clifford action on
$\alpha\in\Lambda( T^{*(0,1)}M)$ by
\begin{equation}\label{cliff}
c(v)\alpha:=\sqrt{2}\,(v^{1,0})^*\wedge\alpha-v^{0,1}\imult\alpha\,,
\end{equation}
where $(v^{1,0})^*$ denotes the metric dual of $v^{1,0}$ in $T^{*(0,1)}M$
with respect to the induced Hermitian metric $g^{TM}$ defined by
\eqref{gTX}. As explained for instance in \cite[Appendix D]{LH89},
the Clifford action \eqref{cliff} on $\Lambda( T^{*(0,1)}M)$ is associated
with the canonical \emph{$\Spin^c$ structure} of the almost Hermitian manifold
$(M,J,g^{TM})$, and there is an induced connection
$\nabla^{\Lambda( T^{*(0,1)}M)}$ on 
$\Lambda( T^{*(0,1)}M)$, which we call the \emph{Clifford connection}.
Following \cite[\S 1.3.1]{MM07}, the Clifford connection
$\nabla^{\Lambda( T^{*(0,1)}M)}$
is given in any local orthonormal frame
$\{e_j\}_{j=1}^n$ of $(TM,g^{TM})$
by the formula
\begin{equation*}
\nabla^{\Lambda( T^{*(0,1)}M)}
=d+\frac{1}{4}\sum_{j,k=1}^{2n}\<\Gamma^{TM} e_j,e_k\>_{g^{TM}}c(e_j)c(e_k)+\frac{1}{2}\Gamma^{\det}\,,
\end{equation*}
where $\Gamma^{TM}$ is the local connection form
of the Levi-Civita connection $\nabla^{TM}$ associated with $g^{TM}$
and $\Gamma^{\det}$ is the local connection form associated with
the connection on $\det(T^{(1,0)}M):=\Lambda^n( T^{(1,0)}M)$ induced
by $\nabla^{TM}$.  Finally, for any $m\in\N$, we write
$\Omega^{0,*}(M,L^m)$ for the space of smooth sections of
$\Lambda( T^{*(0,1)}M)\otimes L^m$, where $L^m$ denotes the $m$-th tensor power of $L$. We then have the following 

\begin{definition}
For any $m\in\N$, the \emph{Spin$^c$-Dirac operator} $D_m$
is defined in any local orthonormal frame $\{e_j\}_{j=1}^n$ of $(TM,g^{TM})$ by the formula
\begin{equation*}\label{spin^c}
D_m=\sum\limits_{j=1}^{2n} c(e_j)\nabla^{\Lambda( T^{*(0,1)}M)\otimes L^m}_{e_j}:\Omega^{0,*}(M,L^m)\longrightarrow\Omega^{0,*}(M,L^m)\,,
\end{equation*}
where $\nabla^{\Lambda( T^{*(0,1)}M)\otimes L^m}$ is the connection
on $\Lambda( T^{*(0,1)}M)\otimes L^m$
induced by $\nabla^{\Lambda( T^{*(0,1)}M)}$ and $\nabla^{L^m}$.
\end{definition}

By definition of the Clifford action \eqref{cliff}, the $\Spin^c$-Dirac
operator $D_m$ interchanges the space $\Omega^{0,+}(M,L^m)
\subset\Omega^{0,*}(M,L^m)$ of even-degree forms with the space
$\Omega^{0,-}(M,L^m)\subset\Omega^{0,*}(M,L^m)$
of odd-degree forms in the decomposition \eqref{spin}
of $\Lambda( T^{*(0,1)}M)$. We write $D_m^\pm$
 for the restriction of $D_m$ to $\Omega^{0,\pm}(M,L^m)$.
On the other hand, as shown for instance in \cite[Lemma 1.3.4]{MM07},
$D_m$ is a formally self-adjoint elliptic
operator on $\Omega^{0,*}(M,L^m)$ with respect to
the $\L^2$-Hermitian product induced by $g^{TM}$ and $h^L$.
In particular, it has finite-dimensional kernel inside
$\Omega^{0,*}(M,L^m)$, which allows us to state the following
definition.  

\begin{definition}\label{RRdef}
For any $m\in\N$, the \emph{Riemann-Roch number} $\text{RR}(M,L^m)\in\N$
of the bundle $L^m$ over $M$ is defined by the formula
\begin{equation*}
\text{RR}(M,L^m):=\dim\Ker D_m^+-\dim\Ker D_m^-\,.
\end{equation*}
\end{definition}
The standard invariance property of indices of elliptic operators
with respect to deformations shows that these Riemann-Roch numbers
do not depend on the choice of an almost complex
structure $J$ nor on $h^L$ and $\nabla^L$ satisfying the prequantization
condition \eqref{preq}.

On the other hand, recall  that the action of $G$ preserves all additional
data on $M$ and $L$ by construction, so that there is an
induced action of $G$ on
$\Omega^{0,*}(M,L^m)$ commuting with $D_m$. In particular,
this induces unitary representations of $G$ on the finite dimensional
Hermitian vector spaces $\Ker D_m^+$ and $\Ker D_m^-$. For any $g\in G$,
we write
\begin{equation*}\label{chi}
\chi^{(m)}(g):=\Tr_{\Ker D_m^+}[g]-\Tr_{\Ker D_m^-}[g]\,.
\end{equation*}
We write
$(\Ker D_m^+)^G$ and $(\Ker D_m^-)^G$ for the subspaces of $G$-invariant
vectors inside $\Ker D_m^+$ and $\Ker D_m^-$ respectively.

\begin{definition}\label{RRGeqdef}
For any $m\in\N$, the \emph{$G$-invariant Riemann-Roch number}
$\text{RR}^G(M,L^m)\in\N$
of the bundle $L^m$ over $M$ is defined by the formula
\begin{equation*}
\text{RR}^G(M,L^m):=\dim(\Ker D_m^+)^G-\dim(\Ker D_m^-)^G\,.
\end{equation*}
\end{definition}
Again by invariance of the index of elliptic operators, the
$G$-invariant Riemann-Roch number does not depend on the choice of a $G$-invariant almost complex
structure $J$ nor on the choice of $G$-invariant $h^L$ and $\nabla^L$
satisfying the prequantization condition \eqref{preq}.
Note also that we have
\begin{equation}\label{RRintG}
\begin{split}
\text{RR}^G(M,L^m)&=\int_G\Tr_{\Ker D_m^+}[g]\,dg-\int_G\Tr_{\Ker D_m^-}[g]\,dg=\int_G\,\chi^{(m)}(g)\,dg\,,
\end{split}
\end{equation}
where $dg$ is the normalized Haar volume form on $G$.

\subsection{Equivariant characteristic forms and Riemann-Roch formulas}
\label{sec:RRflas} 
In the setting of Section \ref{sec:2.3}, and following \cite[Definition 7.5]{berline-getzler-vergne},
we define for any Hermitian vector bundle $(E,h^E)\to M$ to which the $G$-action lifts and which is 
equipped with a $G$-equivariant Hermitian connection $\nabla^E$ the associated
\emph{moment map} $\J^{E}:M\to \End(E)\otimes\g^*$ via the Kostant formula
\begin{equation*}\label{KostantRiem}
\J^{E}(X):=L_X-\nabla^{E}_{\widetilde X}\,, \qquad X \in \g,
\end{equation*}
where $L_X$ denotes the Lie
derivative with respect to $X$ induced by the action of $G$ on $E$. Its \emph{equivariant curvature}
$R^{E}_\g(X)\in\Omega^*(M,\End(E))$ evaluated at $X\in\g$ is then given 
by the formula
\begin{equation*}\label{RTMeqdef}
R^{E}_\g(X):=R^{E}+2\pi i\J^{E}(X)\,,
\end{equation*}
where $R^{E}$ denotes the curvature of $\nabla^E$.
The associated \emph{equivariant Todd form} is the equivariantly closed form
$\Td_\g(E)\in\Omega^*_G(M)$ defined for all $X\in\g$ by the formula
\begin{equation}\label{Toddeq}
\Td_\g(E,X):=\det_{E}\left(\frac{R^{E}_\g(X)/2\pi i}
{\exp(R^{E}_\g(X)/2\pi i)-{\mathrm{Id}_{E}}}\right)\,.
\end{equation}
One readily checks that its cohomology class in $H^*_G(M)$ does not depend
on the choice of a Hermitian metric $h^E$ and connection $\nabla^E$.
In the important case of the tangent bundle $TM\to M$ equipped with
the chosen $G$-invariant
compatible almost complex structure $J\in\End(TM)$ over
$(M,\omega)$,
one can consider the induced Hermitian metric $g^{TM}$ given by Formula
\eqref{gTX} and the associated Levi-Civita connection $\nabla^{TM}$,
which are both $G$-equivariant.
This induces a Hermitian metric and connection
on $T^{(1,0)}M$ via the splitting \eqref{splitJ}, and the associated equivariant 
Todd form $\Td_\g(M)\in\Omega^*_G(M)$
given by $\Td_\g(M,X):=\Td_\g(T^{(1,0)}M,X)$
does not depend on the choice of the 
compatible almost
complex structure, and hence is an invariant of
the Hamiltonian action of $G$ on $(M,\omega)$.
The closed form $\Td(M):=\Td_\g(M,0)\in\Omega^*(M)$ is called
the \emph{Todd form} of $M$, and its
cohomology class is a symplectic invariant of $(M,\omega)$.

Recalling the prequantization condition \eqref{preq}, the \emph{equivariant Chern character} of $L^m$ 
is the equivariantly closed form $\ch_\g(L^m)\in\Omega^*_G(M)$
 defined for all $m\in\N$ and $X\in\g$ by the formula
\begin{equation}\label{ChernLm}
\ch_\g(L^m,X):=
\exp(m\omega +2\pi im\J(X))\,.
\end{equation}
The closed form $\ch(L^m):=\ch_\g(L^m,0)=e^{m\omega}\in\Omega^*(M)$ is called
the \emph{Chern character} of $L^m$.

Consider now for any $g\in G$ the fixed point set
\begin{equation}\label{Mg}
M^g:=\mklm{p \in M~|~g\cdot p =p}\subset M,
\end{equation}
which is a smooth submanifold of $M$,
and write $\Sigma_{M^g}\to M^g$ for its normal bundle inside $TM$.
It is naturally endowed with an action of the centralizer $Z_g\subset G$ of $g$,
and we write $\mathfrak{z}_g$ for its Lie algebra. Then all of the equivariant characteristic
forms considered above make sense for the action of $Z_g$ on the vector bundles
$\Sigma_{M^g}$, $TM^g$, and $L^m|_{M^g}$ over $M^g$. In particular, we write
$R^{\Sigma_{M^g}}_{\mathfrak{z}_g}(X)\in\Omega^2(M^g,\End(\Sigma_{M^g}))$
for the equivariant curvature associated with the connection on $\Sigma_{M^g}$ induced 
by the Levi-Civita connection of $g^{TM}$.

We can now state the following fundamental \emph{Kirillov formula},
which we present in a guise that can be found in \cite[Theorem 3.1 and Remark  (4) after Theorem 2.1]{meinrenken96}.
\begin{thm}\label{Kirfla}
{\cite[Th.\,3.23]{berline-vergne82}}
For any $m\in\N$, $g\in G$, and
$X\in\mathfrak{z}_g$ sufficiently close to $0$ we have
\begin{equation*}
\chi^{(m)}(g\exp(X))=\int_{M^g}\,
\frac{\Td_{\mathfrak{z}_g}(M^g,X)\,\Tr_{L^m}[g^{-1}]\ch_{\mathfrak{z}_g}(L^m|_{M^g},X)}
{\det_{\Sigma_{M^g}}(1-g\exp(R^{\Sigma_{M^g}}_{\mathfrak{z}_g}(X)/2\pi i))}.
\end{equation*}
\end{thm}

Note that $g^{-1}\in G$ acts fiberwise on the complex
line bundle $L^m|_{M^g}$ by multiplication with a scalar denoted by $\Tr_{L^m}[g^{-1}]$ in Theorem \ref{Kirfla}. On the one hand, applied to $g=e$,  the Kirillov formula  takes the  elegant form
\begin{equation}\label{Kirflag=e}
\chi^{(m)}(\exp(X))=\int_M\,\Td_\g(M,X)\,\ch_\g(L^m,X)
\end{equation}
for any $m\in\N$ and $X\in\g$ sufficiently close to $0$, compare \cite[Theorem 8.2]{berline-getzler-vergne}.  On the other hand,
Theorem \ref{Kirfla} applied to $X=0$ recovers the 
\emph{equivariant Atiyah-Segal-Singer index theorem} \cite{AS68}
\begin{equation}\label{eqindfla}
\chi^{(m)}(g)=\int_{M^g}\,
\frac{\Td(M^g)\,\Tr_{L^m}[g^{-1}]\exp(-m R^{L}/2\pi i)}
{\det_{\Sigma_{M^g}}(1-g\exp(R^{\Sigma_{M^g}}/2\pi i))},
\end{equation}
which is valid for any $m\in\N$ and $g\in G$. In particular, Theorem \ref{Kirfla} applied to $g=e$ and $X=0$
recovers  the classical
Hirzebruch-Riemann-Roch formula \eqref{HRRfla} for the usual
Riemann-Roch numbers
of Definition \ref{RRdef}.

\section{Geometry of the zero level set of the moment map}
\label{sec:sympquot}

Let $(M,\omega)$ be a compact symplectic manifold  of dimension $2n$ on which $G=S^1$ acts in a Hamiltonian fashion with moment map 
 $\J: M \to \g^\ast\simeq \R$.  In this section,  we will study the geometry of $\J^{-1}(\{0\})$ using a local normal form theorem for the $S^1$-action, and introduce corresponding retractions that will be needed later for the implementation of homotopy arguments.  We begin by introducing some general notation that will be frequently used.
\begin{notat}\label{def:F}
\emph{We denote by $\F$ the set of all connected components of the fixed point set
$M^{S^1}\subset M$. Given $F\in \F$, we write $\J(F)\in \R$ for the constant value of the moment map $\J:M\to \R$ on $F$. Furthermore, we introduce
the following subsets of
$\F$:
\begin{align*}
\F_+&:=\{F\in \F\,|\, \J\geq \J(F)\text{ near }F\},&
\F_-&:=\{F\in \F\,|\, \J\leq \J(F)\text{ near }F\},
\end{align*}
For the connected components $F\in \F$ in $\J^{-1}(\{0\})$ satisfying  $\J(F)=0$ we set}
$$
\F^0:=\{F\in \F\,|\, F\subset \J^{-1}(\{0\})\},\qquad \F^0_\pm:=\F^0\cap \F_\pm\,.
$$
\end{notat}

\subsection{Local normal form theorem}\label{sec:stratif} Let $F\in \F$ and note that $F$ is a symplectic submanifold of $M$. When considering a fiber bundle over $F$ with total space $E$, we will use the notation $\nu_E:E\to F$ for its bundle projection. Let $\nu_{\Sigma_F}:\Sigma_F\to F$ be the symplectic normal bundle of $F$ inside
$TM$, so that we have a decomposition
\begin{equation}\label{TMFdec}
(TM,\omega)|_{F}=(TF,\omega_{F})\oplus (\Sigma_F,\omega^{\perp})
\end{equation}
into symplectic vector bundles, where $\omega_F:=\mathrm{inc}_F^\ast\,\omega\in\Omega^2(F)$ and $\omega^\perp$ is the fiberwise restriction of $\omega$ to
$\Sigma_F$.  Note that the action of $S^1$ on $(M,\omega)$ induces a fiberwise
action on $\Sigma_F$, and choose an
$S^1$-invariant compatible complex structure
$J_{\Sigma_F}\in\End(\Sigma_F)$ on $(\Sigma_F,\omega^{\perp})$.
The Formula \eqref{gTX} then induces a Hermitian
metric $g_{\Sigma_F}:=\omega^{\perp}(\cdot,J_{\Sigma_F}\cdot)$ on
the complex bundle $(\Sigma_F,J_{\Sigma_F})$, making $\Sigma_F$ into a Hermitian
vector bundle over $F$.  For any $k\in\N$ and under the identification
\eqref{identfla} we write
\begin{equation*}
\Sigma_F^{(k)}:=\left\{v\in\Sigma_F~\big|~\exp(X)\cdot v=e^{2\pi ikx}v\,\text{ for all }
X\in\g\right\} \label{eq:SigmakF}
\end{equation*}
for the associated isotypic component, as well as 
$W:=\{k\in\Z~|~\Sigma_F^{(k)}\neq 0\}\subset\Z$ for the set of weights,
which is finite and does not contain $0$ by definition of $\Sigma_F$,
and let $\ell_F^{(k)}:=\mathrm{rk}_\C\,\Sigma_F^{(k)}$ be the complex rank of the vector bundle $\Sigma_F^{(k)}$.
We thus have a finite decomposition
\begin{equation}\label{eq:splitting111k}
\Sigma_F=:\bigoplus_{k\in W}\Sigma_F^{(k)}
\end{equation}
into hermitian subbundles, with associated structure group
\begin{equation*}
K_F:=\prod_{k\in W}\,\mathrm{U}\big(\ell_F^{(k)}\big)\,,
\end{equation*}
so that there is a principal $K_F$-bundle $\nu_{P_F}:P_F \rightarrow F$
satisfying
\begin{equation}\label{sigmaprbdle}
\Sigma_F\cong P_F\times_{K_F}\bigoplus_{k\in W}\C^{\ell_F^{(k)}}\,,
\end{equation}
and inducing the decomposition \eqref{eq:splitting111k},
where $K_F$ acts linearly on
$\bigoplus_{k\in W}\C^{\ell_F^{(k)}}$. The diagonal action of $S^1$ on
$\bigoplus_{k\in W}\C^{\ell_F^{(k)}}$ of weight $k\in W$ on each summand
$\C^{\ell_F^{(k)}}$ commutes with the action of $K_F$, inducing
an $S^1$-action on the right-hand side of \eqref{sigmaprbdle}, which makes it an
$S^1$-equivariant identification.
Note also that the symplectic structure on  $\Sigma_F$ is induced by the standard
complex structure on each $\C^{\ell_F^{(k)}}$ via this decomposition.
In what follows, we will write $W_\pm :=\{w\in W\,|\,\pm w\in\N\}$ for the sets of positive and negative weights, respectively, so that  by \eqref{eq:splitting111k} there is a decomposition
\begin{equation}
\Sigma_F=\Sigma_F^+\oplus\Sigma_F^-,\label{eq:sigmapmdecomp}
\end{equation}
where
$\Sigma_F^\pm:=\bigoplus_{k\in W_\pm}\Sigma_F^{(k)}$. Setting
$
\ell^\pm_F:=\mathrm{rk}_\C\,\Sigma_F^\pm =\sum_{k\in W_\pm }\ell_F^{(k)}
$,
the decomposition \eqref{eq:sigmapmdecomp} is induced by the natural embedding
$K_F\subset U(\ell^+_F)\times U(\ell^-_F)$.
We write $\C^{\ell^\pm_F}=\bigoplus_{k\in W_\pm}\C^{\ell_F^{(k)}}$ and
set $\ell_F:=\mathrm{rk}_\C\,\Sigma_F$.

Let $\omega_{\Sigma_F}$ be a symplectic form on a neighborhood $V_F$ of the zero section in $\Sigma_F$ such that its restriction to the fibers of $\Sigma_F$ coincides with the standard symplectic
form induced by $\omega^{\perp}$, while
the restriction $(T\Sigma_F,\omega_{\Sigma_F})|_F$
to the zero section $F\subset\Sigma_F$
coincides with $(TM,\omega)|_F$ via the natural identification of
$TM$ with $T\Sigma_F$ over $F$.
The next proposition gives a canonical description of the moment map $\J:M\to\R$
in a neighborhood of each $F \in \F$, and constitutes the basis for the computations of the next sections.
It is  a straightforward consequence of the equivariant Darboux lemma, 
as explained for instance in \cite[Section 2.2]{GLS96}. 

\begin{prop}[\bf Local normal form theorem]\label{prop:localnormform}
For each $F\in \F$ there is an $S^1$-equivariant
symplectomorphism $\Phi_F:V_F\rightarrow U_F$ from the open tubular neighborhood $(V_F,\omega_{\Sigma_F})$ of the zero section in $\Sigma_F$ onto an open
neighborhood $U_F\subset (M,\omega)$ of $F$, such that
\bqn
\label{eq:30.06.2018}
\J \circ \Phi_F([\wp,w])=  \frac 12 Q_F(w) + \J(F),\quad\text{for all}\quad
[\wp,w]\in \Sigma_F\cong P_F\times_{K_F}\bigoplus_{k\in W}\C^{\ell_F^{(k)}},
\eqn
where  $Q_F$ is the $K_F$-invariant quadratic form on $\bigoplus_{k\in W}\C^{\ell_F^{(k)}}$  defined by
\bq
Q_F(w):=\sum_{k\in W}\,k\,\norm{\pi_k(w)}^2.  
\label{eq:QF}
\eq
\end{prop}
We write
\begin{equation}\label{ZFdef}
Z_F:=\{[\wp,w]\in \Sigma_F\,|\,w\neq 0\text{ and } Q_F(w)=0\}\subset\Sigma_F
\end{equation}
for the submanifold formed fiberwise by the smooth points of the $0$-level set of the quadratic form $Q_F$. 
Recalling the stratification \eqref{stratiffla}, the set $Z_F$
then corresponds to the regular stratum of $J^{-1}(\{0\})$
in the local normal coordinates of Proposition \ref{prop:localnormform}, since
\begin{equation}
\Phi_{F}(V_F\cap Z_F)=U_F\cap \J^{-1}(\{0\})^\mathrm{reg}.
\label{eq:diagram_withPhiF}
\end{equation}

\subsection{Local model}\label{sec:elementary} In this section,
we introduce for each $F\in \F$ coordinates on
$\bigoplus_{k\in W}\C^{\ell_F^{(k)}}$ with respect to the zero level set of the quadric $Q_F$, which will be extended in
Section \ref{sec:applicationtonormal} to the symplectic normal bundle using the framework of the previous section. We will consider
$\bigoplus_{k\in W}\C^{\ell_F^{(k)}}$ as equipped
with the diagonal $S^1$-action of weight $k\in W$ on each summand $\C^{\ell_F^{(k)}}$.   To begin, we define the sets
\begin{align}\label{Cbulletdef}
\begin{split}
\C^{\ell_F^\pm}_{\bullet}&:= \C^{\ell_F^\pm}\setminus \{0\}, \qquad\qquad\qquad\qquad\text{in case}\quad \ell_F^\pm>0,\ell_F^\mp=0, \\
\C^{\ell_F^+,\ell_F^-}_{\bullet\bullet}&:=(\C^{\ell_F^+}\setminus \{0\})\times (\C^{\ell_F^-}\setminus \{0\}),\quad\text{in case}\quad \ell_F^+,\ell_F^->0.
\end{split}
\end{align}
Further, for any $k_0\in W$ we write
\bqn
\pi_{k_0}:\bigoplus_{k\in W}\C^{\ell_F^{(k)}}\to\C^{\ell_F^{(k_0)}}\label{eq:pik0}
\eqn 
for the canonical
projection, and introduce the weighted Euclidean norms 
\bq
 \norm{w}_{F\pm}^2:=\pm
 \sum_{k\in W^\pm}\,k\,\norm{\pi_k(w)}^2,
 \quad\text{for all}\quad w\in \C^{\ell_F^\pm}=\bigoplus_{k\in W^\pm}\C^{\ell_F^{(k)}}.\label{eq:normsF}
 \eq
 We then introduce the ellipsoids
\begin{equation}\label{ellipsoid}
 S^{2\ell_F^\pm-1}_\pm:=\big\{w \in\C^{\ell_F^\pm} \mid\; \norm{w}_{F\pm}=1 \big\},\quad\text{when}\quad\ell_F^\pm>0\,,
\end{equation}
 Note that the norms $\norm{\cdot}_{F\pm}$ are both $S^1$- and $K_F$-invariant.
The quadratic form $Q_F$ from \eqref{eq:QF} can be written in norm notation as
\bqn
Q_F(w)=\norm{w^+}_{F+}^2-\norm{w^-}_{F-}^2,\quad\text{for all}\quad w=(w^+,w^-)\in \C^{\ell_F^++\ell_F^-}.\label{eq:QFnorms}
\eqn
Later, it will also be convenient to consider the weighted Hermitian norm 
\bq
\norm{w}_F^2:=\norm{w^+}_{F+}^2+\norm{w^-}_{F-}^2,\quad\text{for all}\quad w=(w^+,w^-)\in \C^{\ell_F^++\ell_F^-}.\label{eq:Fnorm}
\eq
\smallskip
 Next, for each weight $k\in W$,
we  introduce the following standard symplectic and angular forms over
$\C^{\ell_F^{(k)}}$ for any  $w=(w_1'+iw_1'',\cdots, w_{\ell_F^{(k)}}'+iw_{\ell_F^{(k)}}'')
\in\C^{\ell_F^{(k)}}$ by setting
\bq
\omega_{(k)}|_w:=\sum_{j=1}^{\ell_F^{(k)}}\,dw'_j\wedge dw''_j,\qquad
\alpha_{(k)}|_w:=\sum_{j=1}^{\ell_F^{(k)}}\,w'_j \d w''_j-w''_j \d w'_j\,. \label{eq:omega0alpha0}
\eq
They are related by $\omega_{(k)}=\frac{1}{2}d\alpha_{(k)}$.
For any $X\in\g$ inducing the fundamental vector field
$\tilde X\in\Cinft(\C^{\ell_F^{(k)}},T\C^{\ell_F^{(k)}})$
associated with the $S^1$-action on $\C^{\ell_F^{(k)}}$ of weight $k\in W$ 
we then get
\bq
(\tilde X \imult \alpha_{(k)})|_w = k\, x\, \norm{w}^2,\quad\text{for all}\quad
w\in \C^{\ell_F^{(k)}}\,, \label{eq:connform0}
\eq
where $x\in\R$ is the image of $X\in\g$ via the identification \eqref{identfla}. If we now introduce the  $1$-forms 
\begin{equation}\label{eq:alphapm}
\alpha^{\pm}:=\pm\sum_{k\in W^{\pm}} \pi_k^*\alpha_{(k)}
\in\Omega^1(\C^{\ell_F^\pm},\R)\,,
\end{equation}
then \eqref{eq:connform0}  and \eqref{eq:normsF}  imply 
that the $1$-forms  defined by 
\begin{equation}\label{connCdef}
\theta^\pm|_w:=\frac{\alpha^{\pm} |_w}{ \norm{w}_{F\pm}^2}, \qquad  \text{if } w \in \C^{\ell_F^\pm}_{\bullet},
\end{equation}
\begin{equation}\label{connC}
\theta|_w:=\frac{1}{2}\left(\frac{\alpha^{+}|_w}{ \norm{w}_{F+}^2}+\frac{\alpha^{-}|_w}{ \norm{w}_{F-}^2}\right), \qquad \text{if } w \in \C^{\ell_F^+,\ell_F^-}_{\bullet\bullet},
\end{equation}
are connection forms in the sense of Formula \eqref{connfla} for the diagonal $S^1$-action of weight $k\in W$ on the $k$-th summand on $\C^{\ell_F^\pm}_{\bullet}$ and $\C^{\ell_F^+,\ell_F^-}_{\bullet\bullet}$ respectively defined in \eqref{Cbulletdef}. 

The following simple lemma  will be crucial for all our local and global considerations. Recall the general notation  for inclusions and projections introduced at the beginning of Section \ref{sec:2}. 
\begin{lem}\label{lem:fundamentaldefinite}  Let $S\subset \bigoplus_{k\in W}\C^{\ell_F^{(k)}}$ be an embedded smooth submanifold, $f:(0,\infty)\to (0,\infty)$
a smooth function,  and consider the smooth map
$\Psi_{S,f}:S\times(0,\infty)\to\bigoplus_{k\in W}\C^{\ell_F^{(k)}}$, $(w,r)\mapsto f(r)w$. Then one has for each $k\in W$
\bqn
\Psi_{S,f}^\ast (\pi_k^*\alpha_{(k)}) =f(r)^2\, \mathrm{pr}_S^\ast
\mathrm{inc}^\ast_S(\pi_k^*\alpha_{(k)})\,.\label{eq:Psiastalphadef}
\eqn
\end{lem}
\begin{proof}
Let $v\in T_{y}(S  \times(0,\infty))$ be a tangent vector at $y:=(w, r)\in S  \times(0,\infty)$ and $\gamma_v(t)=(w(t), r(t))$ a smooth curve with $\dot \gamma_v(0)=v$, $\gamma_v(0)=y$. Identifying $T_{y}(S  \times(0,\infty))$ with a subspace of $\bigoplus_{k\in W}\C^{\ell_F^{(k)}}\times \R$ we get
\bqn 
(\Psi_{S,f})_{\ast}(v)=\frac d {dt} (f(r(t)) w(t))|_{t=0}=\frac d {dt} f(r(t))|_{t=0} w +  f(r)\dot w(0),
\eqn
so that with the identification $w_j\equiv w_j' \gd_{w_j'}+w_j'' \gd_{w_j''}$ in the notation of \eqref{eq:omega0alpha0} we obtain
\begin{align*}
 \Psi_{S,f}^\ast(\pi_k^*\alpha_{(k)})|_y (v)&= \frac d {dt} f(r(t))|_{t=0} \underbrace{(w_j' dw_j''- w_j'' dw_j')|_{\Psi_{S,f}(y)} (w)}_{=0}\\
 &\quad+ f(r) (w_j' dw_j''- w_j'' dw_j')|_{\Psi_{S,f}(y)}(\dot w(0))\\
&= f(r)^2\, \mathrm{pr}_S^\ast \mathrm{inc}_S^\ast \pi_k^*\alpha_{(k)}|_y (v).
\end{align*}
\end{proof}
We will now proceed with separate treatments of the cases
$0\in\partial\,\J(M)$ and  $0\in\Int\J(M)$, respectively. As explained in the introduction, in the first case $\F_0$ consists of  only one component  $F=\J^{-1}(\mklm{0})$ and $Q_F$ is definite, while in the second case,  $Q_F$ is indefinite for all $F\in\F^0$.

\subsubsection{Definite case}\label{sec:defcase0} Let us begin with the easier definite case, so that either $\ell_F^-=0$ or  $\ell_F^+=0$, and write $\ell_F:=\ell_F^+>0$ if $F\in \F_+$  and $\ell_F:=\ell_F^->0$ if $F\in \F_-$.  The diagonal action of $S^1$ on
$\C^{\ell_F}=\bigoplus_{k\in W}\C^{\ell_F^{(k)}}$
considered in Section \ref{sec:stratif} restricts to a locally free
action on $S^{2\ell_F-1}$, and we introduce the connection form
\begin{equation}\label{conndefdef}
\Theta:=\mathrm{inc}^\ast_{S^{2\ell_F-1}}\theta \in\Omega^1(S^{2\ell_F-1},\R)
\end{equation}
as the restriction of \eqref{connCdef} to $S^{2\ell_F-1}$.

We now introduce spherical polar coordinates in $\C^{\ell_F}_{\bullet}$ 
via the diffeomorphism 
\begin{equation}\label{eq:Psidefinite}
\begin{split}
\Psi:S^{2\ell_F-1}  \times(0,\infty)&\longrightarrow \C^{\ell_F}_{\bullet}\\ \qquad (w,r)&\longmapsto r w\,,
\end{split}
\end{equation}
which is equivariant for the action of $S^1$ introduced above on the first factor
and the trivial action on the second factor of $S^{2\ell_F-1}\times(0,\infty)$.
It defines on $\C^{\ell_F}_{\bullet}$ the radial coordinate
\bq
r=\norm{w}_F=\sqrt{|Q_F(w)|},\quad\text{for all}\quad w\in \C^{\ell_F}_{\bullet},\label{eq:polardefinite}
\eq
which is related to the quadratic form $Q_F$ by
\bq
Q_F|_{\C^{\ell_F}_{\bullet}}=\pm r^2\quad \text{if }F\in \F_\pm.\label{eq:rQ_Fdefinite}
\eq
Applying Lemma \ref{lem:fundamentaldefinite} to $S=S^{2\ell_F-1}$ and $f=\mathrm{id}$ one immediately sees that  
\bqn
\Psi^\ast (\alpha^{\pm}) =r^2\, \mathrm{pr}_{S^{2\ell_F-1}}^\ast
\mathrm{inc}^\ast_{S^{2\ell_F-1}}\alpha^{\pm} \,.
\eqn
Since $\omega_\mathrm{std}=\pm\frac{1}{2}d\theta$ and  $\Theta=\mathrm{inc}^\ast_{S^{2\ell_F-1}}\theta$, we deduce from this the  useful identity
\begin{equation}\label{eq:omegaFpullbackdefinite}
\begin{split}
\Psi^\ast (\omega_{\mathrm{std}}|_{\C^{\ell_F}_{\bullet}})&=\pm\frac{1}{2}d(r^2\, \mathrm{pr}_{S^{2\ell_F-1}}^\ast \Theta)=\pm r\d r \wedge\mathrm{pr}_{S^{2\ell_F-1}}^\ast \Theta \pm\frac{1}{2}r^2\, \mathrm{pr}_{S^{2\ell_F-1}}^\ast d\Theta
\quad \text{if }F\in \F_\pm.
\end{split}
\end{equation}

\subsubsection{Indefinite case}
Let us now turn to the indefinite case, so that $\ell_F^+,\ell_F^->0$. Departing from the quadric formed by the zero level set of $Q_F$ we define the slit quadric
\begin{equation}\label{splitquadric}
\mathcal{Q}_\times:=Q_F^{-1}(\{0\})\setminus\{0\}\subset \C^{\ell_F^+,\ell_F^-}_{\bullet\bullet}\,,
\end{equation}
which is a smooth submanifold of the set $\C^{\ell_F^+,\ell_F^-}_{\bullet\bullet}$
defined in \eqref{Cbulletdef}. 
The diagonal action of $S^1$ on
$\C^{\ell_F}=\bigoplus_{k\in W}\C^{\ell_F^{(k)}}$
considered in Section \ref{sec:stratif} restricts to locally free 
$S^1$-actions on 
$\mathcal{Q}_\times\subset \C^{\ell_F^+,\ell_F^-}_{\bullet\bullet}$ and
the product of ellipsoids 
\begin{equation}\label{prodellipsoid}
S^{2\ell_F^+-1}_+\times S^{2\ell_F^--1}_-\subset \C^{\ell_F^++\ell_F^-}=
 \C^{\ell_F^+}\oplus \C^{\ell_F^-}
\end{equation}
of the ellipsoids \eqref{ellipsoid}.
Besides the connection form \eqref{connC} we introduce now also  the $1$-form 
\begin{equation}\label{anticonnCindef}
\bar \theta|_w:=\frac{1}{2}\left(\frac{\alpha^+|_w}{ \norm{w}_{F+}^2}-\frac{\alpha^-|_w}{ \norm{w}_{F-}^2}\right)
\end{equation}
on $\C^{\ell_F^+,\ell_F^-}_{\bullet\bullet}$, and introduce the corresponding  restrictions 
\bq\label{eq:15.06.2022}
\Theta:=\mathrm{inc}^\ast_{S^{2\ell_F^+-1}_+\times S^{2\ell_F^--1}_-}\theta, 
\qquad \overline \Theta:=\mathrm{inc}^\ast_{S^{2\ell_F^+-1}_+\times S^{2\ell_F^--1}_-}\bar\theta
\eq
to $S^{2\ell_F^+-1}_+\times S^{2\ell_F^--1}_-$. Clearly, 
$\Theta\in\Omega^1(S^{2\ell_F^+-1}_+\times S^{2\ell_F^--1}_-,\R)$
is  a connection for the $S^1$-action on 
 $S^{2\ell_F^+-1}_+\times S^{2\ell_F^--1}_-$.  
 On the other hand, the $1$-form
$\overline \Theta\in\Omega^1(S^{2\ell_F^+-1}_+\times S^{2\ell_F^--1}_-,\R)$ is basic
in the sense of Proposition \ref{pi*iso}.

Let us now introduce polar moment coordinates  in
$\C^{\ell_F^+,\ell_F^-}_{\bullet\bullet}$ via the diffeomorphism
\begin{align}\begin{split}
\Psi:S^{2\ell_F^+-1}_+\times S^{2\ell_F^--1}_-\times(0,\infty)\times  \R &\longrightarrow \C^{\ell_F^+,\ell_F^-}_{\bullet\bullet}, \\
(w^+,w^-,r,q)&\longmapsto \Big (\sqrt{\sqrt{r^4+q^2}+q}\; w^+, \sqrt{\sqrt{r^4+q^2}- q}\; w^-\Big)\,.\label{eq:Psi}\end{split}
\end{align}
The map $\Psi$ is at the basis of our
local model. It is equivariant under the action of $S^1$ introduced above on
$S^{2\ell_F^+-1}_+\times S^{2\ell_F^--1}_-$ and the trivial action on $(0,\infty)\times  \R $, and defines for  $w \in \C^{\ell_F^+,\ell_F^-}_{\bullet\bullet}$
the  coordinates 
\begin{equation}\label{eq:rscoord}
r=\sqrt{\norm{w^+}_{F+} \, \norm{w^-}_{F-}}, \qquad
q=\frac{1}{2}Q_F(w).
\end{equation}
The slit quadric $\mathcal{Q}_\times\subset \C^{\ell_F^+,\ell_F^-}_{\bullet\bullet}$
defined in \eqref{splitquadric}
corresponds to the set $\{q=0\}$, while the coordinate $r$ is radial on
${\mathcal{Q}_\times}$. Moreover, composing the retraction
\begin{equation}\label{eq:retractionQslit}
\begin{split}
\mathrm{ret}_{\mathcal{Q}_\times}:\C^{\ell_F^+,\ell_F^-}_{\bullet\bullet}&\longrightarrow \mathcal{Q}_\times,\\
w=(w^+,w^-)&\longmapsto \sqrt{\norm{w^+}_{F+} \, \norm{w^-}_{F-}}\Big(\frac{w^+}{\norm{w^+}_{F+}},\frac{w^-}{\norm{w^-}_{F-}}\Big)
\end{split}
\end{equation}
with $\Psi$ gives us an $S^1$-equivariant surjection
\begin{align}\begin{split}
\pi_{\mathcal{Q}_\times}:S^{2\ell_F^+-1}_+\times S^{2\ell_F^--1}_-\times(0,\infty)\times \R&\longrightarrow \mathcal{Q}_\times,\\
(w^+,w^-,r,q)&\longmapsto \Psi(w^+,w^-,r,0)=(rw^+,rw^-).\label{eq:PsiQ}\end{split}
\end{align}
Let $\mathrm{pr}_{S^{2\ell_F^+-1}_+\times S^{2\ell_F^--1}_-}:S^{2\ell_F^+-1}_+\times S^{2\ell_F^--1}_-\times (0,\infty)\times \R\longrightarrow S^{2\ell_F^+-1}_+\times S^{2\ell_F^--1}_-$ be  the canonical projection. We then have 

\begin{lem}\label{lem:fundamentalindefinite} Write $\alpha:=\frac12 (\alpha^++\alpha^-)$ and  $\overline \alpha:=\frac12 (\alpha^+-\alpha^-)$.  In the coordinates \eqref{eq:rscoord},
one has 
\begin{align*}
\Psi^\ast (\alpha|_{\C^{\ell_F^+,\ell_F^-}_{\bullet\bullet}}) &=\sqrt{r^4+q^2}  \;\mathrm{pr}_{S^{2\ell_F^+-1}_+\times S^{2\ell_F^--1}_-}^\ast \Theta +  q \,\mathrm{pr}_{S^{2\ell_F^+-1}_+\times S^{2\ell_F^--1}_-}^\ast \overline \Theta,\\
\Psi^\ast (\overline\alpha|_{\C^{\ell_F^+,\ell_F^-}_{\bullet\bullet}}) &=\sqrt{r^4+q^2}\;  \mathrm{pr}_{S^{2\ell_F^+-1}_+\times S^{2\ell_F^--1}_-}^\ast \overline \Theta +  q\, \mathrm{pr}_{S^{2\ell_F^+-1}_+\times S^{2\ell_F^--1}_-}^\ast  \Theta\,.
\end{align*}
\end{lem}
\begin{proof}Consider $\Psi$ and $\mathrm{pr}_{S^{2\ell_F^+-1}_+\times S^{2\ell_F^--1}_-}$ as families of maps defined on $S^{2\ell_F^+-1}_+\times S^{2\ell_F^--1}_-\times (0,\infty)$ and parametrized by $q\in \R$. Defining $f_q^\pm:(0,\infty)\to (0,\infty)$ by $f_q^\pm(r):=\sqrt{\sqrt{r^4+q^2}\pm q}$, we can apply Lemma \ref{lem:fundamentaldefinite} with  $S=S^{2\ell_F^+-1}_+\times S^{2\ell_F^--1}_-$ and $f=f_q^\pm$ to get $$\Psi^\ast (\alpha^\pm|_{\C^{\ell_F^+,\ell_F^-}_{\bullet\bullet}})=f^\pm_q(r)^2\mathrm{pr}_{S^{2\ell_F^+-1}_+\times S^{2\ell_F^--1}_-}^\ast\mathrm{inc}^\ast_{S^{2\ell_F^+-1}_+\times S^{2\ell_F^--1}_-} \alpha^\pm\,,$$ concluding the proof.
\end{proof}
A direct consequence of Lemma \ref{lem:fundamentalindefinite} is the following indefinite version of \eqref{eq:omegaFpullbackdefinite}.
\begin{cor}\label{cor:Pamplona} In the coordinates \eqref{eq:Psi} one has 
\bqn
\Psi^\ast (\omega_\mathrm{std}|_{\C^{\ell_F^+,\ell_F^-}_{\bullet\bullet}}) = \pi_{\mathcal{Q}_\times}^\ast\mathrm{inc}_{\mathcal{Q}_\times}^\ast\omega_\mathrm{std}+   d\Big (  q\, \mathrm{pr}_{S^{2\ell_F^+-1}_+\times S^{2\ell_F^--1}_-}^\ast  \Theta + \big ( \sqrt{r^4+q^2} -r^2\big ) \mathrm{pr}_{S^{2\ell_F^+-1}_+\times S^{2\ell_F^--1}_-}^\ast \overline \Theta\Big ).
\eqn
\end{cor}
\begin{proof}As $\omega_\mathrm{std}=d\bar\alpha$, we apply $d$ on both sides of the equation in Lemma \ref{lem:fundamentalindefinite} involving $\bar\alpha$, yielding
\begin{align*}\label{eq:08.08.2022}
\Psi^\ast (\omega_\mathrm{std}|_{\C^{\ell_F^+,\ell_F^-}_{\bullet\bullet}}) &=  d\big ( q\, \mathrm{pr}_{S^{2\ell_F^+-1}_+\times S^{2\ell_F^--1}_-}^\ast  \Theta + \sqrt{r^4+q^2}  \mathrm{pr}_{S^{2\ell_F^+-1}_+\times S^{2\ell_F^--1}_-}^\ast \overline \Theta\big ). 
\end{align*}
We now assert that
\bq\label{eq:09.08.2022}
\pi_{\mathcal{Q}_\times}^\ast\mathrm{inc}_{\mathcal{Q}_\times}^\ast\omega_\mathrm{std}=d\big (r^2 \mathrm{pr}_{S^{2\ell_F^+-1}_+\times S^{2\ell_F^--1}_-}^\ast \overline  \Theta \big ).
\eq
Indeed, since  $\overline \Theta=\mathrm{inc}^\ast_{S^{2\ell_F^+-1}_+\times S^{2\ell_F^--1}_-}\bar \alpha$, this
is a consequence of  the relation
\bqn
\pi_{\mathcal{Q}_\times}^\ast\mathrm{inc}_{\mathcal{Q}_\times}^\ast\bar \alpha=  r^2\, \mathrm{pr}_{S^{2\ell_F^+-1}_+\times S^{2\ell_F^--1}_-}^\ast\mathrm{inc}_{S^{2\ell_F^+-1}_+\times S^{2\ell_F^--1}_-}^\ast\bar \alpha
\eqn
which follows from Lemma \ref{lem:fundamentaldefinite}
applied with $S=S^{2\ell_F^+-1}_+\times S^{2\ell_F^--1}_-$,
$f=\mathrm{id}$, and $q\in \R$ as an additional parameter.
\end{proof}

\subsection{Application to the symplectic normal bundle}\label{sec:applicationtonormal}
We now translate the fiberwise considerations from Section \ref{sec:elementary} into global statements on the symplectic normal bundle $\Sigma_F$, identified with the associated bundle $P_F\times_{K_F}\bigoplus_{k\in W}\C^{\ell_F^{(k)}}$ as in \eqref{sigmaprbdle}.
 
Choose a $K_F$-connection on the principal bundle $\nu_{P_F}:P_F \rightarrow F$,
inducing a Hermitian connection $\nabla^{\Sigma_F}$ on $\Sigma_F$
that preserves the decomposition \eqref{eq:splitting111k}.
This provides us with a splitting of the short exact sequence
\bqn 
0 \longrightarrow \nu_{\Sigma_F}^*\Sigma_F \longrightarrow T\Sigma_F \longrightarrow \nu_{\Sigma_F}^*TF \longrightarrow 0 
\eqn
of vector bundles over $\Sigma_F$, yielding the global decomposition
\bq\label{splitconnSigmaF}
T\Sigma_F\cong T^\mathrm{hor}F \oplus \nu_{\Sigma_F}^*\Sigma_F\,,
\eq
where $T^\mathrm{hor}F\simeq\nu_{\Sigma_F}^*TF$ is the horizontal
distribution defined by $\nabla^{\Sigma_F}$ and 
$\nu_{\Sigma_F}^*\Sigma_F\subset T\Sigma_F$ is the vertical tangent bundle to the fibers. Choose a compatible almost complex structure $J_F\in\End(TF)$ over
$(F,\omega_F)$. Together with the natural
complex structure $J_{\Sigma_F}\in\End(\Sigma_F)$, this induces via the decomposition
\eqref{splitconnSigmaF} an $S^1$-invariant almost complex structure
$J\in\End(T\Sigma_F)$ over the total space of $\Sigma_F$. 

For any $k_0\in W$,
we define $f_{k_0}\in\Cinft(\Sigma_F,\R)$ by setting 
\begin{equation}\label{fkdef}
f_{k_0}([\wp,w]):=\frac{1}{2}\norm{\pi_{k_0}(w)}^2,\quad
\text{for all}\quad
[\wp,w]\in P_F\times_{K_F}\bigoplus_{k\in W}\C^{\ell_F^{(k)}},
\end{equation}
where $\norm{\cdot}$ denotes the standard Hermitian norm. Note that  $f_{k_0}$
is well-defined since for each $k\in W$ the structure group $K_F$ acts by 
$\mathrm{U}(\ell_F^{(k)})$ on $\C^{\ell_F^{(k)}}$ and thus preserves the Hermitian norm. 
Using this, we define the $1$-form $\alpha_{k_0}\in\Omega^1(\Sigma_F,\R)$ by 
\begin{equation}\label{eq:31.08.22}
\alpha_{k_0}(v):=-df_{k_0}(Jv),\quad\text{for all}\quad
v\in T\Sigma_F\,.
\end{equation}
Comparing with \eqref{eq:omega0alpha0},
one readily checks that for any $p\in F$ the restriction of
$\alpha_{k_0}\in\Omega^1(\Sigma_F,\R)$
to the fiber $(\Sigma_F)_p\subset\Sigma_F$,
seen as a submanifold
in the total space of $\Sigma_F$, satisfies
\begin{equation}\label{thetadef}
\alpha_{k_0}|_{(\Sigma_F)_p}=\pi_{k_0}^*\alpha_{(k_0)}\,,
\end{equation}
in any trivialization $(\Sigma_F)_p\simeq\bigoplus_{k\in W}\C^{\ell_F^{(k)}}$
compatible with the $K_F$-principal bundle structure \eqref{sigmaprbdle}.
Following \eqref{eq:alphapm},
\eqref{connCdef} and \eqref{connC}, we write
\begin{equation}\label{alphaFdef}
 \alpha^\pm_F:=\pm \sum_{k\in W^\pm}\alpha_k\,,
\end{equation}
so that the $1$-forms given by
\begin{equation}\label{eq:XiFdef}
\theta_F^\pm|_{[\wp,w]}:=\frac{\alpha^\pm_F|_{[\wp,w]}}{\norm{w}_{F\pm}^2}\quad \text{for}\quad [\wp,w]\in P_F\times_{K_F}\C^{\ell_F^\pm}_{\bullet},
\end{equation}
\begin{equation}\label{eq:XiF}
\theta_F|_{[\wp,w]}:=\frac{1}{2}\left(\frac{\alpha^+_F|_{[\wp,w]}}{\norm{w}_{F+}^2}+
\frac{\alpha^-_F|_{[\wp,w]}}{\norm{w}_{F-}^2}\right)\quad \text{for}\quad [\wp,w]\in P_F\times_{K_F}\C^{\ell_F^+,\ell_F^-}_{\bullet\bullet},
\end{equation}
define connections for the $S^1$-action on $P_F\times_{K_F}\C^{\ell_F^\pm}_{\bullet}$ and $P_F\times_{K_F}\C^{\ell_F^+,\ell_F^-}_{\bullet\bullet}$,  respectively, 
such that for all $p\in F$ the restriction of $\theta_F^\pm,\,\theta_F$ to the fiber over $p$ coincides with $\theta^\pm,\,\theta$ in any trivialization $(\Sigma_F)_p\simeq\bigoplus_{k\in W}\C^{\ell_F^{(k)}}$
compatible with the $K_F$-principal bundle structure \eqref{sigmaprbdle}.
We also define the vertical
form $\omega_{\mathrm{vert}}\in\Omega^2(\Sigma_F,\R)$ by
\begin{equation}\label{omstdsigmaFdef}
\omega_{\mathrm{vert}}:=\frac{1}{2}\sum_{k\in W}\,d\alpha_k\,.
\end{equation}
Comparing with \eqref{eq:omega0alpha0} and \eqref{thetadef},
for all $p\in F$ the restriction of 
$\omega_{\mathrm{vert}}\in\Omega^2(\Sigma_F,\R)$
to the fiber $(\Sigma_F)_p\subset\Sigma_F$ satisfies
\bq
\omega_{\mathrm{vert}}|_{(\Sigma_F)_p}=\omega_{\mathrm{std}}\,,\label{eq:relomegastdomegavert}
\eq
in any trivialization $(\Sigma_F)_p\simeq\bigoplus_{k\in W}\C^{\ell_F^{(k)}}$
compatible with the $K_F$-principal bundle structure \eqref{sigmaprbdle}, where $\omega_{\mathrm{std}}$ is the standard symplectic form
on $\bigoplus_{k\in W}\C^{\ell_F^{(k)}}$ induced by \eqref{eq:omega0alpha0} for
each $k\in W$.  Furthermore, note that the restriction
$T\Sigma_F|_F$ over the zero section $F\subset\Sigma_F$ naturally
identifies with the restriction $TM|_F$ over $F\subset M$, and  the restriction of
$\omega_{\mathrm{vert}}$ to $T\Sigma_F|_F$
coincides with $\omega^{\perp}$ in the decomposition \eqref{TMFdec}. Consequently, we can explicitly choose the symplectic form $\omega_{\Sigma_F}$ from Proposition \ref{prop:localnormform} 
to be given by 
\begin{equation}\label{sympSigmaF}
\omega_{\Sigma_F}:=\omega_{\mathrm{vert}}+\nu_{\Sigma_F}^*\,\omega_F \in\Omega^2(\Sigma_F,\R).
\end{equation}
By construction, the almost complex structure $J\in\End(T\Sigma_F)$ over $\Sigma_F$
is compatible with $\omega_{\Sigma_F}$, and we write
$g^{T\Sigma_F}:=\omega_{\Sigma_F}(\cdot,J\cdot)$ for the induced $S^1$-invariant
Riemannian metric over $V_F$.

\subsubsection{Definite case} Let $F\in \F$ be such that $Q_F$ is definite. Recall the definition of $\ell_F$ from the beginning of Section \ref{sec:defcase0}. As the linear action of $K_F$ on $\C^{\ell_F}$ preserves  $S^{2\ell_F-1}$, we can consider the associated sphere bundle 
$
S_F:= P_F\times_{K_F}S^{2\ell_F-1}
$ inside $\Sigma_F=P_F\times_{K_F}\C^{\ell_F}$, with bundle projection $\nu_{S_F}:S_F\to F$. The  $S^1$-action on $\Sigma_F$
restricts to a locally free $S^1$-action on $S_F$ for which we have the connection 
\bq
\label{eq:connS_Fdefinite}
\Theta_{S_F}:=\mathrm{inc}^\ast_{S_F}\theta_F
\eq
given by restriction of the connection $\theta_F$ from \eqref{eq:XiFdef}.
Using \eqref{Cbulletdef} to define the subbundle
\bqn
 \Sigma_{F\bullet}:=P_F\times_{K_F}\C^{\ell_F}_{\bullet} = \Sigma_F\setminus F, 
\eqn
the $S^1$-equivariant diffeomorphism $\Psi$ from \eqref{eq:Psidefinite} globalizes to an $S^1$-equivariant diffeomorphism 
\begin{equation}
\begin{split}
\Psi_{F}: S_F\times(0,\infty)&\longrightarrow \Sigma_{F\bullet}\\
([\wp,w], r)&\longmapsto [\wp,\Psi(w,r)]\,,\label{eq:PsiSF}
\end{split}
\end{equation}
where  $S^1$ acts on $S_F\times(0,\infty)$ by the product of the $S^1$-action on $S_F$ and the trivial action on $(0,\infty)$. $\Psi_{F}$  promotes the coordinate $r$ from \eqref{eq:polardefinite} to a global fiber-radial coordinate $r$ on the bundle $\Sigma_{F\bullet}$. Its relation to $Q_F$  is given by \eqref{eq:rQ_Fdefinite} with  $\C^{\ell_F}_{\bullet}$ replaced by $\Sigma_{F\bullet}$.  Note that $S_F\times(0,\infty)$ is a fiber bundle over $F$ with the projection $\nu_{S_F\times(0,\infty)}:=\nu_{S_F}\circ\mathrm{pr}_{S_F}$. We state the global version of \eqref{eq:omegaFpullbackdefinite} in 
\begin{cor}\label{cor:PsiSFastdefinite}
In the coordinate \eqref{eq:polardefinite}
provided by the diffeomorphism \eqref{eq:PsiSF} we have
\bqn
\Psi_{F}^\ast (\omega_{\Sigma_F}|_{\Sigma_{F\bullet}})=\nu_{S_F\times(0,\infty)}^\ast\omega_F\pm d\Big(\frac{r^2}{2} \mathrm{pr}_{S_F}^\ast \Theta_{S_F}\Big)\quad
\text{if }F\in \F_\pm.
\label{eq:PsiFastomegaSF}
\eqn
\end{cor}
\begin{proof}
In view of the definition \eqref{eq:connS_Fdefinite} of $\Theta_{S_F}$
and the definition
\eqref{eq:PsiSF} of $\Psi_{F}$ the assertion follows from
\eqref{sympSigmaF}, \eqref{eq:omegaFpullbackdefinite}, and \eqref{eq:relomegastdomegavert}.
\end{proof}
For a better overview, we illustrate the maps appearing in Corollary \ref{cor:PsiSFastdefinite} in the diagram
\bq\label{diagram_retract_EFdefinite}
\begin{split}
\begin{tikzpicture}[node distance=2cm, auto]
\node (A) {$\Sigma_{F\bullet}$} ;   
\node (B) [right=3cm of A] {$S_F\times(0,\infty)$};
\node (C) [above of=A] {$S_F$}; 
\node (D) [below of=A] {$F$}; 
\draw[left hook->>] (B) to node[below] {$\cong$} node[above] {$\Psi_{F}$} (A); 
\draw[<-left hook] (A) to node[right] {$\mathrm{inc}_{S_F}$} (C);
\draw[->>] (B) to node[above right] {$\mathrm{pr}_{S_F}$} (C);
\draw[->>] (B) to node[below right] {$\nu_{S_F\times(0,\infty)}$} (D);
\draw[->>, bend right] (C) to node[left] {$\nu_{S_F}$} (D);
\draw[->>] (A) to node[right] {$\nu_{\Sigma_{F\bullet}}$}  (D);
\end{tikzpicture}
\end{split}
\eq
in which everything commutes except the triangle formed by $\mathrm{pr}_{S_F}$, $\mathrm{inc}_{S_F}$, and $\Psi_{F}$, which commutes only up to $S^1$-equivariant homotopy.

\subsubsection{Indefinite case}\label{indefsympbdle}
 Let $F\in \F$ be such that $Q_F$ is indefinite. As the linear action of $K_F$ on $\C^{\ell_F^++\ell_F^-}$ preserves the bi-ellipsoid \eqref{prodellipsoid},
 we can consider the associated fiber bundle
 \begin{equation}\label{SFdef}
S_F:= P_F\times_{K_F}(S^{2\ell_F^+-1}_+\times S^{2\ell_F^--1}_-)
\end{equation}
inside $\Sigma_F\cong P_F\times_{K_F}\C^{\ell_F^++\ell_F^-}$. 
Furthermore, the $0$-level set $Z_F\subset\Sigma_F\setminus F$
defined in \eqref{ZFdef} has the fiber bundle structure 
\begin{equation}\label{ZFdef2}
Z_F\cong P_F\times_{K_F}\mathcal{Q}_\times\,
\end{equation}
associated with
the slit quadric \eqref{splitquadric}. Recall that we write $\nu_{S_F}:S_F\to F$ and $\nu_{Z_F}:Z_F\to F$ for the fiber bundle projections. The action of $S^1$ on $\Sigma_F$ restricts to locally free $S^1$-actions on $S_F$ and $Z_F$, respectively, and we set 
\begin{equation}
\label{eq:connS_F}
\begin{split}
\Theta_{S_F}&:=\mathrm{inc}_{S_F}^\ast\theta_F=\frac{1}{2}(\Theta^+_{S_F}+\Theta^-_{S_F}),\quad\overline \Theta_{S_F}:=\frac{1}{2}(\Theta^+_{S_F}-\Theta^-_{S_F})\quad\text{with}\quad
  \Theta^\pm_{S_F}:=\mathrm{inc}_{S_F}^\ast \theta^\pm_F,\\ 
\Theta_{Z_F}&:=\mathrm{inc}_{Z_F}^\ast\theta_F=\frac{1}{2}(\Theta^+_{Z_F}+\Theta^-_{Z_F}),\quad \overline \Theta_{Z_F}:=\frac{1}{2}(\Theta^+_{Z_F}-\Theta^-_{Z_F})\quad \text{with}\quad
\Theta^\pm_{Z_F}:=\mathrm{inc}_{Z_F}^\ast\theta_F^\pm\,,
\end{split}
 \end{equation}
where $\theta_F$ is the connection form on $\Sigma_F\setminus F$ from  \eqref{eq:XiF} and the $1$-forms $\alpha_k$ were defined in \eqref{eq:31.08.22}.  Thus, $\Theta_{S_F}$ and $\Theta_{Z_F}$ are connections 
for the $S^1$-actions on $S_F$ and $Z_F$, respectively. On the other hand,  $\overline \Theta_{S_F}$ and $\overline \Theta_{Z_F}$  are basic differential forms in the sense of Proposition \ref{pi*iso}. Using \eqref{Cbulletdef} to define the
subbundle
$$
 \Sigma_{F\bullet\bullet}:=P_F\times_{K_F}\C^{\ell_F^+,\ell_F^-}_{\bullet\bullet},
 $$ 
the $S^1$-equivariant diffeomorphism $\Psi$ from \eqref{eq:Psi} globalizes to an $S^1$-equivariant diffeomorphism 
\begin{equation}\label{eq:Psiindef}
\begin{split}
\Psi_{F}: S_F\times(0,\infty)\times \R &\longrightarrow \Sigma_{F\bullet\bullet}\\
([\wp,w], r,q)&\longmapsto [\wp,\Psi(w,r,q)],
\end{split}
\end{equation}
where $S^1$ acts on $S_F\times(0,\infty)\times  \R$ by the product of the $S^1$-action on $S_F$ and the trivial action on $(0,\infty)\times \R$. The map $\Psi_{F}$  defines coordinates \eqref{eq:rscoord}
on $\Sigma_{F\bullet\bullet}$, in which $Z_F$ coincides with $\{q=0\}$.  
The retraction $\mathrm{ret}_{\mathcal{Q}_\times}$ from \eqref{eq:retractionQslit}  and the surjection $\pi_{\mathcal{Q}_\times}$ from \eqref{eq:PsiQ} globalize to an $S^1$-equivariant retraction and an $S^1$-equivariant  surjection onto $Z_F$, respectively, by defining 
\begin{equation}
\begin{split}
\mathrm{ret}_{Z_F}: \Sigma_{F\bullet\bullet} &\longrightarrow Z_F\\
[\wp,w]&\longmapsto [\wp,\mathrm{ret}_{\mathcal{Q}_\times}(w)],
\end{split}\label{eq:retZF}
\end{equation}
as well as
\begin{equation*}\label{retZF}
\begin{split}
\pi_{Z_F}:=\mathrm{ret}_{Z_F}\circ \Psi_{F}: S_F\times(0,\infty)\times \R &\longrightarrow Z_F,\\
([\wp,w], r,q)&\longmapsto [\wp,\pi_{\mathcal{Q}_\times}(w,r,q)]=\Psi_{F}([\wp,w], r,0)=[\wp,rw].
\end{split}
\end{equation*}
The restricted diffeomorphism 
\bq
\Psi_{F}:S_F\times(0,\infty)\times\{0\}\longrightarrow Z_F,\label{eq:PsiZF}
\eq
denoted again just by $\Psi_F$ for simplicity, provides the radial coordinate $r$ on $Z_F$. 
Note that the space $S_F\times(0,\infty)$, which we identify canonically with $S_F\times(0,\infty)\times\{0\}$, as well as the space $S_F\times(0,\infty)\times \R$, are fiber bundles over $F$ with the corresponding projections $\nu_{S_F\times(0,\infty)}$, $\nu_{S_F\times(0,\infty)\times \R}$ given by the composition of $\nu_{S_F}$ with the canonical projections onto $S_F$, respectively. These constructions lead us to the following 

\begin{cor}\label{cor:PamplonaF}
In the coordinates \eqref{eq:rscoord} defined by the diffeomorphism \eqref{eq:Psiindef} one has
\bqn
\Psi_{F}^\ast (\omega_{\Sigma_F}|_{\Sigma_{F\bullet\bullet}})= \pi_{Z_F}^\ast \mathrm{inc}_{Z_F}^\ast\omega_{\Sigma_F}+   d\Big (  q\, \mathrm{pr}_{S_F}^\ast  \Theta_{S_F} + \big ( \sqrt{r^4+q^2} -r^2\big ) \mathrm{pr}_{S_F}^\ast \overline \Theta_{S_F}\Big ), 
\eqn
as well as
\begin{align*}
\Psi_{F}^\ast(\Theta_{Z_F})&=\mathrm{pr}_{S_F}^\ast  \Theta_{S_F},\\
\Psi_F^\ast\mathrm{inc}_{Z_F}^\ast\omega_{\Sigma_F}&=r^2\mathrm{pr}_{S_F}^\ast\mathrm{inc}_{S_F}^\ast\omega_{\mathrm{vert}} + r \d r\wedge\mathrm{pr}_{S_F}^\ast \overline \Theta_{S_F}  +\nu_{S_F\times(0,\infty)}^*\,\omega_F.
\end{align*}
\end{cor}
\begin{proof}
By the definition of $\omega_{\Sigma_F}$ in \eqref{sympSigmaF} and \eqref{eq:relomegastdomegavert}, of $\Theta_{Z_F}$ in \eqref{eq:connS_F}, and of $\omega_{\Sigma_F}$, $\omega_{\mathrm{vert}}$ in \eqref{sympSigmaF} and \eqref{omstdsigmaFdef}, respectively, the claims are direct consequences of Lemma \ref{lem:fundamentalindefinite} and Corollary \ref{cor:Pamplona}.
\end{proof}

For an overview, the maps introduced above, and in particular those involved in Corollary \ref{cor:PamplonaF}, are illustrated in the  diagram 
\begin{equation}
\begin{split}
\hspace*{-0.5em}\begin{tikzpicture}[node distance=3cm, auto]
\node (A) {$\Sigma_{F\bullet\bullet}$} ;  
\node (E) [above=3cm of A] {$Z_F$};  
\node (B) [right=4.5cm of A] {$S_F\times(0,\infty)\times \R$};
\node (D) [below=3cm of A] {$F$}; 
\node (H) [above=3cm of B] {$S_F\times(0,\infty)\times \{0\}\cong S_F\times(0,\infty) $};
\node (C) [above=3cm of E] {$S_F$};
\node (J) [above right=0.2cm and -0.5cm of H] {}; 
\node (K) [above right=0cm and 1cm of B] {}; 
\draw[left hook->>] (B) to node[below] {$\cong$} node[above] {$\Psi_{F}$} (A); 
\draw[left hook->>] (H) to node[below] {$\cong$} node[above] {$\Psi_{F}$}  (E); 
\draw[left hook->] (C) to node[right] {$\mathrm{inc}_{S_F}$} (E);
\draw[left hook->] (E) to node[left] {$\mathrm{inc}_{Z_F}$} (A);
\draw[left hook->] (H) to node[left] {$\mathrm{inc}_{S_F\times(0,\infty)\times \{0\}}$} (B);
\draw[->>] (B) to node[below left] {$\pi_{Z_F}$} (E);
\draw[->>, bend right=20] (A) to node[right] {$\mathrm{ret}_{Z_F}$} (E);
\draw[->>, bend right=20] (B) to node[above right] {$\mathrm{pr}_{S_F\times(0,\infty)}$} (H);
\draw[->>] (H) to node[below left] {$\mathrm{pr}_{S_F}$} (C);
\draw[-, out=30, in=-20] (B) to (J.east);
\draw[->>, out=160, in=-18] (J.east) to node[above right] {$\mathrm{pr}_{S_F}$} (C.east);
\draw[->>] (B) to node[above left] {$\nu_{S_F\times(0,\infty)\times \R}$} (D);
\draw[->>, bend right] (C) to node[left] {$\nu_{S_F}$} (D.west);
\draw[->>, out=225, in=115] (E.west) to node[right] {$\nu_{Z_F}$} (D);
\draw[->>] (A) to node[right] {$\nu_{\Sigma_{F\bullet\bullet}}$}  (D);
\draw[->>, out=-153, in=21] (K.east) to node[below right] {$\nu_{S_F\times(0,\infty)}$} (D.east);
\draw[-, out=-25, in=27] (H) to (K.east);
\end{tikzpicture}
\end{split}
\label{eq:diagram1}
\end{equation}
This is the indefinite version of \eqref{diagram_retract_EFdefinite}. In the diagram, everything commutes except the triangles involving an inclusion followed by the retraction $\mathrm{ret}_{Z_F}$, a projection, or the surjection $\pi_{Z_F}$, which commute only up to $S^1$-equivariant homotopy. Note that the two canonical projections onto $S_F$ are both denoted by $\mathrm{pr}_{S_F}$, whereas each fiber bundle projection onto $F$ has its own name.

\section{Kirwan maps} 
\label{sec:Kirwan} 

In the setting of Section \ref{sec:eqcoh}, let $(M,\omega,\J)$ be a Hamiltonian $S^1$-manifold such that the restriction
of the $S^1$-action to $\J^{-1}(\{0\})$ is locally free, and write
$\textup{inc}_0:\J^{-1}(\{0\})\to M$ for the inclusion map.
The following
classical notion goes back to Kirwan \cite{Kir84},
which we introduce here only for $G=S^1$. 
\begin{prop}\label{Kirwanth}
{\cite[(18)]{kalkman95}}
Assume that $S^1$ acts locally freely on the level set $\J^{-1}(\{0\})\subset M$,
and let $\Theta\in\Omega^1(\J^{-1}(\{0\}),\R)$ be a connection for the $S^1$-action.
Then,  for any $\varrho(x)\in\Omega_{S^1}^*(M)$ 
\begin{equation}\label{Kirwandef}
\kappa(\varrho(X)):=\textup{inc}_0^*\,\varrho\big(\frac{i}{2\pi}d\Theta\big)-\frac{i}{2\pi}\Theta\wedge\left(\widetilde{X}\imult\textup{inc}_0^*\,
\varrho\big(\frac{i}{2\pi}d\Theta\big)\right)
\end{equation}
defines a map of complexes
$\kappa:(\Omega_{S^1}^*(M),d_{\g})\to
(\Omega(\J^{-1}(\{0\}))^{S^1}_{\mathrm{bas}},d)$,
inducing via the isomorphism of Proposition \ref{pi*iso} a morphism
\begin{equation}\label{Kircohwandef}
\kappa:H_{S^1}^*(M)\longrightarrow H^*(\M_0)\,,
\end{equation}
which does not depend on the choice of the connection
$\Theta\in\Omega^1(\J^{-1}(\{0\}),\R)$. The morphism $\kappa$ is called the \emph{Kirwan map}.
\end{prop}
The notation
$\textup{inc}_0^*\,\varrho(\frac{i}{2\pi}d\Theta)\in\Omega^*(\J^{-1}(\{0\}),\R)$
means that we substitute $\frac{i}{2\pi}d\Theta$ for $X$ in  $\textup{inc}_0^*\,\varrho(X)$, taking the wedge product with the coefficients of the power series $\textup{inc}_0^*\,\varrho(X)$.

\subsection{Regular Kirwan map of a singular symplectic quotient}

Let now $(M,\omega,\J)$ be a general Hamiltonian $S^1$-manifold, so that
the symplectic quotient \eqref{M0} is only a stratified space, and assume
that $0\in\partial\,\J(M)$.
To introduce the regular Kirwan maps appearing in Theorem \ref{mainth}, we will
need the following definition.

\begin{definition}\label{connormalform}
A connection $\Theta\in\Omega^1(\J^{-1}(\{0\})^{\mathrm{reg}},\R)$
for the $S^1$-action on $\J^{-1}(\{0\})^{\mathrm{reg}}$ is said
to have \emph{normal form near the singularities} if for each $F\in\F^0$ one has 
\bqn
\Phi_F^\ast (\Theta|_{U_F\cap \J^{-1}(\{0\})^{\mathrm{reg}}}) = \Theta_{Z_F}|_{V_F\cap Z_F},\label{eq:conncompat}
\eqn
where  $U_F\subset M$  is a neighborhood of $F$ as in Proposition 
\ref{prop:localnormform} and $\Theta_{Z_F}\in\Omega^1(Z_F,\R)$
denotes the connection form 
\eqref{eq:connS_F} over $Z_F\subset\Sigma_F$ via \eqref{eq:diagram_withPhiF}. 
\end{definition}
Using a partition of unity, one readily constructs a connection with normal
form near all $F\in\F^0$. 
\begin{definition}\label{Kirwantildedef}
The \emph{regular Kirwan map} is the map of complexes
\begin{equation}\label{Kirwantilde}
\kappa:(\Omega_{S^1}^*(M),d_{\g})\longrightarrow
(\Omega(\M_0^{\mathrm{reg}}),d)
\end{equation}
defined by the restriction of the formula \eqref{Kirwandef} to the regular stratum $\J^{-1}(\{0\})^{\mathrm{reg}}$ of the stratification
\eqref{stratiffla} and using a connection 
$\Theta\in\Omega^1(\J^{-1}(\{0\})^{\mathrm{reg}},\R)$ with normal form
near the singularities in the sense of Definition \ref{connormalform}.
\end{definition}

The regular Kirwan map of Definition \ref{Kirwantildedef} appears in the regular term of 
the Hirzebruch-Riemann-Roch type formula of Theorem \ref{mainth} inside an integral
over the non-compact manifold $\M_0^{\mathrm{reg}}$, which is naturally oriented by its symplectic
volume form. The next Lemma shows that this integral in fact converges.

\begin{lem}\label{lem:existenceintegral}
For all $\varrho\in\Omega_{S^1}^*(M)$ and $m\in\N$, the integral
\begin{equation}
\int_{\M_0^{\mathrm{reg}}}e^{m\omega_0}\kappa(\varrho):=\lim_{j\to \infty}\int_{U_j}e^{m\omega_0}\kappa(\varrho)\label{existenceintfla}
\end{equation}
converges, where $\{U_j\}_{j\in \N}$ denotes an arbitrary exhaustion of $\M_0^{\mathrm{reg}}$ by relatively compact open sets, and does not depend on the choice of the exhaustion.
\end{lem}
\begin{proof}
For any $F\in\F^0$, recall from \eqref{eq:diagram_withPhiF} that the diffeomorphism $\Phi_F$ of the local normal form of Proposition
\ref{prop:localnormform} sends $Z_F\cap V_F$ to $\J^{-1}(\{0\})^{\mathrm{reg}}\cap U_F$, where $Z_F=Q_F^{-1}(\{0\})\setminus F$ has been introduced in
Equation \eqref{ZFdef}. Keeping in mind the coordinates \eqref{eq:rscoord},
for any small enough $\epsilon>0$ let
$B_{F,\epsilon}\subset \J^{-1}(\{0\})\cap U_F$
be the neighborhood of $F$ defined by
\begin{equation}\label{BFesp}
B_{F,\epsilon}:=\left\{x\in\J^{-1}(\{0\})~\bigg|~ \Phi_F^{-1}(x)=[\wp,w]\text{ with }
\sqrt{\|w^+\|_{F+}\,\|w^-\|_{F-}}<\epsilon\right\}\,.
\end{equation}
Let $\Theta\in\Omega^1(\J^{-1}(\{0\})^{\mathrm{reg}},\R)$ be a connection of normal form near the singularities as in Definition \ref{connormalform}, so that 
its restriction to
$\J^{-1}(\{0\})^{\mathrm{reg}}\cap U_F$ pulls back along $\Phi_F$
to the connection $\Theta_{Z_F}\in\Omega^1(Z_F,\R)$ introduced in \eqref{eq:connS_F}. 
Let
$\{U_j\}_{j\in \N}$ be an exhaustion of $\M_0^{\mathrm{reg}}$ by relatively compact open sets.
Using Lemma \ref{intXi=int} and
the fact that the Kirwan map \eqref{Kirwantilde} is a map of complexes we now get for any $\varrho\in \Omega_{S^1}^*(M)$, $j,\,m\in\N$, and $\eps>0$ small enough the equality 
\begin{equation*}\label{8239604869048693407}
\begin{split}
\int_{U_j}e^{m\omega_0}\kappa(\varrho)&=
\int_{\pi_0^{-1}(U_j)}
e^{m\inc_0^*\omega}\kappa(\varrho)\wedge\Theta\\
&=\int_{\pi_0^{-1}(U_j)\setminus \bigcup_{F\in\F^0}B_{F,\epsilon}}
e^{m\inc_0^*\omega}\kappa(\varrho)\wedge\Theta+\sum_{F\in\F^0}\int_{\pi_0^{-1}(U_j)\cap B_{F,\epsilon}}
e^{m\inc_0^*\omega}\kappa(\varrho)\wedge\Theta.
\end{split}
\end{equation*}
Since $\J^{-1}(\{0\})^{\mathrm{reg}}\setminus \bigcup_{F\in\F^0}B_{F,\epsilon}$ is compact,
we have a well-defined limit
\[
\lim_{j\to \infty}\int_{\pi_0^{-1}(U_j)\setminus \bigcup_{F\in\F^0}B_{F,\epsilon}}
e^{m\inc_0^*\omega}\kappa(\varrho)\wedge\Theta=\int_{\J^{-1}(\{0\})^{\mathrm{reg}}\setminus \bigcup_{F\in\F^0}B_{F,\epsilon}}
e^{m\inc_0^*\omega}\kappa(\varrho)\wedge\Theta,
\]
which is clearly independent of the exhaustion. Let us now fix $F\in \F^0$ and focus on the corresponding summand 
\begin{align*}
\int_{\pi_0^{-1}(U_j)\cap B_{F,\epsilon}}
e^{m\inc_0^*\omega}\kappa(\varrho)\wedge\Theta&=\int_{\pi_0^{-1}(U_j)\cap B_{F,\epsilon}}
e^{m\inc_0^*\omega}\varrho\Big(\frac{i}{2\pi} d \Theta\Big)\wedge\Theta\\
&=\int_{\Phi_F^{-1}(\pi_0^{-1}(U_j)\cap B_{F,\epsilon})}       e^{m\, \mathrm{inc}_{Z_F}^\ast  \omega_{\Sigma_F}  }  \mathrm{inc}_{Z_F}^\ast\Phi_F^\ast  \varrho\Big(\frac{i}{2\pi}d\Theta_{Z_F}\Big)\,\Theta_{Z_F}.
\end{align*}
Since $\{U_j\}_{j\in \N}$ is an exhaustion of $\M_0^{\mathrm{reg}}$,
for all $j\in\N$ large enough, there exist $R_j,r_j\in(0,\eps)$ with $R_j>r_j$  satisfying $R_j,r_j\to 0$ as $j\to \infty$ such that the set
$\mathscr U_{j,\eps}:=\Psi_F^{-1}(\Phi_F^{-1}(\pi_0^{-1}(U_j)\cap B_{F,\epsilon}))\subset S_F\times (0,\infty)$ satisfies $S_F\times (R_j,\eps)\subset\mathscr U_{j,\eps}\subset S_F\times (r_j,\eps)$.
%
Pulling back along $\Psi_F$ and using Corollary \ref{cor:PamplonaF}, we get
\begin{multline}
\int_{\pi_0^{-1}(U_j)\cap B_{F,\epsilon}}
e^{m\inc_0^*\omega}\kappa(\varrho)\wedge\Theta=\\
\int_{\mathscr U_{j,\eps}}   e^{m (r^2\mathrm{pr}_{S_F}^\ast\mathrm{inc}_{S_F}^\ast\omega_{\mathrm{vert}} + r \d r\wedge\mathrm{pr}_{S_F}^\ast \overline \Theta_{S_F}  +\nu_{S_F\times(0,\infty)}^*\,\omega_F)}  
\Psi_{F}^\ast\mathrm{inc}_{Z_F}^\ast\Phi_F^\ast  \varrho\Big(\frac{i}{2\pi} \mathrm{pr}_{S_F}^\ast  d \Theta_{S_F}\Big)\, \mathrm{pr}_{S_F}^\ast  \Theta_{S_F}.\label{flacvint}
\end{multline}
Now, by definition \eqref{eq:connS_F} of $\Theta_{S_F}\in\Omega^1(Z_F,\R)$ one readily sees that
all forms in the integrand of the right-hand side of \eqref{flacvint} are bounded on  
the whole space $S_F\times (0,\eps)$, except possibly the form
$\Psi_{F}^\ast\mathrm{inc}_{Z_F}^\ast\Phi_F^\ast  \varrho$.
However, from \eqref{eq:PsiZF} and \eqref{eq:Psi} we have
\begin{equation}\label{eq:PsiZFintlem}
\begin{split}
\Psi_{F}:S_F\times(0,\infty)&\longrightarrow Z_F\subset\Sigma_F,\\
([\wp,w],r)&\longmapsto [\wp,rw],
\end{split}
\end{equation}
so that $\Psi_{F}$ is simply the restriction to $Z_F\subset\Sigma_F$ of polar coordinates
over $\Sigma_F$.
Since $\Phi_F^\ast\varrho$ is defined on a neighborhood of the $0$-section $V_F\subset\Sigma_F$
rather than just on $Z_F\cap V_F$, one readily computes that it stays bounded in polar coordinates
in the neighborhood $V_F\subset\Sigma_F$ of the $0$-section,
hence its restriction $\Psi_{F}^\ast\mathrm{inc}_{Z_F}^\ast\Phi_F^\ast  \varrho$ to $Z_F$
stays bounded over $S_F\times (0,\eps)$ as well.
Therefore, the integral over the whole $S_F\times (0,\eps)$ of the integrand of the right-hand side
of \eqref{flacvint} converges and we have a well-defined limit
\begin{multline}
\lim_{j\to+\infty}\int_{\pi_0^{-1}(U_j)\cap B_{F,\epsilon}}
e^{m\inc_0^*\omega}\kappa(\varrho)\wedge\Theta=\\
\int_{S_F\times (0,\eps)}   e^{m (r^2\mathrm{pr}_{S_F}^\ast\mathrm{inc}_{S_F}^\ast\omega_{\mathrm{vert}} + r \d r\wedge\mathrm{pr}_{S_F}^\ast \overline \Theta_{S_F}  +\nu_{S_F\times(0,\infty)}^*\,\omega_F)}  
\Psi_{F}^\ast\mathrm{inc}_{Z_F}^\ast\Phi_F^\ast  \varrho\Big(\frac{i}{2\pi} \mathrm{pr}_{S_F}^\ast  d \Theta_{S_F}\Big)\, \mathrm{pr}_{S_F}^\ast  \Theta_{S_F},
\end{multline}
which is independent of the exhaustion. This finishes the proof.
\end{proof}
Note from Equation \eqref{Kirwandef} that for any $\varrho\in\Omega^*_{S^1}(M)$ with real values, 
the integrand of formula \eqref{existenceintfla} defines a signed measure on 
$\M_0^{\mathrm{reg}}$ up to a multiplicative constant, so that the fact
that the limit \eqref{existenceintfla} exists and is independent of the exhaustion
actually shows that the integral converges in the usual sense. 
This is also evident from the proof of Lemma \ref{lem:existenceintegral}.

The following result will be used in a crucial way in Section
\ref{fullWitsec} to establish Theorem \ref{thm:wittenasymp} on
the topological interpretation of the asymptotics of
the Witten integral, which forms the core of the proof of Theorem \ref{mainth}.

\begin{lem}\label{Kirwanexact}
Let $\sigma\in\Omega^*_{S^1}(M)$ be such that for every
$F\in\F^0$ we have $\inc_F^*\sigma\equiv 0$.
Then for
all $m\in\N$, we have
\begin{equation}\label{Kirwanexactfla}
\int_{\M_0^{\mathrm{reg}}}e^{m\omega_0}\kappa(d_\g\sigma)=0\,,
\end{equation}
where $\kappa:\Omega_{S^1}^*(M)\to
\Omega(\M_0^{\mathrm{reg}})$ denotes the Kirwan map \eqref{Kirwantilde}.
\end{lem}
\begin{proof}  For any $F\in\F^0$, recall the definition of the neighborhood \eqref{BFesp} of $F$.  Using Lemma \ref{intXi=int} and
the fact that the Kirwan map \eqref{Kirwantilde} is a map of complexes
we can apply Stokes' theorem to get
\begin{equation*}\label{Kirwanexactcomput}
\begin{split}
\int_{\M_0^{\mathrm{reg}}}e^{m\omega_0}\kappa(d_\g\sigma)&=
\int_{\J^{-1}(\{0\})^{\mathrm{reg}}}
e^{m\inc_0^*\omega}d\kappa(\sigma)\wedge\Theta\\
&=\lim_{\epsilon\to 0}
\int_{\J^{-1}(\{0\})\setminus \bigcup_{F\in\F^0}B_{F,\epsilon}}
e^{m\inc_0^*\omega}d\kappa(\sigma)\wedge\Theta\\
&=\sum_{F\in\F^0}\lim_{\epsilon\to 0}
\int_{\partial B_{F,\epsilon}}
e^{m\inc_0^*\omega}\kappa(\sigma)\wedge\Theta\\
&\qquad\qquad\qquad\qquad-\lim_{\epsilon\to 0}
\int_{\J^{-1}(\{0\})\setminus \bigcup_{F\in\F^0}B_{F,\epsilon}}
e^{m\inc_0^\ast\omega}\kappa(\sigma)\wedge d\Theta\\
&=\sum_{F\in\F^0}\lim_{\epsilon\to 0}
\int_{\partial B_{F,\epsilon}}
e^{m\inc_0^*\omega_0}
\inc_0^\ast\sigma\big(\frac{i}{2\pi}d\Theta\big)\wedge\Theta,
\end{split}
\end{equation*}
where the second term of the third line vanishes since the integrand
is basic for the $S^1$-action and therefore its top form component is zero. Next, we define for all small enough $\eps>0$ the diffeomorphism
\begin{equation}
\begin{split}
\tilde b_{F,\epsilon}:S_F&\longrightarrow
\Phi_F^{-1}(\partial B_{F,\epsilon})=\Psi_F(S_F\times\{\epsilon\}\times\{0\})\subset Z_F\subset \Sigma_F,\\
p=[\wp,w]&\longmapsto\Psi_F(p,\eps,0)=[\wp,\eps w]\label{eq:bepsdiffeo}
\end{split}
\end{equation}
using the diffeomorphism $\Psi_F$ from \eqref{eq:Psiindef}, and we pull back along $\Phi_F\circ \tilde b_{F,\epsilon}$   to get
\begin{equation}\label{Kirwanexactcomput2}
\int_{\M_0^{\mathrm{reg}}}e^{m\omega_0}\kappa(d_\g\sigma)
=\sum_{F\in\F^0}\lim_{\epsilon\to 0}
\int_{S_F}e^{m\tilde b_{F,\epsilon}^*\omega_{\Sigma_F}}
\tilde b_{F,\epsilon}^*\Phi_F^*{\sigma_{|U_F}}\big(\frac{i}{2\pi}d\tilde b_{F,\epsilon}^*\Theta_{Z_F}\big)\wedge\tilde b_{F,\epsilon}^*\Theta_{Z_F}.
\end{equation}
We now investigate each of the three pullbacks along $\tilde b_{F,\epsilon}$ on the right-hand side. From Corollary \ref{cor:PamplonaF}, we get 
\begin{equation}\label{eq:tildebFepsXiZF}
\begin{split}
\tilde b_{F,\epsilon}^*\omega_{\Sigma_F}&=\eps^2\mathrm{inc}_{S_F}^\ast\omega_{\mathrm{vert}}+\nu_{S_F}^*\,\omega_F,\\
\tilde b_{F,\epsilon}^*\Theta_{Z_F}&=\Theta_{S_F}.
\end{split}
\end{equation}
Furthermore, we claim that for any differential form $\alpha\in\Omega^*(M)$ we have
\begin{equation}\label{claimepsto0}
\lim_{\epsilon\to 0}\tilde b_{F,\epsilon}^*\Phi_F^*{\alpha_{|U_F}}=
\nu_{S_F}^*\alpha_F\quad \text{in } \Omega^*(S_F),
\end{equation}
where we write $\alpha_F:=\inc_F^*\alpha\in\Omega^*(F)$ and we use on $\Omega^*(S_F)$ the standard Fr\'echet topology of uniform convergence of all derivatives, recalling that $S_F$ is compact. To prove the claim, we first note that it is enough to prove \eqref{claimepsto0}
when $\alpha\in\Omega^0(M)=\Cinft(M)$ is a smooth function,
since \eqref{claimepsto0} is a local formula involving pullbacks which commute with wedge products and the exterior differential $d$, which is a continuous operator $\Omega^*(S_F)\to \Omega^*(S_F)$, and any differential form of positive degree can locally be written
as a sum of wedge products of differentials of smooth functions. Now, by passing to a local trivialization of the fiber bundle $\Sigma_F$ over a local chart of $F$, the claim \eqref{claimepsto0} for $\alpha\in\Cinft(M)$ reduces to the claim that for $n,m\in \N_0$, $m>0$, the operator
\begin{equation*}
\begin{split}
b_\eps:  \CT(\R^n\times \R^m)&\longrightarrow\CT(\R^n\times S^{m-1}), \qquad 
f\longmapsto b_\eps(f)(x,y):=f(x,\eps y)\,,
\end{split}
\end{equation*}
satisfies for each $f\in \CT(\R^n\times \R^m)$ that
\bqn
b_\eps(f)\stackrel{\eps\to 0}{\longrightarrow}f_0 \quad \text{in } \CT(\R^n\times S^{m-1}),
\eqn
where $f_0(x,y):=f(x,0)$. One easily verifies that this holds, finishing the proof of the claim.

Combining \eqref{eq:tildebFepsXiZF} and \eqref{claimepsto0} allows us to evaluate the limit in \eqref{Kirwanexactcomput2}, yielding
\begin{equation*}\label{Kirwanexactcomput3}
\begin{split}
\int_{\M_0^{\mathrm{reg}}}e^{m\omega_0}\kappa(d_\g\sigma)&
=\sum_{F\in\F^0}
\int_{S_{F}}e^{m\nu_{S_F}^*\omega_F}
\nu_{S_F}^*\sigma_F\big(\frac{i}{2\pi}d\Theta_{S_F}\big)\wedge\Theta_{S_F} =0\,,
\end{split}
\end{equation*}
since $\sigma_F:=\inc_F^*\sigma\equiv 0$ by assumption.
This concludes the proof.
\end{proof}
Lemma \ref{Kirwanexact} gives a partial cohomological interpretation
for the regular Kirwan map  \eqref{Kirwantilde}, since it shows that
the integral of the left-hand side of \eqref{Kirwanexactfla} does not depend
on the choice of a connection with normal form near the singularities.
Theorem \ref{thm:wittenasymp} will actually show a posteriori that
the condition $\inc_F^*\sigma\equiv 0$ is not necessary for \eqref{Kirwanexactfla}
to hold.
 
\subsection{Kirwan map of a partial resolution}
\label{Kirwanressec}

In order to compare our results with those in \cite{meinrenken-sjamaar}, let us now discuss the Kirwan map of a partial resolution. Assume that $0\in\Int\J(M)$, and
recall the notations of Section \ref{sec:sympquot}. For any connected component $F\in\F^0$
of the set of fixed points $M^{S^1}$ contained in
$\J^{-1}(\{0\})$
we consider the \emph{complex blow-up}
\begin{equation}\label{locFblowup}
\begin{split}
\beta_{\widetilde{\Sigma}_F}:
\widetilde{\Sigma}_F&\longrightarrow\Sigma_F
\end{split}
\end{equation}
of $\Sigma_F$ along its zero section
$F\subset\Sigma_F$ in the sense of \cite[Section 8]{GS89}, which is equivariant with respect to the $S^1$-action considered in Section \ref{sec:sympquot}.
The \emph{strict transform} of the submanifold
$Z_F\subset\Sigma_F$ introduced in
\eqref{ZFdef} is defined as the closure
$\widetilde{Z}_F:= \overline{\beta_{\widetilde{\Sigma}_F}^{-1}(Z_F)}\subset\widetilde{\Sigma}_F$
and inherits a natural structure of a smooth $S^1$-submanifold of
$\widetilde{\Sigma}_F$.  Setting $\mathbb{F}^0:=\sqcup_{F\in\F^0} F\subset M^{S^1}$ we then define the
\emph{complex blow-up} of $M$ along $\mathbb{F}^0\subset M$
as the unique $S^1$-equivariant map
\begin{equation}\label{blowup}
\beta:\widetilde M\longrightarrow M
\end{equation}
which restricts to an $S^1$-equivariant diffeomorphism over $M\setminus \mathbb{F}^0$ and which,
over each open set $U_F\subset M$ as in Proposition
\ref{prop:localnormform} is given by 
the map 
$\beta_{\widetilde{\Sigma}_F}$ over $V_F$.  Like $\widetilde Z_F$, the strict transform
$\widetilde Z$ of the subset $\J^{-1}(\{0\})^{\mathrm{reg}}\subset M$ is
defined as the closure $\widetilde Z:=\overline{\beta^{-1}
(\J^{-1}(\{0\})^{{\mathrm{reg}}})}\subset\widetilde M$ and 
inherits a natural structure of a smooth
$S^1$-submanifold of $\widetilde M$. Since the $S^1$-action on $\widetilde Z$ is locally free, one is led to the following definition.

\begin{definition}\label{partialdesingdef}
The \emph{partial resolution} of the symplectic quotient $\M_0$ is 
defined as the orbifold
\begin{equation*}
\widetilde\M_0:=\widetilde Z/S^1\,,
\end{equation*}
together with the map $\beta_0:\widetilde \M_0\longrightarrow \M_0$
induced by the map \eqref{blowup} after taking quotients by the $S^1$-action. The unique form
$\widetilde\omega_0\in\Omega^2(\widetilde\M_0,\R)$ satisfying
$$\widetilde\pi_0^*\widetilde\omega_0=\inc_{\widetilde Z}^\ast\beta^*\omega\,,$$
where $\widetilde\pi_0:\widetilde Z\longrightarrow\widetilde\M_0$
denotes the quotient map,
is called the \emph{degenerate symplectic form} of $\widetilde\M_0$.
\end{definition}

Partial resolutions were introduced for arbitrary compact Lie group actions 
by Kirwan in \cite{kirwan85} in the algebraic case and by Meinrenken and Sjamaar in \cite[Section 4.1.2]{meinrenken-sjamaar} in the symplectic case.
As it is explained in \cite[Corollary 3.14]{delarue-schmitt-ramacher23}, our
definition agrees with this definition in the case $G=S^1$.

We now define the following
alternative notion of a Kirwan map for singular symplectic quotients
by an $S^1$-action,  even though $\widetilde M$ does  not admit a canonical structure of  a
Hamiltonian $S^1$-manifold. 

\begin{definition}\label{Kirwanres}
The \emph{resolution Kirwan map} is the map
\begin{equation}\label{eq:10.04.2023}
\widetilde\kappa:=\kappa\circ\beta^*:H^*_{S^1}(M)\longrightarrow H^*(\widetilde\M_0)
\end{equation}
where
$\beta^*:H^*_{S^1}(M)\to H^*_{S^1}(\widetilde M)$
is induced by the  map \eqref{blowup} and
$\kappa:H^*_{S^1}(\widetilde{M})\longrightarrow H^*(\widetilde\M_0)$
is defined at the level of complexes by the formula
\eqref{Kirwandef} with $\inc_0:\J^{-1}(\{0\})\to M$ replaced by the inclusion map 
$\inc_{\widetilde Z}:\widetilde Z\to\widetilde M$.
\end{definition}
The following result gives another topological interpretation of the regular Kirwan
map \eqref{Kirwantilde}, strenghtening Lemma \ref{Kirwanexact}
under an appropriate combinatorial condition 
on the $S^1$-action around the fixed points.

\begin{lem}\label{lem:25.04.2023}
With the notation of Section \ref{sec:stratif}, assume that 
the weights of the $S^1$-action satisfy
the condition
\begin{equation}\label{weightcondext}
\frac{\# W^+}{\sum_{k\in W^+}|k|}=\frac{\# W^-}{\sum_{k\in W^-}|k|}
\end{equation}
for each $F\in\F^0$.
Then, for any connection $\Theta\in\Omega^1(\J^{-1}(\{0\})^{\mathrm{reg}},\R)$ with
normal form near the singularities in the sense of Definition
\ref{connormalform} there exists a connection
$\widetilde \Theta\in\Omega^1(\widetilde Z,\R)$
for the $S^1$-action on the strict transform $\widetilde Z$ satisfying
\begin{equation}\label{lemblupfla}
{\widetilde\Theta}|_{\beta^{-1}(\J^{-1}(\{0\})^{\mathrm{reg}})}=
(\beta|_{\beta^{-1}(\J^{-1}(\{0\})^{\mathrm{reg}})})^*\Theta\,.
\end{equation}
In  particular, for any $\varrho\in\Omega_{S^1}^*(M)$ we have
\begin{equation}\label{kir=tildekir}
\int_{\M_0^\text{reg}}\,e^{\omega_0}\kappa(\varrho)
=\int_{\widetilde\M_0}\,e^{\widetilde\omega_0}\widetilde\kappa(\varrho)\,.
\end{equation}
\end{lem}
\begin{proof}
Let $F\in\F^0$ and equip the complex vector bundle $\nu_F:\Sigma_F\to F$ with the Hermitian norm $\frac{1}{\sqrt{2}}\|\cdot\|_F$ defined by Equation \eqref{eq:Fnorm}. Write
$\nu_{\mathscr{S}_F}:\mathscr{S}_F\to F$ for
the associated unit sphere bundle and consider the natural diffeomorphism
\begin{equation}\label{PsiFblowup}
\begin{split}
\Psi_{\mathscr{S}_F}:\mathscr{S}_F\times(0,\infty)&
\xrightarrow{~\sim~}\Sigma_F\backslash F\\
(w,t)&\longmapsto tw\,.
\end{split}
\end{equation}
Then, the bi-ellipsoid bundle $S_F\to F$ defined in 
\eqref{SFdef} satisfies $S_F\subset\mathscr{S}_F$. Next, using the notation \eqref{alphaFdef},  let us 
set
\begin{equation}\label{extformalpha}
\Theta_{\mathscr{S}_F}:=\frac{1}{2}\inc_{\mathscr{S}_F}^*(\alpha^+_F+\alpha^-_F) \in\Omega^1(\mathscr{S}_F,\R)\,.
\end{equation}
Comparing with Equations
\eqref{eq:XiF} and \eqref{eq:connS_Fdefinite} one gets the identity $\Theta_{S_F}=\inc_{S_F}^*\Theta_{\mathscr{S}_F}$,
while via the
restricted diffeomorphism \eqref{eq:PsiZF} one has
$\Psi_F^*\Theta_{Z_F}=\pr_{S_F}^*\Theta_{S_F}$.
On the other hand, by Definitions \eqref{eq:connform0} and \eqref{eq:alphapm} one readily checks that 
\eqref{extformalpha} is basic
for the $S^1$-action over $\mathscr{S}_F$ induced by multiplication with
$e^{2\pi i\theta}\in\C^*$ for all $\theta \in \R/\Z$  if and only if
the condition
\eqref{weightcondext} is satisfied.  Write
$\vartheta_F\in\Omega^1(\Sigma_F\backslash F,\R)$ for the unique form
over the complement of the zero section inside $\Sigma_F$
satisfying $\Psi_{\mathscr{S}_F}^*\vartheta_F
=\pr_{\mathscr{S}_F}^*\Theta_{\mathscr{S}_F}$.
Then, according to \cite[Section 11]{GS89} there exists a unique form
$\widetilde\vartheta_F\in\Omega^1(\widetilde{\Sigma}_F,\R)$  satisfying
\begin{equation*}
\widetilde\vartheta_F|_{\beta_{\widetilde \Sigma_F}^{-1}(\Sigma_F\backslash F)}
=(\beta_{\widetilde \Sigma_F}|_{\beta_{\widetilde \Sigma_F}^{-1}(\Sigma_F\backslash F)})^*\vartheta_{F}\,.
\end{equation*}
By restricting $\widetilde \vartheta_F$ to 
$\widetilde{Z}_F\subset\widetilde{\Sigma}_F$ one gets
a form ${\Theta}_{\widetilde Z_F}\in\Omega^1(\widetilde{Z}_F,\R)$
satisfying
$$ \Theta_{\widetilde Z_F}|_{\beta_{\widetilde \Sigma_F}^{-1}(Z_F)}
=(\beta_{\widetilde \Sigma_F}|_{\beta_{\widetilde \Sigma_F}^{-1}(Z_F)})^*\Theta_{Z_F}.
$$
In view of Definition \ref{connormalform} of a connection
$\Theta\in\Omega^1(\J^{-1}(\{0\})^{\mathrm{reg}},\R)$ with normal form
near the singularities this concludes the proof of \eqref{lemblupfla}. 
Formula \eqref{kir=tildekir} is then a
straightforward consequence of the Definitions \ref{Kirwantildedef}
and \ref{Kirwanres} of the regular and resolution Kirwan maps,
using the fact that $ \beta^{-1}_0(\M_0^\text{reg})$
is of full measure inside $\widetilde\M_0$.
\end{proof}

Lemma \ref{lem:25.04.2023} shows that, under the combinatorial condition
\eqref{weightcondext},
the first term in \eqref{mainthfla} can be interpreted topologically in terms of the
partial resolution $\widetilde \M_0$ of $\M$ introduced in Definition
\ref{partialdesingdef}. We will show in Section \ref{sec:ex} that this condition
is actually realised on a large class of examples.

\begin{rem} \label{rem:27.04.2023} It is instructive to compare the resolution Kirwan
map of Definition \ref{Kirwanres} with the Kirwan map considered by Jeffrey, Kiem,
Kirwan, and Woolf in \cite{JKKW03} in the case $G=S^1$.
They work in the complex case, so that $\M_0$ is
obtained as a GIT quotient by a $\C^*$-action of a smooth
projective variety $M$, and work instead with the so-called
\emph{intersection cohomology} $IH^{*}(\M_{0})$ of the singular quotient.
Relying on the fact that the intersection cohomology naturally
occurs as a summand inside the
usual cohomology $H^{*}(\widetilde\M_{0})$ of the partial resolution
$\widetilde\M_0$ of Definition \ref{partialdesingdef},
they consider the canonical map
\begin{equation}\label{KirwanIH}
\kappa_{IH}\colon H^{*}_{G}(M)\longrightarrow H^{*}(\widetilde\M_{0})\longrightarrow IH^{*}(\M_{0}),
\end{equation}
obtained by composition of the resolution Kirwan map
$\widetilde\kappa:H^*_{S^1}(M)\longrightarrow H^*(\widetilde\M_0)$
of Definition \ref{Kirwanres} with the canonical projection
from $H^{*}(\widetilde\M_{0})$ onto the summand $IH^{*}(\M_{0})$. Nevertheless, 
 in the general symplectic setting, the relation between
$IH^{*}(\M_{0})$ and $H^{*}(\widetilde\M_{0})$ is not as clear as
in the complex setting considered in \cite{JKKW03}, and
there might not be  a canonical choice of an intersection Kirwan map in general. 
\end{rem}

 \section{Asymptotic expansion of the Witten integral}
 \label{sec:I}

We shall now introduce  the Witten integral, which is our main tool in our approach to geometric quantization
of singular symplectic quotients. We work in the setting of Section
\ref{sec:stratif}, so that $G=S^1$ and the identification $\g\simeq\R$
of \eqref{identfla} is understood. Recall in particular the identification
\eqref{OmS1} of $\Omega_{S^1}^\ast(M)$ with the space of $S^1$-invariant
differential forms with values in entire analytic series of the variable $x\in\R$,
and the identification
\eqref{JX=xJ} of the moment map with a real-valued function $\J:M\to\R$.

\begin{definition}\label{def:mainint}
For any equivariantly closed $\varrho \in \Omega_{S^1}^\ast(M)$
and any test function $\phi \in \CT(\R)$,
the associated \emph{Witten integral} depending on a parameter $m\in\N$ is given by the formula
\bq
\label{eq:mainint}
\<\W_m(\varrho),\phi\>:=\int_{\R} \left[\int_{M} e^{2\pi imx\J} e^{m \omega} \varrho(x) \right]  \phi(x) \d x\,.
\eq
\end{definition}
The precise form of the exponents in \eqref{eq:mainint} is determined by our conventions in \eqref{eq:sign24} and  \eqref{eq:dg}, as it is crucial that
$\bar \omega(X):=\omega+2\pi i\J(X)\in\Omega^*_{S^1}(M)$ is equivariantly closed. 
  In  \eqref{eq:mainint},
$\phi \in \CT(\R)$ plays the role of a test function at which the distribution
$\W_m(\varrho)\in \mathcal \D'(\R) $ is evaluated, and this distribution is the object that we are primarily interested in, the main goal being a description of the asymptotic behavior of $\W_m(\varrho) $ as $m \to \infty$.

The analytic difficulties underlying the study of the asymptotic behavior of the Witten integral have already been overcome in the previous work \cite{kuester-ramacher20}.  In that work the problem of deriving asymptotics of  \eqref{eq:mainint} as $m\to \infty$  is considered in the more general setting of studying the asymptotic behavior of so-called \emph{generalized Witten integrals} 
\bq
\int_{\R} \int_{M} e^{i\psi(p,x)/\eps } a(p) \d M(p)\phi(x)\d x\label{int}
\eq
as $\eps  \to 0^+$,  with amplitudes  $a\in \CT(M)$, $\phi \in \CT(\R)$,
and phase function $\psi\in \Cinft(M\times\R)$ given by $\psi(p,x):=2\pi\J(p)(x)$, $dM:={\omega^n}/{n!}$ being the symplectic volume form on $M$.  Since the critical set  
 \begin{align*}
 \label{eq:4}
 \Crit(\psi)&:=\mklm{ (p,x) \in {M} \times \R: d\psi (p,x) =0}=\{(p,x)\in \J^{-1}(\{0\}) \times \R: \widetilde X_p=0 \}
\end{align*}
is not clean unless $0$ is a regular value of $\J$, the stationary phase principle cannot be applied.Instead, a complete asymptotic expansion for integrals of the form \eqref{int}  was derived in \cite[Theorem 1.1]{kuester-ramacher20} via a process called \emph{destratification},  the coefficients in the asymptotics being given by integrals over the strata of the singular symplectic quotient $\M_0$. From this,  the existence of an expansion of \eqref{eq:mainint}, and also some rough properties of its coefficients, can be inferred. However, within that more general framework the equivariant-cohomological interpretation of the coefficients in the obtained expansion remains elusive. Therefore, we shall not use the results derived there but proceed  from scratch,  based on the fact that the amplitude in the Witten integral is an equivariantly closed form, which allows a simpler and concise treatment.

\subsection{Preliminaries}\label{prel]}

 We begin now with the study of the asymptotics of the Witten integral 
\ref{def:mainint}. Conceptually, we will follow
the method of Meinrenken in \cite[Proof of Theorem\ 3.3]{meinrenken96}, the crucial 
idea being a retraction onto the zero level set of the moment map and the use of the
equivariant homotopy Lemma \ref{eqHom} and equivariant Stokes' Lemma \ref{eqStokes}, 
combined with the classical stationary phase principle. But since we do not assume
zero to be a regular value of the moment map, the situation is much more involved.
In fact, we will combine a retraction onto the regular stratum of $\J^{-1}(\mklm{0})$ with 
retractions onto the several components of the singular stratum of $\J^{-1}(\mklm{0})$. 

As a first step, we choose -- once and for all -- for each $F\in \F$ open sets $U_F\subset M$ and $V_F\subset \Sigma_F$  as in Proposition \ref{prop:localnormform}. For technical purposes, we choose them small enough such that the symplectic form \eqref{sympSigmaF} is actually non-degenerate on a slightly larger tubular neighborhood $\widetilde V_F\subset\Sigma_F$ of the zero section of $\Sigma_F$ containing $\overline V_F$, so that Proposition \ref{prop:localnormform} applies to a corresponding neighborhood $\widetilde U_F\subset M$ containing $\overline U_F$ with a symplectomorphism $\widetilde \Phi_F:\widetilde V_F\to\widetilde U_F$ extending $\Phi_F$, while keeping the sets $\widetilde U_F$ obtained for the different $F$ disjoint. This setup will be kept fixed in all of the following. 

\begin{lem}\label{lem:globalretraction} 
There exists $\delta>0$ and an $S^1$-invariant open neighborhood $W\subset M\setminus M^{S^1}$ of $\J^{-1}(\{0\})^{\mathrm{reg}}$
together with a retraction $\mathrm{ret}_0:W\to \J^{-1}(\{0\})^{\mathrm{reg}}$
and an $S^1$-equivariant diffeomorphism 
\begin{equation}\label{PhiW}
\begin{split}
\Phi_W:W&\xrightarrow{\,\sim\,}\J^{-1}(\{0\})^{\mathrm{reg}}\times(-\delta,\delta),\qquad p\longmapsto(\mathrm{ret}_0(p),\J(p)),
\end{split}
\end{equation}
such that, in case $0\in\Int\,J(M)$,
 we have $\Phi_F^{-1}(U_F\cap W)\subset \Sigma_{F\bullet\bullet}$ for all $F\in\F^0$ and
\begin{equation}\label{retVF}
\Phi_F^{-1}\circ( \mathrm{ret}_{0}|_{U_F\cap W} )
=\mathrm{ret}_{Z_F}\circ\Phi_F^{-1}|_{U_F\cap W},
\end{equation}
while $W\cap \widetilde{U}_F=\emptyset$ in case $0\in\partial\,\J(M)$.
\end{lem}
\begin{proof}
Let $g^{TM}$ be an $S^1$-invariant Riemannian metric on $M$ satisfying
$\Phi_F^*(g^{TM}|_{\widetilde{U}_F})=g^{T\Sigma_F}$ in the local normal form coordinates of Proposition \ref{prop:localnormform} for all $F\in\F$. Write
$\mathrm{grad}\,\J\in\Cinft(M,TM)$ for the associated
Riemannian gradient of $\J:M\to\R$,
and let $\xi\in\Cinft(M\backslash M^{S^1},TM)$ be the vector field
\begin{equation}\label{xi}
\xi:=\frac{\mathrm{grad}\,\J}{\|\mathrm{grad}\,\J\|_{g^{TM}}^2},
\end{equation}
where we use that the Hamiltonian $S^1$-action is locally free on $M\backslash M^{S^1}$,
so that $\mathrm{grad}\,\J$ is nowhere vanishing on $M\backslash M^{S^1}$. In case $0\in\Int\J(M)$,
the vector field
\eqref{xi} is transverse to $\J^{-1}(\{0\})^{\mathrm{reg}}$,
satisfies $d\J\imult \xi=1$ over $M\backslash M^{S^1}$, and in view of 
\eqref{eq:rscoord}  is mapped to the vector field $\frac{\partial}{\partial q}$ 
by the diffeomorphism of Proposition \ref{prop:localnormform} over $\widetilde U_F\subset M$  
for each $F\in\F^0$.  From this explicit description near each $F\in\F^0$ and the fact that the set 
$K:=\J^{-1}(\{0\})^{\mathrm{reg}}\cap(M\setminus\bigcup_{F\in \F^0} U_F)$ is compact, 
it follows that there is a $\delta>0$ such that the flow $\Phi_t^{\xi}$ of $\xi$ is defined on  $\J^{-1}(\{0\})^{\mathrm{reg}}$ for all $t\in (-\delta,\delta)$. Consequently,  the smooth map
\begin{equation}\label{Phixi}
\begin{split}
\Phi^{\xi}:\J^{-1}(\{0\})^{\mathrm{reg}}\times(-\delta,\delta)&
\longrightarrow M\\
(p,t)&\longmapsto\Phi_t^{\xi}(p)\,,
\end{split}
\end{equation}
is a diffeomorphism onto its image satisfying $\J(\Phi^{\xi}(p,t))=t$ for all $p\in \J^{-1}(\{0\})^{\mathrm{reg}}$ and $t\in(-\delta,\delta)$. 
Setting $W:=\mathrm{im}(\Phi^{\xi})\subset M$, we  define the diffeomorphism \eqref{PhiW} to be the inverse of
\eqref{Phixi}. Recalling the definition of $\mathrm{ret}_{Z_F}$ in \eqref{eq:retZF}, all claimed properties except the last one are satisfied by construction. Finally, in the case $0\in\partial\,\J(M)$, then $K$ is disjoint from $\bigcup_{F\in \F^0} \widetilde U_F$ simply  because $\J^{-1}(\{0\})^{\mathrm{reg}}$ is disjoint from $\widetilde U_F$. Hence in that case, it suffices to take $\delta$ small enough to get that $W\cap \widetilde{U}_F=\emptyset$ for all $F\in \F^0$. 
\end{proof}

The following simple consequence of the equivariant Stokes' Lemma \ref{eqStokes}
and the equivariant homotopy Lemma \ref{eqHom} will allow us to greatly
simplify the computations of the next section.

\begin{lem}\label{assumption} The Witten integral of Definition \ref{def:mainint}
only depends
on the equivariant cohomology class of $\varrho\in\Omega_{S^1}^\ast(M)$ inside
$H_{S^1}^*(M)$. Moreover, each equivariant cohomology class in
$H_{S^1}^*(M)$ has a representative $\varrho\in \Omega_{S^1}^\ast(M)$ satisfying
\begin{equation}\label{eq:12.5.2021}
\Phi_{F}^\ast(\varrho|_{U_F})
= (\nu_{\Sigma_F}^\ast\mathrm{inc}_F^\ast \varrho)|_{V_F}\quad\text{for all}\;F\in
\F^0
\end{equation}
and 
\begin{equation}\label{retractform}
\varrho|_W=\ret_{0}^*\inc_0^*\varrho,
\end{equation}
 where $\ret_{0}:W\to \J^{-1}(\{0\})^{\mathrm{reg}}$ is as in Lemma \ref{lem:globalretraction} with a suitable $\delta>0$ and  $\textup{inc}_0:\J^{-1}(\{0\})^{\mathrm{reg}}\to M$ denotes the inclusion map.
\end{lem}
\begin{proof} The first claim is standard. Indeed, let $\varrho\in\Omega_{S^1}^\ast(M)$ be equivariantly closed. In view of the fact that $ e^{2\pi im\J(x)} e^{m \omega}\in\Omega_{S^1}^\ast(M)$ is equivariantly closed, Lemma \ref{eqStokes} implies  for any equivariant form
$\gamma \in  \Omega_{S^1}^\ast(M)$ that 
\begin{equation*}\label{Witteninvcoh}
\begin{split}
\<\W_m(\varrho+d_{\g}\gamma),\phi\>&=
\int_{\R} \left[\int_{M} e^{2\pi im x \J} e^{m \omega} \varrho(x)
+\int_{M} e^{2\pi imx\J} e^{m \omega} d_{\g}\gamma(x) \right] 
\phi(x) \d x\\
&=\int_{\R} \left[\int_{M} e^{2\pi imx\J} e^{m \omega} \varrho(x) 
+\int_{M} d_{\g}\big(e^{2\pi im x \J} e^{m \omega} \gamma(x)\big) \right] 
\phi(x) \d x\\
&=\int_{\R} \left[\int_{M} e^{2\pi imx\J} e^{m \omega} \varrho(x) \right] 
\phi(x) \d x=\<\W_m(\varrho),\phi\>,
\end{split}
\end{equation*}
proving the first claim of the lemma.

Let us consider now an arbitrary equivariantly closed
$\widetilde\varrho\in\Omega_{S^1}^\ast(M)$. To prove the remaining assertions
of the lemma,
we need to construct an equivariantly closed $\varrho\in\Omega_{S^1}^\ast(M)$
in the cohomology class $[\widetilde \varrho]\in H_{S^1}^*(M)$ satisfying
\eqref{eq:12.5.2021} and \eqref{retractform}.  Using Lemma \ref{eqHom}  with 
$N=F$ and $U_N=\widetilde U_F$  we get for each $F\in \F^0$  
an equivariant form $\gamma_F\in\Omega_{S^1}^\ast(\widetilde V_F)$ such that
\begin{equation}\label{eq:1.1.2023}
\widetilde \Phi_{F}^\ast(\widetilde\varrho|_{\widetilde U_F})=(\nu_{\Sigma_F}^\ast\inc_F^*\widetilde\varrho)|_{\widetilde V_F}+d_{\g}\gamma_F\,.
\end{equation}
Let $\chi_F\in\Cinft(M)$ have compact support in $\widetilde U_F$ and be identically equal to $1$ on $U_F$. Setting
\begin{equation*}
 \gamma:=\sum_{F\in \F^0}\chi_F (\widetilde\Phi_F^{-1})^\ast\gamma_F\in\Omega_{S^1}^\ast(M)\,,   
\end{equation*}
the equivariantly closed form
$\widehat \varrho:=\widetilde\varrho-d_{\g}\gamma\in\Omega_{S^1}^\ast(M)$
satisfies $\widetilde\Phi_{F}^\ast( \widehat\varrho|_{U_F})
= (\nu_{\Sigma_F}^\ast\inc_F^\ast \widetilde \varrho)|_{V_F}$. But
$\nu_{\Sigma_F}\circ \widetilde \Phi_F^{-1}\circ \mathrm{inc}_F=\mathrm{inc}_F$,
so that \eqref{eq:1.1.2023} implies for each $F\in \F^0$
$$
\mathrm{inc}_F^\ast\, d_{\g}\gamma=\mathrm{inc}_F^*\widetilde\varrho-\mathrm{inc}_F^\ast(\widetilde \Phi_F^{-1})^\ast((\nu_{\Sigma_F}^\ast\mathrm{inc}_F^*\widetilde\varrho)|_{\widetilde V_F})=0,
$$
yielding  $\mathrm{inc}_F^*\widehat\varrho=\mathrm{inc}_F^*\widetilde\varrho$ and consequently 
\begin{equation}\label{hatvarrho}
\widetilde\Phi_{F}^\ast(\widehat\varrho|_{U_F})=(\nu_{\Sigma_F}^\ast\inc_F^*\widehat\varrho)|_{V_F}\,. 
\end{equation}

Next, choose $\delta>0$ in Lemma \ref{lem:globalretraction} such that it applies also with a slightly larger $\widehat \delta>\delta$ and gives us two corresponding retractions $\widehat{\ret}_{0}:\widehat W\to \J^{-1}(\{0\})^{\mathrm{reg}}$ and $\ret_{0}:W\to \J^{-1}(\{0\})^{\mathrm{reg}}$ from open subsets $\widehat W,W\subset M$ onto $J^{-1}(\{0\})^{\mathrm{reg}}$ such that $\widehat{\ret}_{0}|_{W}=\ret_{0}$. Let $\overline{W}$ be the closure of $W$ in $M$. Then the compact set $\overline{W}\setminus \bigcup_{F\in \F^0}U_F$ lies in $M\setminus M^{S^1}$ and its intersection with $J^{-1}(\{0\})^{\mathrm{reg}}$ is compact. By making $\delta$, and hence $W$, smaller while keeping $\widehat \delta$ fixed, we can therefore achieve that $\widehat W$ contains $\overline{W}\setminus \bigcup_{F\in \F^0}U_F$.   Using Lemma \ref{eqHom} again, now with
$N=\J^{-1}(\{0\})^{\mathrm{reg}}$ and $U_N=\widehat W$, 
there is a  form $\sigma\in\Omega_{S^1}^\ast(\widehat W)$ such that
\begin{equation}\label{retractformeq1}
\widehat\varrho|_{\widehat W}=\widehat \ret_{0}^*\inc_{0}^*\widehat\varrho+d_\g\sigma\,.
\end{equation}
On the other hand, in case $0\in\Int\J(M)$, for any $F\in \F^0$ one has $\nu_{\Sigma_F}\circ\inc_{Z_F}\circ\ret_{Z_F}=\nu_{\Sigma_F|\Sigma_{F\bullet \bullet}}$ so that  from \eqref{hatvarrho} we deduce that $\widetilde\Phi_{F}^\ast(\widehat\varrho|_{\widehat W \cap U_F})
=\ret_{Z_F}^*\inc_{Z_F}^*\widetilde\Phi_{F}^\ast(\widehat\varrho|_{\widehat W \cap U_F})$  for such an $F$. Thanks to Lemma \ref{lem:globalretraction} we therefore already know that
\begin{equation*}
\widehat\varrho|_{\widehat W\cap U_F}=(\widehat \ret_{0}|_{\widehat W\cap U_F})^*\inc_{0}^*(\widehat\varrho|_{\widehat W \cap U_F})\,.
\end{equation*}
Invoking the second part of the equivariant homotopy Lemma \ref{eqHom} and recalling that $\widehat W$ is disjoint from $\widetilde U_F$ for each $F\in \F^0$ in case $0\in\partial\,\J(M)$, we can
thus choose $\sigma$ in Equation \eqref{retractformeq1}
such that $\sigma|_{\widehat W\cap U_F}\equiv 0$ for all $F\in \F^0$. Let $\chi\in\CT(\widehat W)$ be identically equal to $1$ on $\overline W\setminus \bigcup_{F\in \F^0}U_F$. Then we can extend $\sigma$ by $0$ by setting  $\widehat\sigma:=\chi \sigma \in\Omega_{S^1}^\ast(M)$, and 
$\varrho:=\widehat \varrho-d_{\g}\widehat\sigma \in\Omega_{S^1}^\ast(M)$ 
satisfies the properties \eqref{eq:12.5.2021} and \eqref{retractform}. 
Since it differs from $\widehat\varrho\in\Omega_{S^1}^\ast(M)$,  and hence from $\widetilde\varrho\in\Omega_{S^1}^\ast(M)$, by an equivariantly exact form, $[\varrho]=[\widetilde\varrho]\in H_{S^1}^*(M)$,
which concludes the proof. 
\end{proof}

Conceptually, this lemma implies that when computing the Witten integral we can make the following assumption, which will lead to substantial simplifications later.

\begin{assumption}\label{ass:rho} The equivariantly closed
form $\varrho \in \Omega_{S^1}^\ast(M)$ in the Witten integral \eqref{eq:mainint} satisfies the conditions \eqref{eq:12.5.2021} and \eqref{retractform}. 
\end{assumption}

In order to obtain a meaningful geometric interpretation for
the asymptotic expansion of the Witten integral, we will also make the following
assumption.

\begin{assumption}\label{ass:phi} The test function
$\phi\in \CT(\R)$ in Definition \ref{def:mainint} is identically
equal to $1$ in a neighborhood of $0$.
\end{assumption}

Next, we choose a $\delta>0$ as in Lemma \ref{assumption} and 
a cutoff function $\tau\in \CT(\R)$
with $\textup{supp }\tau\subset(-\delta,\delta)$ such that $\tau\equiv 1$
on $[-\delta/2,\delta/2]$.
Since $e^{m\omega}$ is a polynomial in $m\in\N$,
the non-stationary phase principle implies that as $m\to+\infty$ we have for any equivariantly closed $\varrho$ 
\begin{equation}
\<\W_m(\varrho),\phi\>=\int_{\R} \left[\int_{M} e^{2\pi imx\J} e^{m \omega}
 \varrho(x)~\tau\circ\J \right]  \phi(x) \d x+\O(m^{-\infty})\,,\label{eq:Wittenminfty}
\end{equation}
so that the integral localizes around the zero level set $\J^{-1}(\{0\})\subset M$. 
Now, since $M$ is compact, we can take $\delta$ small enough such that with the coordinates defined by the diffeomorphisms \eqref{eq:PsiSF}
and \eqref{eq:Psiindef}  introduced in Section \ref{sec:applicationtonormal}, we can define for every  $F\in\F^0$ cutoff functions $\chi_F\in\CT(U_F)$ and $\tau_F\in\CT(V_F)$  by the characterizing relations
\begin{align}\label{chiFdef}
 \nonumber (\tau_F\circ\Psi_{F})(p,r)&=\tau(r^2/2), && (p,r)\in S_F\times(0,\infty), &&\text{if $0\in\partial\,\J(M)$,}\\(\tau_F\circ\Psi_{F})(p,r,q)&=\tau(r)\tau(q), && (p,r,q)\in S_F\times\R\times(0,\infty),& &\text{if $0\in\Int\J(M)$,}\\
\nonumber \chi_F\circ\Phi_F&=\tau_F.
\end{align}
We then decompose the integral in \eqref{eq:Wittenminfty} further as
\begin{align*}
\begin{split}
  \<\W_m(\varrho),\phi\>&=
  \underbrace{\int_{\R} \left[\int_{M}
  e^{2\pi imx\J} e^{m \omega}  \varrho(x)~
   \left(\tau\circ\J-\sum_{F\in \F^0}\chi_F\right)
   \right]  \phi(x) \d x}_{=:\<\W^{\text{reg}}_m(\varrho),\phi\>}\\
   &\qquad+\sum_{F\in \F^0}\underbrace{\int_{\R}
  \left[\int_{M} e^{2\pi imx\J} e^{m \omega} \varrho(x)
  ~\chi_F\right] 
  \phi(x) \d x}_{=:\<\W^{F}_m( \varrho),\phi\>}\;+\;\O(m^{-\infty}).
\end{split}
\end{align*}
We then obtain 
\begin{align}
\label{eq:25.5.2019}
\begin{split}
  \<\W_m(\varrho),\phi\>&= \<\W^{\text{reg}}_m(\varrho),\phi\>+\sum_{F\in \F^0}\<\W^{F}_m(\varrho),\phi\>+ \O(m^{-\infty})\,,
\end{split}
\end{align}
where  by pullback along $\Phi_F$ and thanks to Assumption \ref{ass:rho} we have
\bq
\label{WF}
\<\W^{F}_m(\varrho),\phi\>= \int_{\R} \left[\int_{\Sigma_F}
e^{2\pi im\frac{x}{2}Q_F}  e^{m\omega_{\Sigma_F}} \nu_{\Sigma_F}^*\varrho_{F}(x)\,
 \tau_F \right]  \phi(x) \d x.
\eq

\subsection{Definite fixed point set components}\label{sec:defcase}

Let us begin with the simpler case 
when $0\in\partial\,\J(M)$, in which case $Q_F$ is definite and
either $F\in \F^0_+$ or $F\in \F^0_-$. Recall also that  we have
$\J^{-1}(\{0\})\cap U_F=F$, which is fixed by the action of $S^1$. 
Pulling back the inner integral in \eqref{WF} along the diffeomorphism
$\Psi_{F}$ introduced in \eqref{eq:PsiSF}, taking into account Diagram \eqref{diagram_retract_EFdefinite} and the fact that $\Sigma_{F\bullet}$ has full measure in $\Sigma_F$, and using \eqref{chiFdef} and \eqref{eq:rQ_Fdefinite} we get
\begin{align*}
\<\W^{F}_m(\varrho),\phi\>&= \int_{\R} \left[\int_{S_F\times(0,\infty)}   \tau(r^2/2) e^{\pm 2\pi im\frac{x}{2} r^2 }
e^{m \Psi_{F}^\ast (\omega_{\Sigma_F}|_{\Sigma_{F\bullet}})}   \nu_{S_F\times(0,\infty)}^\ast  \varrho_{F}(x)  \right]  \phi(x) \d x\quad\text{if }F\in\F_\pm. 
\end{align*}
Corollary \ref{cor:PsiSFastdefinite} implies that
\bq\label{defcase1}
\Psi_{F}^\ast (\omega_{\Sigma_F}|_{\Sigma_{F\bullet}}) =  \nu_{S_F\times(0,\infty)}^\ast \omega_{F} \pm r\d r \wedge\mathrm{pr}_{S_F}^\ast \Theta_{S_F}  \pm \frac{r^2}{2}\, \mathrm{pr}_{S_F}^\ast d\Theta_{S_F}\quad\text{if }F\in\F_\pm\,,
\eq
so that expanding the exponential series and taking into account the fact that the first and last terms 
of the right-hand side of \eqref{defcase1} have no $dr$ part as they are pulled
back from $F$ and $S_F$ respectively, we get
\begin{align}
\begin{split}
 \<\W^{F}_m(\varrho)&,\phi\>=
\sum_{k=0}^{\infty}\frac{(\pm m)^{k+1}}{2^k}\int_{\R}\bigg[
\int_{S_F\times(0,\infty)}  \tau(r^2/2) e^{\pm \pi i m x  r^2}
e^{m\nu_{S_F\times(0,\infty)}^\ast\omega_{F}}\\
 &\qquad \qquad \qquad \qquad \qquad \qquad \qquad   \nu_{S_F\times(0,\infty)}^\ast  \varrho_{F}(x)  \frac{d\Theta_{S_F}^k}{k!}\d r\wedge\Theta_{S_F} r^{2k+1} \bigg]  \phi(x)     \d x\\
&=\sum_{k=0}^{\infty}\left(\frac{\pm 1}{2\pi}\right)^{k+1}\int_{\R}
\int_0^\infty e^{\pm i x s } s^k  \tau(\pm s/(2\pi m)) \left[\int_{S_F}  e^{m \nu_{S_F}^\ast\omega_{F}}   \nu_{S_F}^\ast  \varrho_{F}(x) \frac{d\Theta_{S_F}^k}{k!}  \Theta_{S_F}\right]\phi(x) \d s \d x.\label{W0mdefcomput}
\end{split}
\end{align}
Here  we performed the change of variable $s=m \pi r^2$ in the last line. Note also
that the infinite sum in \eqref{W0mdefcomput} has only finitely many non-zero summands because $d\Theta_{S_F}^k=0$ for $k>\ell_F$. 
 
Recall now that the \emph{Heaviside function} $H:\R\to\R$ is defined by $H(s)=1$
for all $s\geq 0$ and $H(s)=0$ otherwise. Using a standard formula for its Fourier
transform as a tempered distribution \cite[Example \ 7.1.17]{hoermanderI}
and the fact that $\tau\in\CT(\R)$ is identically equal to
$1$ on $(-\delta/2,\delta/2)$, one gets for any $k\in\N$ and any Schwartz function
$\psi\in\Scal(\R)$ as $m\to \infty$ the equality
\begin{align}
\nonumber \int_0^{\infty} \int_{\R}  e^{\pm i sx}
  s^k  \tau(\pm s/(2\pi m))\psi(x)\d x\d s
\nonumber &=\int_0^{\infty} s^k\, \hat\psi(\mp s)\,\d s
+\int_0^{\infty} s^k\,(\tau(\pm s/(2\pi m))-1)
\hat\psi(\mp s)\d s\\
\nonumber &=(\mp 1)^{k}\eklm{q^k H(\mp q), \hat \psi}+\O\bigg (\int_{m\pi\delta}^\infty  s^k|\hat\psi(\mp s)| \d s \bigg )\\
\label{pvfla}&=(\pm i)^{k+1}k! \eklm{{\underline{x}^{-(k+1)}_\pm}, \psi}+\O(m^{-\infty} ),
\end{align}
where we used that $\hat\psi$ is rapidly decreasing and introduced the distribution
$\underline{x}^{-k}_\pm$ defined by
\begin{equation}\label{pvdef}
\underline{x}^{-k}_\pm:=\lim_{\epsilon\to 0^+}(x\pm i\epsilon)^{-k}\,, 
\end{equation}
which satisfies the standard relation
$\frac d {dx} \underline{x}^{-k}_\pm=-k \underline{x}^{-k-1}_\pm$ for all $k\in\N$
as distributions. 
More generally, for any absolutely converging
Laurent series $P(x)=\sum_{j=-N}^{+\infty} a_j\,x^j$,
we will write $P(\underline x_\pm)$ for the distribution defined by
\begin{equation}\label{Laurentpvdef}
   \eklm{P(\underline x_\pm),\psi}:=
   \sum_{j=-N}^{-1} a_j\eklm{\underline x^j_\pm,\psi}
   +\eklm{\sum_{j=0}^{\infty} a_jx^j,\psi}\,.
\end{equation}
We thus get from \eqref{W0mdefcomput} that as $m\to\infty$
\begin{equation}\label{W0mdefcomputfinal}
\begin{split}
\<\W^{F}_m(\varrho),\phi\>&=\bigg\<\int_{S_{F}}\,\nu_{S_F}^\ast\varrho_{F}(x) 
\,e^{m\nu_{S_F}^\ast\omega_{F}}
\sum_{k=0}^{\infty}
\frac{(\frac{i}{2\pi}d\Theta_{S_F})^k}{\underline x^{k+1}_\pm}\frac{i}{2\pi}\Theta_{S_F},\phi\bigg\>
+\O(m^{-\infty})\\
&=\bigg\langle\int_{S_{F}}\,e^{m\nu_{S_F}^\ast\omega_{F}}\nu_{S_F}^\ast
\varrho_{F}(x) 
\frac{\frac{i}{2\pi}\Theta_{S_F}}{\underline{x}_\pm-\frac{i}{2\pi}d\Theta_{S_F}},\phi\bigg\rangle+\O(m^{-\infty})\quad\text{if }F\in\F_\pm\,.
\end{split}
\end{equation}

\begin{rem}\label{rem:DH1}
Adapting the proof of Duistermaat and Heckman in
\cite[(2.11)-(2.31)]{duistermaat-heckman83} 
one gets from the first line of \eqref{W0mdefcomputfinal} that
\begin{equation}\label{W0defrmk0}
\<\W^{F}_m(\varrho),\phi\>=\bigg\<\int_F \frac{e^{ m \omega_F}\varrho_{F}(\underline x_\pm)}{e_F(\underline x_\pm)},\phi\bigg\>
+\O(m^{-\infty})\quad\text{if }F\in\F_\pm\,,
\end{equation}
where $e_F(x)\in\Omega^*_T(F)$ is the \emph{equivariant Euler class}
of the normal bundle $\nu_{\Sigma_F}:\Sigma_F\to F$. 
Note that the integrand coincides with the integrand
appearing in Berline-Vergne localization
formula \cite{berline-vergne82}, and one can actually see that
our method does provide an alternative proof of it.
In the application to the computation
of the Riemann-Roch numbers in Section \ref{sec:asyRR}, 
we will not use Formula \eqref{W0defrmk0} but instead provide a direct proof based on \eqref{W0mdefcomputfinal} and the Kirillov formula
of Theorem \ref{Kirfla}.
\end{rem}

\subsection{Indefinite fixed point set components}  \label{sec:indefinitefixedpoints}

Let now deal with the case when we have $0\in\Int\J(M)$,
so that $Q_F$ is indefinite for all $F\in\F^0$ and
we have $\ell_F^+>0$ and
$\ell_F^->0$.
Using the diffeomorphism 
$\Psi_{F}$ introduced in \eqref{eq:Psiindef} and the coordinates
\eqref{eq:rscoord}, the
integral \eqref{WF} becomes
\bq
\<\W^{F}_m(\varrho),\phi\>= \int_{\R} \left[\int_{S_F\times(0,\infty)\times \R}   \tau(r)\tau(q) e^{2\pi i mx q}   e^{m \Psi_{F}^\ast (\omega_{\Sigma_F}|_{\Sigma_{F\bullet\bullet}})}  \nu_{S_F\times(0,\infty)\times \R}^\ast  \varrho_{F}(x)  \right]  \phi(x) \d x. \label{eq:13.5.2021}
\eq
Now, by Corollary \ref{cor:PamplonaF} we know that
\bq
\Psi_{F}^\ast (\omega_{\Sigma_F}|_{\Sigma_{F\bullet\bullet}})=  \pi_{Z_F}^\ast \mathrm{inc}_{Z_F}^\ast\omega_{\Sigma_F}+ d\big (q\, \mathrm{pr}_{S_F}^\ast  \Theta_{S_F}\big )+ d_\g \bar\beta_F\label{eq:01.07.2022}
\eq
where
\bqn
\bar\beta_F:=( \sqrt{r^4+q^2} -r^2) \mathrm{pr}_{S_F}^\ast \overline \Theta_{S_F} \in\Omega^1(S_F\times(0,\infty)\times \R)_{\textup{bas}}^{S^1}\label{eq:betabar}
\eqn
is a basic form in the sense of \eqref{basdef}, so that in particular  $d \bar\beta_F =d_\g\bar\beta_F$ by Definition \ref{OmG}. 
Since 
\bq
d_\g\bar\beta_F=d \bar\beta_F= ( \sqrt{r^4+q^2} -r^2) \mathrm{pr}_{S_F}^\ast d\overline \Theta_{S_F} + \Big( \frac{2r^3 \d r + q\d q}{\sqrt{r^4+q^2}} -2r\d r\Big)\wedge \mathrm{pr}_{S_F}^\ast \overline \Theta_{S_F}, \label{eq:dbarbeta}
\eq
both $\bar \beta_F$ and  $d_\g\bar\beta_F$ have continuous extensions to the manifold with boundary $S_F \times [0,\infty) \times \R$. On the other hand, the summand $d\big (q\, \mathrm{pr}_{S_F}^\ast  \Theta_{S_F}\big )$ is constant with respect to $r$, and also the summand $\pi_{Z_F}^\ast \mathrm{inc}_{Z_F}^\ast\omega_{\Sigma_F}$ extends continuously to $r=0$ in view of Corollary \ref{cor:PamplonaF}. 
 Inserting \eqref{eq:01.07.2022}  into \eqref{eq:13.5.2021} therefore yields 
\bq
\label{I1I2I3I4}
\<\W^{F}_m(\varrho),\phi\>=I_1^F(m) +I_2^F(m),
\eq
where the integrals
\begin{equation*}
\begin{split}
I_1^F(m):=& \int_{\R} \left[\int_{S_F\times(0,\infty)\times \R}   \tau(r)\tau(q) e^{2\pi i  m x q  }   e^{m\pi_{Z_F}^\ast \mathrm{inc}_{Z_F}^\ast\omega_{\Sigma_F} +md (q\, \mathrm{pr}_{S_F}^\ast  \Theta_{S_F}) } \nu_{S_F\times(0,\infty)\times \R}^\ast  \varrho_{F}(x) \right]  \phi(x) \d x,\\
   I_2^F(m):=& \int_{\R} \bigg[\int_{S_F\times(0,\infty)\times \R}   \tau(r)\tau(q) e^{2\pi i  m x q } e^{m\pi_{Z_F}^\ast \mathrm{inc}_{Z_F}^\ast\omega_{\Sigma_F} +md (q\, \mathrm{pr}_{S_F}^\ast  \Theta_{S_F} ) } \nu_{S_F\times(0,\infty)\times \R}^\ast   \varrho_{F}(x)\\
   &\qquad\qquad\qquad\qquad \Big ( e^{m d_\g\bar\beta_F } -1\Big )  \bigg]  \phi(x) \d x 
\end{split}
\end{equation*}
are absolutely convergent.

\subsubsection{The integral $I_1^F(m)$}
\label{sec:I1}
Let us first turn to the integral $I_1^F(m)$.
Here we closely follow the method of \cite[Proof of Theorem\ 3.3]{meinrenken96}. Isolating the $dq$ part in
\eqref{eq:01.07.2022} and expanding the corresponding exponential series we get
\begin{align*}
\nonumber &I_1^F(m)=\sum_{k=0}^{\infty}\frac{m^{k+1}}{k!}\int_{\R}\bigg[\int_{S_F\times (0,\infty)\times\R} {q^k} e^{2\pi i  mx q }    \tau(r)\tau(q)   \nu_{S_F\times(0,\infty)\times \R}^\ast  \varrho_{F}(x)\\
\nonumber &\qquad\qquad\qquad\qquad\qquad\qquad\qquad e^{m \pi_{Z_F}^\ast \mathrm{inc}_{Z_F}^\ast\omega_{\Sigma_F}}\d q\,d\Theta_{S_F }^k \Theta_{S_F } \bigg]  \phi(x) \d x\\
&=\sum_{k=0}^{\infty}\frac{m^{k+1}}{k!}\int_{S_F\times (0,\infty)}
\bigg[  \underbrace{\int_{\R^2} q^k e^{2\pi i  mx q }     \nu_{S_F\times(0,\infty)}^\ast  \varrho_{F}(x)  \tau(q)\phi(x)  \d q\d x}_{=:\tilde I(m)}\bigg]
\tau(r)e^{m \Psi_{F}^\ast \mathrm{inc}_{Z_F}^\ast\omega_{\Sigma_F}}d\Theta_{S_F }^k
\Theta_{S_F}.\label{I1Mcomput0}
\end{align*}
Here we took Diagram \eqref{eq:diagram1} into account. 
Note that, as before, the infinite sums occurring here are actually finite because $d\Theta_{S_F }^k=0$ for $k\in\N$ large enough. Moreover, the integrand of $\tilde I(m)$ depends on the external parameters $(p,r)\in S_F\times (0,\infty)$ only via the point $\nu_{S_F\times(0,\infty)}(p,r)=\nu_{S_F}(p)\in F$, so that
$\tilde I(m)$ is constant with respect to $r\in (0,\infty)$.

Applying the classical stationary phase principle \cite[Lemma 7.7.3]{hoermanderI} to $\tilde I(m)$ while taking into account Assumption \ref{ass:phi}, we get 
 \bqn
\tilde I(m)  =  \frac{i^k}{(2\pi)^k m^{k+1}} \frac{\partial^{k}\nu_{S_F\times(0,\infty)}^\ast  \varrho_{F}}{\partial x^k}(0) + \O(m^{-\infty}), \qquad m\to+\infty,
\eqn
the estimate being uniform on  $S_F\times (0,\infty)$ because $S_F$ is compact and $\tilde I(m)$ is constant with respect to $r\in (0,\infty)$. This yields
\begin{equation}\label{eq:prelimexpindefI}
\begin{split}
I_1^F(m)&=\int_{S_F\times (0,\infty)}   \tau(r)   \,  e^{m \Psi_{F}^\ast \mathrm{inc}_{Z_F}^\ast\omega_{\Sigma_F}}  
\nu_{S_F\times(0,\infty)}^\ast  \varrho_{F}\Big(\frac{i}{2\pi} d \Theta_{S_F }\Big)
\, \Theta_{S_F } + \O(m^{-\infty})\\
&= \int_{Z_F}   \tau(r)   \,  e^{m \mathrm{inc}_{Z_F}^\ast  \omega_{\Sigma_F}  }  \nu_{Z_F}^\ast  \varrho_{F}\Big(\frac{i}{2\pi}d\Theta_{Z_F}\Big)\,\Theta_{Z_F} + \O(m^{-\infty})\\
&=\int_{\J^{-1}(\{0\})^\mathrm{reg}}\chi_F \,e^{m \, \textup{inc}_{\J^{-1}(\{0\})^\mathrm{reg}}^*\omega}
\textup{inc}_{\J^{-1}(\{0\})^\mathrm{reg}}^*\varrho\Big(\frac{i}{2\pi} d \Theta\Big)\,\Theta + \O(m^{-\infty}).
\end{split}
\end{equation}
Here, in the second line, we consider $r$ as a coordinate on $Z_F$ using the diffeomorphism $\Psi_{F}$ from \eqref{eq:PsiZF}. The third line is obtained by pullback along the diffeomorphism $\Phi_F^{-1}$, taking into account Assumption \ref{ass:rho}, \eqref{eq:diagram_withPhiF}, the last line in Equation \eqref{chiFdef}, and the fact that $\chi_F$ is supported in $U_F$.

\subsubsection{The integral $I_2^F(m)$} Let us now turn to the integral
$ I_2^F(m)$. First note from \eqref{eq:01.07.2022} that in the coordinates
\eqref{eq:rscoord} defined by the diffeomorphism \eqref{eq:Psiindef},
the equivariant form
\begin{equation}\label{eqclosed}
e^{2\pi im xq}e^{\pi_{Z_F}^\ast \mathrm{inc}_{Z_F}^\ast\omega_{\Sigma_F}+ d(q\, \mathrm{pr}_{S_F}^\ast  \Theta_{S_F})}=\Psi_F^*(e^{2\pi im\frac{x}{2}Q_F} e^{m\omega_{\Sigma_F}})e^{-d_\g\bar{\beta}_F}
\in\Omega^*_{S^1}(S_F\times(0,\infty)\times \R)
\end{equation}
is equivariantly closed.
Note also that
\bq\label{emdb-1}
e^{md_\g\bar\beta_F  } -1=\sum_{k=1}^{n} \frac 1 { k!} (m  d_\g \bar\beta_F  )^k =  d_\g \Big(\underbrace{m\sum_{k=1}^{n} \frac {m^k} { k!}  \bar\beta_F  (d_\g \bar\beta_F  )^{k-1}}_{=:\bar\beta_{m,F}}\Big)  
\eq
is an equivariantly exact form. By \eqref{eq:dbarbeta} one concludes that
\bq
\label{eq:07.07.2022}
\bar\beta_{m,F}= m\sum_{k=1}^n \frac {m^k}{k!}  (\sqrt{r^4+q^2}-r^2  )^k ( \mathrm{pr}_{S_F}^\ast  d \overline \Theta_{S_F } )^{k-1}  \mathrm{pr}_{S_F}^\ast   \overline \Theta_{S_F }
\eq
has a continuous extension to the manifold with boundary
$S_F\times [0,\infty)\times \R$. Analogously as in \eqref{eq:bepsdiffeo}, consider further for $\epsilon\geq 0$ the manifold with boundary $S_F \times [\epsilon,\infty)\times \R$, together with  the canonical parametrization of its boundary
given by the orientation preserving map
\begin{equation}\label{bFeps}
\begin{split}
b_{F,\epsilon}:S_F\times \R &\stackrel{\cong}{\longrightarrow} S_F \times \{\epsilon\}\times \R=\partial(S_F \times [\epsilon,\infty)\times \R),\\
(p,q) &\longmapsto (p,\epsilon,q)\,.
\end{split}
\end{equation}
Using Lebesgue's convergence theorem, Lemma \ref{eqStokes},  and the fact that the form \eqref{eqclosed}
is equivariantly closed we  get as $m\to+\infty$
\begin{align*}
I_2^F(m)&= \lim_{\epsilon \to 0} \int_{\R} \bigg[\int_{S_F \times [\epsilon,\infty)\times \R}    \tau(r)\tau(q) e^{2\pi i  m x q } e^{m\pi_{Z_F}^\ast \mathrm{inc}_{Z_F}^\ast\omega_{\Sigma_F} +md (q\, \mathrm{pr}_{S_F}^\ast  \Theta_{S_F} ) } \\ 
 &\qquad\qquad\qquad\qquad\qquad\d_\g \bar\beta_{m,F}\,  \nu_{S_F\times(0,\infty)\times \R}^\ast  \varrho_{F}(x) \bigg]  \phi(x) \d x \\
 &=\lim_{\epsilon \to 0} \int_{\R} \bigg[\int_{S_F \times [\epsilon,\infty)\times \R}   d_\g\bigg(\tau(r)
 \tau(q) e^{2\pi i  m x q } e^{m\pi_{Z_F}^\ast \mathrm{inc}_{Z_F}^\ast\omega_{\Sigma_F} +md (q\, \mathrm{pr}_{S_F}^\ast  \Theta_{S_F} ) } \\ 
 &\qquad\qquad\qquad\qquad\qquad\bar\beta_{m,F}\,  \nu_{S_F\times(0,\infty)\times \R}^\ast  \varrho_{F}(x)\bigg) \bigg]  \phi(x) \d x-R_2^F(m)\\
  &=\lim_{\epsilon \to 0} \int_{\R} \bigg[\int_{S_F \times \R}  b_{F,\epsilon}^\ast \Big (  \tau(r)\tau(q) e^{2\pi i  m x q } e^{m\pi_{Z_F}^\ast \mathrm{inc}_{Z_F}^\ast\omega_{\Sigma_F} +md (q\, \mathrm{pr}_{S_F}^\ast  \Theta_{S_F} ) }\\
  &\qquad\qquad\qquad\qquad\qquad\bar\beta_{m,F} \,  \nu_{S_F\times\{\eps\}\times \R}^\ast  \varrho_{F}(x)\Big )  \bigg]  \phi(x) \d x
  -R_2^F(m),
 \end{align*} where we wrote 
 \begin{multline}\label{R2F1}
R_2^F(m):=\int_\R\int_{S_F \times (0,\infty)\times \R}d(\tau(r)\tau(q)) e^{2\pi i  m x q } e^{m\pi_{Z_F}^\ast \mathrm{inc}_{Z_F}^\ast\omega_{\Sigma_F} +md (q\, \mathrm{pr}_{S_F}^\ast  \Theta_{S_F} ) }\\  \bar\beta_{m,F}\,  \nu_{S_F\times(0,\infty)\times \R}^\ast  \varrho_{F}(x) \phi(x) \d x. 
 \end{multline}
Now, analogously as in \eqref{eq:tildebFepsXiZF}, Corollary \ref{cor:PamplonaF} tells us that
\bq
b_{F,\epsilon}^*\pi_{Z_F}^\ast \mathrm{inc}_{Z_F}^\ast\omega_{\Sigma_F}=\eps^2\mathrm{pr}_{S_F}^\ast \mathrm{inc}_{S_F}^\ast\omega_{\mathrm{vert}}+\mathrm{pr}_{S_F}^\ast \nu_{S_F}^*\,\omega_F\label{eq:3829380690326}.
\eq
Taking into account  \eqref{eq:07.07.2022} and expanding the exponential series we therefore see that with 
\begin{multline*}
\tilde I_2^F(m):=  \int_{\R^2} e^{2\pi i  m x q}\sum_{k=1}^n \frac { |q|^km^{k+1}}{ k!}  \bigg[\int_{S_F}   e^{m \nu_{S_F}^\ast \omega_{F}  +m q d \Theta_{S_F } }    (  d \overline \Theta_{S_F } )^{k-1}\,\nu_{S_F}^\ast\varrho_{F}(x)\,
\Theta_{S_F }\,\overline \Theta_{S_F }  \bigg ] \\
\tau(q)\phi(x) \d x \d q
\end{multline*}
we obtain 
 \begin{equation}\label{bigcomputI2}
I_2^F(m) =\tilde I_2^F(m)-R_2^F(m).
\end{equation}
Separating the sum over $k$ into even and odd indices
and using the \emph{sign function} $\mathrm{sgn}:\R\to\R$ defined for all
$q\in\R$ by $\mathrm{sgn}(q):=H(q)-H(-q)$, we rewrite $\tilde I_2^F(m)$ as
\begin{align}
\nonumber \tilde I_2^F(m)&=m \int_{\R^2}
  e^{2\pi i  m x q } \tau(q)\int_{S_F}
  e^{m \nu_{S_F}^\ast  \omega_{F}}\,\nu_{S_F}^\ast \varrho_{F}(x )\, \phi(x)
  \,\Theta_{S_F }\,\overline \Theta_{S_F } \\
\nonumber   &\qquad\qquad e^{m q\, d\Theta_{S_F }}\bigg[
  \mathrm{sgn}(q)\sum_{\substack{k=1\\ k\text{ odd}}}^n\frac { q^{k}m^k}{ k!}
  d \overline \Theta_{S_F }^{k-1}+
  \sum_{\substack{k=1\\ k\text{ even}}}^n \frac {q^km^k }{k!}d \overline \Theta_{S_F }^{k-1}\bigg]
   \,\d q\,{\d x}\\
 \nonumber  &=m \int_{\R^2}  e^{2\pi i  m x q } \tau(q)\int_{S_F}
  e^{m \nu_{S_F}^\ast  \omega_{F}}\,\nu_{S_F}^\ast \varrho_{F}(x )\, \phi(x)
  \,\Theta_{S_F }\,\overline \Theta_{S_F }\\
  &\qquad\bigg[\mathrm{sgn}(q)\sum_{\substack{j,k=0\\ k\text{ odd}}}^n\frac { q^{j+k}m^{j+k}}{j!\,k!}d \Theta_{S_F }^j d \overline \Theta_{S_F }^{k-1}+\frac{e^{mq\,d\Theta_{S_F }^+}+e^{mq\,d\Theta_{S_F }^
  -}-2\,e^{mq\,d\Theta_{S_F }}}{2\,d\overline \Theta_{S_F }}\bigg]\,\d q{\d x}\,,\label{I2epsexp}
\end{align}
where we used the fact that, by the definition  of $\Theta_{S_F }$ and $\overline \Theta_{S_F }$ in \eqref{eq:connS_F}, one has
\begin{equation*}
\begin{split}
e^{m q\, d\Theta_{S_F }}\sum_{\substack{k=1\\ k\text{ even}}}^n \frac {q^km^k }{k!}d \overline \Theta_{S_F }^{k-1}
  &=\frac{e^{m q\,d \Theta_{S_F }}}{2}
  \bigg(\frac{e^{mq\,d \overline \Theta_{S_F }}-1}
  {d \overline \Theta_{S_F }}+\frac{e^{-mq\,d \overline \Theta_{S_F }}-1}
  {d \overline \Theta_{S_F }}\bigg)\\
  &=\frac{e^{mq\,d\Theta_{S_F }^+}+e^{mq\,d\Theta_{S_F }^
  -}-2\,e^{mq\,d\Theta_{S_F }}}{2\,d\overline \Theta_{S_F }}\,.
\end{split}
\end{equation*}
Expanding the exponential series and applying the classical stationary phase principle \cite[Lemma 7.7.3]{hoermanderI}   to the second term in the square bracket  of \eqref{I2epsexp} yields the asymptotic expansion
\begin{equation*}\label{evencomput}
\begin{split}
& m \int_{\R^2} e^{2\pi i  mxq}  \tau(q) \nu_{S_F}^\ast \varrho_{F}(x )\, \phi(x)
  \frac{e^{mq\,d\Theta_{S_F }^+}+e^{mq\,d\Theta_{S_F }^
  -}-2\,e^{mq\,d\Theta_{S_F }}}{2\,d\overline \Theta_{S_F }} \,\d q{\d x}\\
  &=\frac{\nu_{S_F}^\ast\varrho_{F}(\frac{i}{2\pi}d \Theta_{S_F }^+)
    +\nu_{S_F}^\ast\varrho_{F}(\frac{i}{2\pi}d \Theta_{S_F }^-)-2\,\nu_{S_F}^\ast\varrho_{F}(\frac{i}{2\pi}d \Theta_{S_F })}{2\,d\overline \Theta_{S_F}}
    +\O(m^{-\infty}),
\end{split}
\end{equation*}
where the remainder estimate is uniform in the base point in $S_F$ because the latter is compact. 
To deal with  the first term in the square bracket  of \eqref{I2epsexp} we note that  with the substitution $s=2\pi mq$ and  \eqref{pvfla} one computes
\begin{multline*}
\int_{\R^2}  e^{2\pi i  m x q } \tau(q) \,\nu_{S_F}^\ast \varrho_{F}(x )\, \phi(x) \mathrm{sgn}(q) \,  q^{j+k}  \,\d q{\d x}\\ 
=2\Big(\frac{i}{2\pi m}\Big)^{j+k+1}(j+k)! \eklm{{\underline{x}^{-(j+k+1)}},\nu_{S_F}^\ast \varrho_{F}(x )\, \phi(x)}+\O(m^{-\infty} ),
\end{multline*}
where we introduced the distributions
\begin{equation}\label{pvdef2}
\underline{x}^{-k}:=\frac 12 \big (\underline{x}_+^{-k}+ \underline{x}_-^{-k}\big ),\qquad k\in \N_0.
\end{equation}
 We will also use the notation \eqref{Laurentpvdef} with
the distributions \eqref{pvdef2}. Expanding $e^{m \nu_{S_F}^\ast  \omega_{F}}$, separating powers of $m$, and recalling Assumption \ref{ass:phi} we get
from \eqref{I2epsexp} the asymptotic expansion
\begin{equation}
\begin{split}
&\tilde I_2^F(m)= \int_{S_F}   e^{m \nu_{S_F}^\ast  \omega_{F}}
\frac{\nu_{S_F}^\ast\varrho_{F}(\frac{i}{2\pi}d \Theta_{S_F }^+)
    +\nu_{S_F}^\ast\varrho_{F}(\frac{i}{2\pi}d \Theta_{S_F }^-)-2\,\nu_{S_F}^\ast\varrho_{F}(\frac{i}{2\pi}d \Theta_{S_F })}{2\,d\overline \Theta_{S_F}}
\Theta_{S_F }\,\overline \Theta_{S_F } \\
  &+2    \sum_{\substack{j,k=0\\ k\text{ odd}}}^n\frac{(\frac{i}{2\pi})^{j+k+1}(j+k)!}{j!\,k!} 
  \eklm{ \underline x^{-(j+k+1)}, 
  \int_{S_F}   e^{ m \nu_{S_F}^\ast  \omega_{F}}\nu_{S_F}^\ast \varrho_{F}(x)d \Theta_{S_F }^j (d \overline \Theta_{S_F })^{k-1}\,
   \Theta_{S_F }\,\overline \Theta_{S_F }  \phi(x)}\label{eq:2385302689206}
\end{split}
\end{equation}
up to terms of order $\O(m^{-\infty})$ as $m \to \infty$. Let us now simplify \eqref{eq:2385302689206} with the help of the binomial theorem by writing for $x\neq 0$
\begin{align*}
   2\sum_{\substack{j,k=0\\ k\text{ odd}}}^n\frac{(j+k)!}{j!\,k!}
 & \frac{(\frac{i}{2\pi})^{j+k+1}}{x^{j+k+1}} d \Theta_{S_F }^j (d \overline \Theta_{S_F })^{k-1}
   \Theta_{S_F }\,\overline \Theta_{S_F } \\ &=\sum_{j,k=0}^n\frac{(j+k)!}{j!\,k!}
   \frac{d \Theta_{S_F }^j
   (d \overline \Theta_{S_F })^{k}-d \Theta_{S_F }^j
   (-d \overline \Theta_{S_F })^{k}}{(-2\pi ix)^{j+k+1}\,d \overline \Theta_{S_F }}
  \Theta_{S_F }\,\overline \Theta_{S_F } \\
     &=\sum_{r=0}^n\frac{(d \Theta_{S_F }+d \overline \Theta_{S_F })^{r}-
    (d \Theta_{S_F }-d \overline \Theta_{S_F })^{r}}{(-2\pi ix)^{r+1}\,d \overline \Theta_{S_F }}
    \Theta_{S_F }\,\overline \Theta_{S_F } \\
  &=\frac{\Theta_{S_F }\,\overline \Theta_{S_F } }{d \overline \Theta_{S_F }}\bigg(\frac{1}{-2\pi ix-(d \Theta_{S_F }+d \overline \Theta_{S_F })}-\frac{1}{-2\pi ix-(d \Theta_{S_F }-d \overline \Theta_{S_F })}\bigg)\\
   &=\frac{\frac{i}{2\pi}\Theta_{S_F }^+}{x-\frac{i}{2\pi}d\Theta_{S_F }^+}\frac{\frac{i}{2\pi}\Theta_{S_F }^-}{x-\frac{i}{2\pi}d\Theta_{S_F }^-},
\end{align*}
where in the last step we used the definitions  of $\Theta_{S_F }$ and $\overline \Theta_{S_F }$ given in  \eqref{eq:connS_F}. Using the exceptional Kirwan map introduced in \eqref{excKirdef} we finally arrive at
\begin{multline}\label{I2epsfinalpre}
\tilde I_2^F(m)=\int_{S_F}   e^{ m\nu_{S_F}^\ast  \omega_{F}}
\kappa_F(\varrho_{F})+\eklm{ \int_{S_F}e^{ m \nu_{S_F}^\ast \omega_{F}}\nu_{S_F}^\ast\varrho_{F}(\underline x)
\frac{\frac{i}{2\pi}\Theta_{S_F }^+}{\underline x-\frac{i}{2\pi}d\Theta_{S_F }^+}
\frac{\frac{i}{2\pi}\Theta_{S_F }^-}{\underline x-\frac{i}{2\pi}d\Theta_{S_F }^-}
,  \phi}\\
+\O(m^{-\infty}).
\end{multline}
This deals with the first term  of the right-hand side of \eqref{bigcomputI2}.
Using equations \eqref{chiFdef} and \eqref{eqclosed}
the second term $R_2^F(m)$ given in \eqref{R2F1} is seen to be equal to
\begin{equation}\label{R2Fmfinal}
R_2^F(m)=\int_\R\int_{M}(d\chi_F)\,
e^{2\pi i m x\J}e^{m\omega}e^{-d_\g\tilde\beta_F}
\tilde\beta_{m,F}\,\varrho(x)\phi(x)dx\,,
\end{equation}
where the basic forms
$\tilde\beta_F,\,\tilde\beta_{m,F}\in\Omega^*(U_F\backslash F)^{S^1}_{\textup{bas}}$
are characterized
by $\Phi_F^*\tilde\beta_F:={(\Psi_F^{-1})}^*\bar \beta_F$ and
$\Phi_F^*\tilde\beta_{m,F}:={(\Psi_F^{-1})}^*\bar \beta_{m,F}$
with $\Phi_F$ restricted to  $V_F\cap \Sigma_{F\bullet\bullet}$; 
they can be extended smoothly from
$\Phi_F(V_F\cap \Sigma_{F\bullet\bullet})$ to $U_F\backslash F$
thanks to Formulas \eqref{eq:betabar} and \eqref{eq:07.07.2022}. 
Collecting everything  we finally get
\begin{multline}\label{I2epsfinal}
I_2^F(m)=\int_{S_F}   e^{ m\nu_{S_F}^\ast  \omega_{F}}
\kappa_F(\varrho_{F})+\eklm{ \int_{S_F}e^{ m \nu_{S_F}^\ast \omega_{F}}\nu_{S_F}^\ast\varrho_{F}(\underline x)
\frac{\frac{i}{2\pi}\Theta_{S_F }^+}{\underline x-\frac{i}{2\pi}d\Theta_{S_F }^+}
\frac{\frac{i}{2\pi}\Theta_{S_F }^-}{\underline x-\frac{i}{2\pi}d\Theta_{S_F }^-}
,  \phi}\\
\qquad\qquad\qquad-R_2^F(m)+\O(m^{-\infty})\,.
\end{multline}

 \subsection{Full asymptotic expansion of the Witten integral}
 \label{fullWitsec}
 
 Let us now combine the computations of the whole section to
 establish the full asymptotic expansion of the Witten integral \eqref{eq:mainint}.  
 
 \begin{thm}\label{thm:wittenasymp} Let $\phi$ be identically equal
 to $1$ in a neighborhood of $0$.
 Then, for any equivariantly closed form $\varrho \in \Omega_{S^1}^\ast(M)$, one has the asymptotic expansion
   \begin{multline}\label{Wfinal}
 \<\W_m(\varrho),\phi\>=\int_{\M_0^{\mathrm{reg}}}e^{m\omega_0}
 \kappa(\varrho)+\sum_{F\in \mathcal{F}^0_\pm}\bigg\langle\int_{S_{F}}\,e^{m\nu_{S_F}^\ast\omega_{F}}\nu_{S_F}^\ast
\varrho_{F}(\underline{x}_\pm)
\frac{\frac{i}{2\pi}\Theta_{S_F}}{\underline{x}_\pm-\frac{i}{2\pi}d\Theta_{S_F}},\phi\bigg\rangle\\
  +\sum_{F\in \mathcal{F}^0_\mathrm{exc}} \int_{S_F} e^{ m \nu_{S_F}^\ast \omega_{F}}
\kappa_F(\varrho_{F})
\;+\;\eklm{ \int_{S_F}e^{ m \nu_{S_F}^\ast\omega_{F}}\nu_{S_F}^\ast\varrho_{F}(\underline x)
\frac{\frac{i}{2\pi}\Theta_{S_F }^+}{\underline x-\frac{i}{2\pi}d\Theta_{S_F }^+}
\frac{\frac{i}{2\pi}\Theta_{S_F }^-}{\underline x-\frac{i}{2\pi}d\Theta_{S_F }^-}
,  \phi}\;+\;\O(m^{-\infty})
\end{multline} 
as $m\to+\infty$, where $F\in \mathcal{F}^0_\mathrm{exc}$
is defined as in \eqref{F0exc}.
  Furthermore, each summand in \eqref{Wfinal} only depends on the
equivariant cohomology class $[\varrho] \in H_{S^1}^*(M)$.
 \end{thm}
 \begin{proof} Let us first establish \eqref{Wfinal} under Assumption \ref{ass:rho},
  assuming that  $0\in\Int\J(M)$.  
Having treated in the previous subsections all terms in \eqref{eq:25.5.2019} except  $\<\W^{\textup{reg}}_m(\varrho),\phi\>$, we are left with studying the latter integral. To this end, we use the $S^1$-equivariant
diffeomorphism $\Phi_W:W\xrightarrow{\sim}\J^{-1}(\{0\})^\mathrm{reg}\times(-\delta,\delta)$
of Lemma \ref{lem:globalretraction}, which provides the coordinate
$q:=\J(p)$ on $W$.
This coordinate coincides for all $F\in\F^0$ with the coordinate
$q$ of \eqref{eq:rscoord} via the diffeomorphism \eqref{eq:Psiindef}
when restricted to $U_F\cap W$.
Now, following \cite[Proof of Theorem\ 3.3]{meinrenken96}, and analogously to Equation \eqref{eq:01.07.2022} one has
\begin{equation}
\omega|_W=\mathrm{ret}_{0}^*\mathrm{inc}^*_{0}\omega+
d(q\,\mathrm{ret}_{0}^*\Theta)+d_\g\tilde \beta,\label{eq:23589}
\end{equation}
where $\textup{inc}_0:\J^{-1}(\{0\})^{\mathrm{reg}}\to M$ is the inclusion, $\Theta\in\Omega^1(\J^{-1}(\{0\})^{\mathrm{reg}},\R)$ is any connection for the $S^1$-action on
$\J^{-1}(\{0\})^{\mathrm{reg}}$ in the sense of \eqref{connfla}, and
$\tilde\beta\in\Omega^*(W)^{S^1}_{\textup{bas}}$ is a basic form in the sense of
\eqref{basdef} satisfying $\inc_{0}^*\tilde\beta\equiv 0$. Furthermore, thanks to \eqref{retVF} and comparing with
\eqref{eq:01.07.2022},
we can choose the forms in \eqref{eq:23589} in such a way that for all $F\in\F^0$ 
we have 
\bqn
\Phi_F^*(\Theta|_{U_F\cap \J^{-1}(\{0\})^{\mathrm{reg} }})=(\Psi_F^{-1})^*\pr_{S_F}^*\Theta_{S_F} \quad\text{and}\quad  \Phi_F^*(\tilde \beta|_{U_F\cap W })=(\Psi_F^{-1})^*\bar\beta_F\,,
 \eqn
 with $\Phi_F$ suitably restricted. In the same way as in  \eqref{emdb-1}, let us write
\begin{equation*}\label{emdb-1bis}
e^{md_\g\tilde \beta  } -1 =  d_\g \tilde \beta_{m} \,,
\end{equation*}
with $\tilde \beta_{m}\in\Omega^*(W)^{S^1}_{\textup{bas}}$ satisfying
$\Phi_F^*(\tilde \beta_m|_{ U_F\cap W})=(\Psi_F^{-1})^*\bar \beta_{m,F}$ for all $F\in\F^0$. In particular, recalling the definitions
of the forms appearing in  \eqref{R2Fmfinal},
we have that $\tilde \beta=\tilde\beta_F$ and $\tilde \beta_m=\tilde\beta_{m,F}$ on $U_F\cap W$.
Following \cite[Proof of Theorem\ 3.3]{meinrenken96},
which boils down to carrying over the computations from Section
\ref{sec:indefinitefixedpoints}
to the present situation, and using Assumption
\ref{ass:rho}, we then readily obtain
 \begin{align}
\nonumber \<\W^{\textup{reg}}_m(\varrho),\phi\>&=\int_{\J^{-1}(\{0\})^{\mathrm{reg}}}
 \left(1-\sum_{F\in \F^0}\chi_F\right)e^{m \, \textup{inc}_{0}^*\omega}
 \textup{inc}^*_{0}\varrho\Big(\frac{i}{2\pi} d \Theta\Big)\wedge\Theta\\
 \nonumber &\qquad\qquad +\sum_{F\in\F^0}\int_\R\int_{M}(d\chi_F)\, e^{2\pi i m x\J} e^{m\omega}e^{-d_\g\tilde\beta}
\tilde \beta_{m}\varrho(x)\phi(x)dx+\O(m^{-\infty})\\
&=\int_{\J^{-1}(\{0\})^{\mathrm{reg}}}
 \left(1-\sum_{F\in \F^0}\chi_F\right)e^{m \, \textup{inc}_{0}^*\omega}
 \textup{inc}^*_{0}\varrho\Big(\frac{i}{2\pi} d \Theta\Big)\wedge\Theta
 +R_2^F(m)+\O(m^{-\infty})\label{Wtopasy}
 \end{align}
  as  $m\to+\infty$,  with $R_2^F(m)$ as in \eqref{R2Fmfinal}.
Inserting  \eqref{Wtopasy} into the original expression \eqref{eq:25.5.2019} for $\<\W_m(\varrho),\phi\>$ while taking into account \eqref{I1I2I3I4}, the first term on the right-hand side of \eqref{Wtopasy} combines with the integrals $I_1^F(m)$ computed in \eqref{eq:prelimexpindefI} in such a way that all terms involving the cutoff functions $\chi_F$ disappear. Similarly, the remainder term $R_2^F(m)$ appears in \eqref{Wtopasy} and in \eqref{I2epsfinal} with opposite signs and thus gets cancelled out,  
 which leaves only the term \eqref{I2epsfinalpre}.  Summing up,  we get 
 \begin{align}
 \nonumber \<\W_m(\varrho),\phi\>&=\<\W^{\textup{reg}}_m(\varrho),\phi\>+\sum_{F\in \mathcal{F}^0} \big ( I_1^{F}(m)+
 I_2^{F}(m)\big ) +\O(m^{-\infty})\\
 &=\int_{\J^{-1}(\{0\})^\mathrm{reg}}e^{m \, \textup{inc}_0^*\omega}
 \textup{inc}_0^*\varrho\Big(\frac{i}{2\pi} d \Theta\Big)\,\wedge\,\Theta
 +\sum_{F\in \mathcal{F}^0}\tilde I_2^{F}(m)\label{Wtopfinal}\\
 \nonumber &\qquad +\O(m^{-\infty})
 \end{align}
as $m\to+\infty$. If we now insert the expression \eqref{W0mdefcomputfinal} for the integrals $\<\W^{F}_m(\varrho),\phi\>$ with definite $F$ and the expression \eqref{I2epsfinalpre} for the integrals $\tilde I_2^F(m)$
and apply Lemma \ref{intXi=int} to the regular Kirwan map  \eqref{Kirwantilde}, we get the asymptotics \eqref{Wfinal}
for any equivariantly closed $\varrho \in \Omega_{S^1}^\ast(M)$
satisfying Assumption \ref{ass:rho}.
 The asymptotics \eqref{Wfinal} in case $0\in\partial\,\J(M)$ follow in the same way from the computations of
Section \ref{sec:defcase}.

Let us now deal with the case of a general
equivariantly closed
form $\varrho \in \Omega_{S^1}^\ast(M)$ which not necessarily
satisfies Assumption \ref{ass:rho}.
First, note that  thanks to Lemma \ref{assumption}  the left-hand
side of \eqref{Wfinal} only depends on the equivariant cohomology
class of $\varrho$.
On the other hand, by the same lemma
one can write $\varrho=\tilde\varrho+d_\g\beta$ with
$\tilde\varrho \in \Omega_{S^1}^\ast(M)$ satisfying Assumption \ref{ass:rho}, so that we can apply \eqref{Wfinal} to $ \<\W_m(\varrho),\phi\>= \<\W_m(\tilde \varrho),\phi\>$. Since the equivariant homotopy Lemma \ref{eqHom} implies that we can choose $\beta$ such that  $\beta_F:=\inc_F^*\beta\equiv 0$, it follows from Lemma \ref{Kirwanexact}  that the first term on the right-hand side of  \eqref{Wfinal} remains unchanged if we replace $ \varrho$ by $\tilde \varrho$; as $ \varrho_F=\tilde\varrho_F$, all other terms on the right-hand side are also unchanged, and we obtain  \eqref{Wfinal} for $\varrho$.

Finally note that 
Lemma \ref{intXi=int}, 
applied to the $S^1$-actions
on $S_F^{\pm}$ in \eqref{S1SFS1} with connections
$\Theta^{\pm}_{S_F}$, respectively, 
implies that all terms on the right-hand side  of  \eqref{Wfinal}, except maybe the first one, 
depend only on the equivariant cohomology
class of $\varrho$. Since this is also true for the left-hand side
by Lemma \ref{assumption}, this implies that
the first term of the right-hand side only depends on the equivariant cohomology
class of $\varrho$ up to an error of $\O(m^{-\infty})$ as $m\to+\infty$.
But as this term is polynomial in $m\in\N$, this error actually vanishes
identically in $m\in\N$. This concludes the proof of the theorem. 
 \end{proof}

 \begin{rem}\label{rem:DH2}
 Recalling the definition of $S_F$ from \eqref{SFdef}, we see   that 
 $S_F=S_F^+\times_F S_F^-$ as a fiberwise product, where we set
 $S_F^{\pm}:=P_F\times_{K_F}S_\pm^{2l^{\pm}+1}$. On the other hand,
 the $1$-forms $\Theta_{S_F}^\pm$ are connections for the
 $S^1$-actions on $S_F^{\pm}$, respectively. Thus,  adapting the proof of Duistermaat and Heckman in
\cite[(2.11)-(2.31)]{duistermaat-heckman83}, and 
using the multiplicativity of the Euler class
one gets from \eqref{Wtopfinal} that as $m\to+\infty$ one has
\begin{multline}\label{W0defrmk}
\<\W_m(\varrho),\phi\>=\int_{\M_0^{\mathrm{reg}}}e^{m\omega_0}
 \kappa(\varrho)+\sum_{F\in \mathcal{F}^0_\pm} \eklm{ \int_{F}\frac{e^{ m \omega_F}
\varrho_{F}(\underline{x}_\pm)}{e_F(\underline{x}_\pm)}
,  \phi}\\
+
\sum_{F\in \mathcal{F}^0_\mathrm{exc}}\int_{S_F}   e^{ m \nu_{S_F}^\ast \omega_{F}}
\kappa_F(\varrho_{F})+ \eklm{ \int_{F}\frac{e^{ m \omega_F}
\varrho_{F}(\underline x)}{e_F(\underline x)}
,  \phi}+\O(m^{-\infty})\,,
\end{multline}
where $e_F(x)\in\Omega^*_{S^1}(F)$ is the equivariant Euler class
of the normal bundle $\nu_{\Sigma_F}:\Sigma_F\to F$. 
Note again that the integrands in the last terms coincide with the integrand
appearing in Berline-Vergne localization
formula \cite{berline-vergne82}. As already pointed out
in Remark \ref{rem:DH1}, 
we will not use this fact for the computation of
the Riemann-Roch numbers in the next section, but instead provide a direct proof based on \eqref{Wfinal} and the Kirillov formula
of Theorem \ref{Kirfla}.
\end{rem}
 
  \section{Invariant Riemann-Roch formula}
\label{sec:6}

In this section, we will proceed to the proof of our main result Theorem \ref{mainth}. 

\subsection{Asymptotics of Riemann-Roch numbers}\label{sec:asyRR}
In what follows, we will  relate the asymptotics as $m\to+\infty$ of the
$G$-invariant Riemann-Roch numbers of Definition \ref{RRGeqdef} for $G=S^1$
with the asymptotics of the Witten integral \eqref{eq:mainint}.
As  in Theorem \ref{mainth}, we assume  that the action of $S^1$ is free on
$\J^{-1}(\{0\})\setminus M^{S^1}$. Recall the 
identification \eqref{identfla} sending $X\in\g$ to $x\in\R$. We then have the following 

\begin{prop}\label{RRasyprop}
Under the assumptions of Theorem \ref{mainth}, there exists a neighborhood
$V_0\subset\g$ of $0\in \g$ and a test function 
$\psi\in\CT(\exp (V_0),\R)$ which is identically equal to $1$ around $e\in S^1$ such that 
\begin{multline}\label{RRasyfla3}
\text{\emph{RR}}^{S^1}(M,L^m)=\langle\W_{m}(\varrho),\phi\rangle
\\ +\sum_{F\in \F^0}\int_{S^1}(1-\psi(g))\int_{F}\,
\frac{\Td(F)\,e^{m\omega_F}}
{\det_{\Sigma_F}(1-g\exp(R^{\Sigma_F}/2\pi i))}\,dg + \O(m^{-\infty})
\end{multline}
as $m\to+\infty$, where  $\langle\W_{m}(\varrho),\phi\rangle$ is the Witten integral \eqref{eq:mainint} evaluated on $\varrho(x):=\Td_\g(M,X)\in\Omega^*_{S^1}(M)$ and $\phi\in \CT(V_0,\R)$
is defined by $\phi(x):=\psi(\exp(X))\in \CT(V_0,\R)$ for all $x\in\R$. 
\end{prop}
\begin{proof}
In view of the compactness of $S^1$ we know that there exists an 
$\epsilon>0$ and a
finite subset $\{g_0=e,g_1,\cdots,  g_r\}\subset S^1$
such that the Kirillov formula of Theorem \ref{Kirfla} can be applied with $g=g_j$ for all $|X|<\epsilon$ and all $1\leq j\leq r$, and such that $\{g_j\exp(V_0)\}_{j=1}^r$ is an open cover of $S^1$,
where $V_0\simeq (-\epsilon,\epsilon)\subset\g$ under the identification
$\g\simeq\R$ of \eqref{identfla}. Let $\{\psi_j\}_{j=1}^r$
be an associated partition of unity, and set
$\phi_j(x):=\psi_j(\exp(X))\in \CT(V_0,\R)$ for all $x\in\R$
under the identification \eqref{identfla}. We set
$\psi:=\psi_0$ and $\phi:=\phi_0$.
Note also that for any $1\leq j\leq r$ we have $\mathfrak{z}_{g_j}=\g$ in Theorem \ref{Kirfla} since $S^1$ is abelian.
Applying Theorem \ref{Kirfla} to \eqref{RRintG} we get with \eqref{ChernLm}
\begin{align}\label{RRasyfla1}
 \text{RR}^{S^1}(M,L^m)
 \nonumber &=\sum_{j=0}^r\int_{g_j\exp(V_0)}\,\psi_j(g)\chi^{(m)}(g)\,dg
 =\sum_{j=0}^r\int_{\exp(V_0)}\,\psi_j(g_jg)\chi^{(m)}(g_jg)\,dg\\
  &=\underbrace{\int_{V_0}\,\int_M\,\Td_\g(M,X)\,\ch_\g(L^m,X)\,\psi(\exp(X))\,dX}_{=\<\W_m(\varrho),\phi\>}\\
 \nonumber  &+\sum_{j=1}^r\int_{V_0}\psi_j(g_j\exp(X))\,\int_{M^{g_j}}\,
\frac{\Td_{\g}(M^{g_j},X)\,\Tr_{L^m}[g^{-1}_j]\ch_{\g}(L^m|_{M^{g_j}},X)}
{\det_{\Sigma_{M^{g_j}}}(1-g_j\exp(R^{\Sigma_{M^{g_j}}}_{\g}(X)/2\pi i))}\,
\,dX\,.
\end{align}
To establish \eqref{RRasyfla3},
we thus need to identify the second term of the right-hand side of 
\eqref{RRasyfla1} with the second term of the right-hand side of  \eqref{RRasyfla3} up to
$O(m^{-\infty})$.

For this sake,  note that by the 
Kostant formula \eqref{Kostant} and the  
identification \eqref{identfla} we have on $M^{S^1}$ the equality 
\begin{equation*}
\Tr_{L^m}[\exp(-X)]=e^{2\pi imx\J}\,.
\end{equation*}
For each $1\leq j\leq r$, let $X_j\in\g$ be such that $g_j=\exp(X_j)$,
and let $x_j\in\R$ be its image under the identification
\eqref{identfla} sending $X\in\g$ to $x\in\R$. Further, for any $1\leq j\leq r$ let us write $\F_j$ for the set of connected
components of $M^{g_j}$, and for each $F\in\F_j$,
consider the associated isotypic decomposition
as in \eqref{eq:splitting111k} of its normal bundle $\nu_{\Sigma_F}:\Sigma_F\to F$
with respect to the induced ${S^1}$-action. Write $R^{\Sigma_F}=\sum_{k\in W} R^{\Sigma_F^{(k)}}$ for the splitting
of the curvature of the connection $\nabla^{\Sigma_F}$
in this decomposition. Since $M^{S^1}\subset M^{g_j}$ and the $S^1$-action is free on
$\J^{-1}(\mklm{0})^\text{reg}$ we know that $\F^0\subset\F_j$, where $\F^0$ is as in Notation \ref{def:F}.
Furthermore, by the local normal form theorem
of Proposition \ref{prop:localnormform} and the compactness of  $M^{g_j}$ we know that
there is a $\delta>0$ such that
\bq
|\J||_F>\delta\qquad\forall\;F\in\F_j\backslash\F^0.\label{eq:estfornonstphase}
\eq
Now, from \eqref{Toddeq} it is clear that for all $F\in\F_j$ 
the forms $\Td_{\g}(F,x)\in\Omega^*(M)$ induced by
the equivariant Todd class $\Td_{\g}(F,X)\in\Omega^*(M)\in\Omega^*_G(M)$ via the identification
\eqref{identfla} are smooth in $x\in\R$, and that for all $F\in\F^0$ we have $\Td_{\g}(F,x)=\Td(F)$.
Using \eqref{preq} and the non-stationary phase principle we then get  
\begin{align*}
&\int_{V_0}\,\psi_j(g_j\exp(X))\int_{M^{g_j}}\,
\frac{\Td_{\g}(M^{g_j},X)\,\Tr_{L^m}[g^{-1}_j]\ch_{\g}(L^m|_{M^{g_j}},X)}
{\det_{\Sigma_{M^{g_j}}}(1-g_j\exp(R^{\Sigma_{M^{g_j}}}_{\g}(X)/2\pi i))}\,
\,dX\\
&=\sum_{F\in\F_j}\int_{-\epsilon}^{\epsilon}\phi_j(x+x_j)\,\int_{F}\,e^{2\pi im(x-x_j)\J}
\frac{\Td_{\g}(F,x)\,e^{m\omega_F}}
{\prod_{k\in W}\det_{\Sigma_F^{(k)}}(1-e^{2\pi i k(x+x_j)}\exp(R^{\Sigma_F^{(k)}}/2\pi i))}\,
\,dx\\
&\stackrel{\hspace*{-0.45em}\text{\eqref{eq:estfornonstphase}}\hspace*{-0.25em}}{=}\sum_{F\in\F^0}\int_{-\epsilon}^{\epsilon}\phi_j(x+x_j)\,\int_{F}\,
\frac{\Td(F)\,e^{m\omega_F}}
{\prod_{k\in W}\det_{\Sigma_F^{(k)}}(1-e^{2\pi i k(x+x_j)}\exp(R^{\Sigma_F^{(k)}}/2\pi i))}\,dx+\O(m^{-\infty})\\
&=\sum_{F\in\F^0}\int_{g_j\exp(V_0)}\,\psi_j(g)\,\int_{F}\,
\frac{\Td(F)\,e^{m\omega_F}}
{\det_{\Sigma_F}(1-g\exp(R^{\Sigma_F}/2\pi i))}\,dg+\O(m^{-\infty})
\end{align*}
as $m\to+\infty$. Summing over all $1\leq j\leq r$
and using that $\sum_{j=1}^r\psi_j=1-\psi$, we get that
the second term of the right-hand side of 
\eqref{RRasyfla1} equals
the second term of the right-hand side of \eqref{RRasyfla3} up to an error of size $\O(m^{-\infty})$, concluding the proof.

%
%
%
%
%
\end{proof}

\subsection{Proof of Theorem \ref{mainth}}
We are now ready to  prove Theorem \ref{mainth} combining Proposition \ref{RRasyprop} with
the full asymptotic expansion of the Witten integral in Theorem \ref{thm:wittenasymp}. The case of interest is when $0\in\Int\J(M)$, so let's treat this first. Let $\varrho$ and $\phi$ be as in the previous proposition, and take $\tilde\psi\in\Cinft({S^1},\R)$ with compact support in the small neighborhood $U_e$ of $e\in {S^1}$, but such that $e\notin\text{supp}\,\tilde\psi$, and write $\tilde \phi(x):=\tilde \psi(\exp(X))\in \CT(V_0,\R)$ 
for the induced function. Since $\psi + \tilde \psi\in \CT(U_e,\R)$ is identically $1$ close to $e$, Formula \eqref{RRasyfla3} also holds with the cut-off function $\psi$ replaced by $\psi+ \tilde \psi$. Therefore, taking the difference of the two resulting formulas, whose left-hand sides do not depend on the cut-off functions, Formula  \eqref{Wfinal} applied to $\phi+\tilde\phi$ and $\tilde\phi$ implies 
\begin{multline}\label{Kirilapplinew}
\sum_{F\in \F^0}\int_{S^1} \tilde\psi(g)\int_{F}\,
\frac{\Td(F)\,e^{m\omega_F}}
{\det_{\Sigma_F}(1-g\exp(R^{\Sigma_F}/2\pi i))}\,dg\\
 =\sum_{F\in \mathcal{F}^0} \int_\R \int_{S_F}e^{ m \nu_{S_F}^\ast\omega_{F}}\nu_{S_F}^\ast\varrho_{F}(x)
\frac{\frac{i}{2\pi}\Theta_{S_F }^+}{x-\frac{i}{2\pi}d\Theta_{S_F }^+}
\frac{\frac{i}{2\pi}\Theta_{S_F }^-}{x-\frac{i}{2\pi}d\Theta_{S_F }^-}
  \tilde \phi(x)  \, d x +\O(m^{-\infty})
\end{multline}
as $m\to+\infty$, where we took into account that the singular support of the distributions appearing in \eqref{Wfinal}
is given by $\{0\}\subset\R$, so that they turn into regular integrals over $\R$ as $0\notin\supp\tilde\phi$ by assumption.
Now, as  \eqref{Kirilapplinew} holds for all test functions 
$\tilde\psi\in\CT({U_e\setminus\{e\}},\R)$, one deduces from this for $x\neq 0$ the  identity of Laurent polynomials 
\begin{multline*}
\sum_{F\in \F^0}\int_{F}\frac{\Td(F)\,e^{m\omega_F}}
{\prod_{k\in W}\det_{\Sigma_F^{(k)}}(1-e^{2\pi i kx}\exp(R^{\Sigma_F^{(k)}}/2\pi i))}
\\
 =\sum_{F\in \mathcal{F}^0}  \int_{S_F}e^{ m \nu_{S_F}^\ast\omega_{F}}\nu_{S_F}^\ast\varrho_{F}(x)
\frac{\frac{i}{2\pi}\Theta_{S_F }^+}{x-\frac{i}{2\pi}d\Theta_{S_F }^+}
\frac{\frac{i}{2\pi}\Theta_{S_F }^-}{x-\frac{i}{2\pi}d\Theta_{S_F }^-} ,
\end{multline*}
since both sides are polynomials in $m\in\N$, so that the error
term $\O(m^{-\infty})$ vanishes. Comparing this with the full asymptotic expansion of the Witten
integral  \eqref{Wfinal} and using the notation introduced in \eqref{Laurentpvdef}, this implies
\begin{multline*}
\<\W_m(\varrho),\phi\>=\int_{\M_0^{\mathrm{reg}}}e^{m\omega_0}
 \kappa(\varrho) 
+\sum_{F\in \F^0} \int_{S_F}   e^{ m\nu_{S_F}^\ast \omega_{F}}
\kappa_F(\varrho_{F})\\
 +
\bigg\langle\int_{F}\frac{\Td(F)\,e^{m\omega_F}}
{\prod_{k\in W}\det_{\Sigma_F^{(k)}}
(1-e^{2\pi i k\underline x}\exp(R^{\Sigma_F^{(k)}}/2\pi i))},\phi\bigg\rangle+\O(m^{-\infty})\,.
\end{multline*}
Plugging this into Proposition \ref{RRasyprop} and expressing the distributions
\eqref{pvdef} and \eqref{pvdef2} as distributions on  the Lie group ${S^1}$ we get
\begin{multline}\label{RRS1final}
\text{RR}^{S^1}(M,L^m)=\int_{\M_0^{\mathrm{reg}}}e^{m\omega_0}
 \kappa(\varrho)
 +\sum_{F\in\F_+^0}\lim_{\epsilon\to 0^+}\int_0^{1}\,\widetilde{\chi}^{(m)}_F(e^{2\pi ix-\epsilon})\,dx\\
 +\sum_{F\in\F_-}\lim_{\epsilon\to 0^+}\int_0^{1}\,
\widetilde{\chi}^{(m)}_F(e^{2\pi ix+\epsilon})\,dx
+\sum_{F\in\F^0}\int_{S_F}   e^{ m\nu_{S_F}^\ast \omega_{F}}
\kappa_F(\varrho_{F})\\
+\frac{1}{2}  \left(\lim_{\epsilon\to 0^+}\int_0^{1}\,\widetilde{\chi}^{(m)}_F(e^{2\pi ix-\epsilon})\,dx+\lim_{\epsilon\to 0^+}\int_0^{1}\,\widetilde{\chi}^{(m)}(e^{2\pi ix+\epsilon})\,dx\right)
+\O(m^{-\infty})\,,
\end{multline}
where $\widetilde{\chi}^{(m)}_F$ is the meromorphic function
defined for any $F\in\F^0$ and $z\in\C$ by
\begin{equation*}\label{xitildedef}
\widetilde{\chi}^{(m)}_F(z):=\int_{F}\,
\frac{\Td(F)\,e^{m\omega}}
{\prod_{k\in W}\det_{\Sigma_F^{(k)}}(1-z^k\exp(R^{\Sigma_F^{(k)}}/2\pi i))}\,.\\
\end{equation*}
Notice that the dependence on the test function $\psi$ has cancelled out in \eqref{RRS1final}. Now, by a result of Meinrenken in
\cite[Theorem 5.1]{meinrenken96},
the left-hand side of \eqref{RRS1final} is an arithmetic polynomial in $m\in\N$,
while the terms of the right-hand side except $\O(m^{-\infty})$
are polynomials in $m\in\N$, 
so that the error term in \eqref{RRS1final} actually vanishes identically in $m\in\N$. 

Note that, as seen for instance from \cite[(2.5)]{DGMW95}, 
the poles of $\widetilde{\chi}^{(m)}_F(z)$ are contained in $\{0,1\}\subset\C$. Therefore, with a change of coordinates  the residue theorem implies  that 
\begin{equation*}
\begin{split}
\lim_{\epsilon\to 0^+}\int_0^{1}\,\widetilde{\chi}^{(m)}_F(e^{2\pi ix-\epsilon})\,dx&=\frac{1}{2\pi i}\lim_{\epsilon\to 0^+}\int_{\{|z|=e^{-\epsilon}\}}
\widetilde{\chi}^{(m)}_F(z)\,\frac{dz}{z} =\underset{z=0}{\mathrm{Res}}\,\frac{\widetilde{\chi}^{(m)}_F(z)}{z}\,.
\end{split}
\end{equation*}
Using a change of variable $z\mapsto z^{-1}$ we get in the same way
\begin{equation}\label{Resinfty}
\lim_{\epsilon\to 0^+}\int_0^{1}\,\widetilde{\chi}^{(m)}_F(e^{2\pi ix+\epsilon})\,dx=\underset{z=0}{\mathrm{Res}}\,\frac{ \widetilde{\chi}^{(m)}_F(z^{-1})}{z}\,.
\end{equation}
Inserting this in \eqref{RRS1final} and taking $m=1$ then finally yields \eqref{mainthfla} in case that $0\in\Int\J(M)$. The case $0\in\partial\,\J(M)$   follows in a completely analogous way  and reduces to  the original proof of Duistermaat, Guillemin,
Meinrenken, and Wu in \cite[\S\,2]{DGMW95}. 
 \qed

\begin{rem}\label{rem:DH3}
Note that Formula \eqref{Kirilapplinew} can also be obtained from the formula
for the Witten integral given in Remark \ref{rem:DH2} by using the
explicit Formula \eqref{Toddeq} for the equivariant Todd class
and the explicit formula for the equivariant Euler class following
for instance \cite[\S\,8.1]{berline-getzler-vergne}.
This fact is actually at the basis of the deduction of the Kirillov formula
of Theorem \ref{Kirfla} from the equivariant index theorem of
\eqref{eqindfla}. In the definite setting of Section \ref{sec:defcase},
our method actually
gives an alternative proof of this fact, using the theory of
distributions instead of the standard methods of equivariant cohomology.
\end{rem}

\subsection{Examples}
\label{sec:ex}
Let us close by illustrating Theorem \ref{mainth} through a family of examples. Fix an $m$-tuple $(k_1,\cdots,  k_m)\in\Z^m$ with $m\in\N$. We
consider the product $M=\Pi_{j=1}^m S^2$ of $m$ copies of the
$2$-sphere $S^2$ endowed with the symplectic form $\omega$  whose pullback to each $S^2$-factor is the standard volume form of volume $1$. We regard $M$ as  equipped with the diagonal $S^1$-action such that
$\phi\in\R/\Z\simeq S^1$ acts on the $j$-th sphere by  a rotation of angle
$2\pi k_j\phi$ around the $z$-axis of $S^2\subset\R^3$. This action is
Hamiltonian for the moment map
\begin{equation}
\J:=\sum_{j=1}^m\pi_j^*\J_{k_j}-\sum_{j=1}^m k_j\,,\label{eq:Jprodspheres}
\end{equation}
where for any $k\in\Z$, the map $\J_k:S^2\to\R$ denotes the moment map
associated with the $S^1$-action of weight $k$, which is explicitly given by $\J_k(x,y,z)=kz$ in the Euclidean coordinates $x,y,z$ of $\R^3\supset S^2$. The connected components of the fixed point set $M^{S^1}$ are all isolated
points, given by products of north and south poles, and the constant term subtracted in \eqref{eq:Jprodspheres} ensures that at least the product of all north poles (where $z=1$) is contained in $\J^{-1}(\{0\})$, with weights counted with multiplicity
given by $(k_1,\ldots, k_m)\in\Z^m$.
This shows in particular that any given
set of weights $W\subset\Z$ with any multiplicities can occur.
Theorem \ref{mainth} then
provides an explicitly computable formula for the $S^1$-invariant
Riemann-Roch numbers $RR(M,L)^{S^1}$, and under the combinatorial condition \eqref{weightcondext} the regular term can be expressed in terms of characteristic classes involving only the topology of the resolution $\widetilde\M_0$.

In the particular case of $M=S^2\times S^2$ with weights $k_1=-1$ and $k_2=1$,
the partial resolution $\widetilde\M_0$ is diffeomorphic to $S^2$,
and the quotient map $\widetilde Z\to\widetilde \M_0$
is a trivial $S^1$-principal bundle. One then checks
that we recover the usual Riemann-Roch
formula for the sphere $S^2$ as the right-hand side of \eqref{mainthfla},
and that the second and third terms of the right-hand side of Formula
\eqref{mainthfla} vanish, so that Formula \eqref{mainthfla} reads
$RR(S^2\times S^2,L^m\boxtimes L^m)^{S^1}=RR(S^2,L^m)$ for all $m\in\N$
and $L$ the prequantizing line bundle of $(S^2,\omega)$,
as one can easily compute explicitly.


\begin{thebibliography}{10}

\bibitem{AS68}
M.~F. Atiyah and G.~B. Segal, \emph{The index of elliptic operators. {II}},
  Ann. Math. (2) \textbf{87} (1968), 531--545.

\bibitem{berline-getzler-vergne}
N.~Berline, E.~Getzler, and M.~Vergne, \emph{Heat kernels and {D}irac
  operators}, Springer-Verlag, Berlin, Heidelberg, New York, 1992.

\bibitem{berline-vergne82}
N.~Berline and M.~Vergne, \emph{Classes caract\'eristiques \'equivariantes.
  {Formule} de localisation en cohomologie \'equivariante}, C. R. Acad. Sci.
  Paris \textbf{295} (1982), 539--541.

\bibitem{BL91}
J.-M. Bismut and G.~Lebeau, \emph{Complex immersions and {Q}uillen metrics},
 Publ. Math. Inst. Hautes Etudes Sci. \textbf{74} (1991), 1--297.

\bibitem{bott-tu}
R.~Bott and L.~W. Tu, \emph{Differential forms in algebraic topology},
  Springer-Verlag, New York, Heidelberg, Berlin, 1982.

\bibitem{kuester-ramacher20}
B.~Delarue and P.~Ramacher, \emph{Asymptotic expansion of generalized {W}itten
  integrals for {H}amiltonian circle actions}, J. Symplectic Geom. \textbf{19}
  (2021), no.~6, 1281--1337.

\bibitem{delarue-schmitt-ramacher23}
B.~Delarue, M.~Schmitt, and P.~Ramacher, \emph{Singular cohomology of symplectic
  quotients by circle actions and Kirwan surjectivity}, Transform. Groups \textbf{31} (2026), 1343--1389.

\bibitem{DGMW95}
J.~J. Duistermaat, V.~W. Guillemin, E.~Meinrenken, and S.~D. Wu,
  \emph{{Symplectic reduction and Riemann-Roch for circle actions}}, Math. Res.
  Lett. \textbf{2} (1995), 259--266.

\bibitem{duistermaat-heckman83}
J.~J. Duistermaat and G.~Heckman, \emph{{Addendum to "On the variation in the
  cohomology of the symplectic form of the reduced phase space"}}, Invent.
  Math. \textbf{72} (1983), 153--158.

\bibitem{goertscheszoller}
O.~{Goertsches} and L.~{Zoller}, \emph{{Equivariant de Rham cohomology: theory
  and applications}}, {S\~ao Paulo J. Math. Sci.} \textbf{13} (2019), no.~2,
  539--596.

\bibitem{guillemin-sternberg84}
V.~Guillemin and S.~Sternberg, \emph{A normal form for the moment map}, In:
  Differential Geometric Methods in Mathematical Physics, Proc. XIth Annu. Conf., Jerusalem 1982, Math. Phys. Stud. \textbf{6}, 161--175, D. Reidel, Dortrecht.

\bibitem{GS89}
\bysame, \emph{Birational equivalence in the symplectic category}, Invent.
  Math. \textbf{97} (1989), no.~3, 485--522.

\bibitem{GLS96}
V.~Guillemin, E.~Lerman, and S.~Sternberg, \emph{Symplectic
  fibrations and multiplicity diagrams}, Cambridge: Cambridge University Press, 1996.

\bibitem{hoermanderI}
L.~H{\"{o}}rmander, \emph{The analysis of linear partial differential
  operators}, vol.~I, Springer--Verlag, Berlin, Heidelberg, New York, 1983.

\bibitem{MMC19}
C.-Y. Hsiao, X.~Ma, and G.~Marinescu, \emph{Geometric quantization on {CR}
  manifolds}, Commun. Contemp. Math. (2023),
  https://doi.org/10.1142/S0219199722500742.

\bibitem{JKKW03}
L.~C. Jeffrey, Y.~Kiem, F.~C. Kirwan, and J.~Woolf, \emph{Cohomology pairings
  on singular quotients in geometric invariant theory}, Transformation Groups
  \textbf{8} (2003), no.~3, 217--259.

\bibitem{jeffrey-kirwan95}
L.~C. Jeffrey and F.~C. Kirwan, \emph{Localization for nonabelian group
  actions}, Topology \textbf{34} (1995), 291--327.

\bibitem{jeffrey-kirwan97}
\bysame, \emph{Localization and the quantization conjecture}, Topology \textbf{36} (1997), no.~3, 647--693.

\bibitem{kalkman95}
J.~Kalkman, \emph{{Cohomology rings of symplectic quotients.}}, {J. Reine Angew.
  Math.} \textbf{458} (1995), 37--52.

\bibitem{Kir84}
F.~C. Kirwan, \emph{Cohomology of quotients in symplectic and algebraic
  geometry}, Mathematical Notes, vol.~31, Princeton, NJ: Princeton University Press,
  Princeton, NJ, 1984.
  
  \bibitem{Kir84b}
F.~C. Kirwan, \emph{Convexity properties of the moment mapping. {III}}, Invent.
  Math. \textbf{77} (1984), no.~3, 547--552.

\bibitem{kirwan85}
\bysame, \emph{Partial desingularisations of quotients of nonsingular varieties
  and their {B}etti numbers}, Ann. Math. \textbf{122} (1985), 41--85.

\bibitem{LH89}
H.~B. Lawson and M.~L. Michelsohn, \emph{Spin geometry}, Princeton Mathematical
  Series, Princeton, NJ: Princeton University Press, 1989.

\bibitem{Ler95}
E.~Lerman, \emph{Symplectic cuts}, Math. Res. Lett. \textbf{2} (1995), no.~3,
  247--258.

\bibitem{lerman-tolman00}
E.~{Lerman} and S.~{Tolman}, \emph{{Intersection cohomology of \(S^1\)
  symplectic quotients and small resolutions.}}, {Duke Math. J.} \textbf{103}
  (2000), no.~1, 79--99.

\bibitem{MM07}
X.~Ma and G.~Marinescu, \emph{Holomorphic {M}orse inequalities and {B}ergman
  kernels}, Progress in Mathematics, vol. 254, Birkh\"auser Verlag, Basel,
  2007.

\bibitem{MZ14}
X.~Ma and W.~Zhang, \emph{Geometric quantization for proper moment maps: the
  {Vergne} conjecture}, Acta Math. \textbf{212} (2014), no.~1, 11--57.

\bibitem{meinrenken96}
E.~Meinrenken, \emph{{On Riemann-Roch formulas for multiplicities}}, {J. Am.
  Math. Soc.} \textbf{9} (1996), no.~2, 373--389.

\bibitem{Mei98}
\bysame, \emph{{Symplectic surgery and the Spinc--Dirac operator}}, Adv. Math.
  \textbf{134} (1998), no.~2, 240--277.

\bibitem{meinrenken-sjamaar}
E.~Meinrenken and R.~Sjamaar, \emph{Singular reduction and quantization},
  Topology \textbf{38} (1999), no.~4, 699--762.

\bibitem{paradan00}
P.-E. Paradan, \emph{{The moment map and equivariant cohomology with
  generalized coefficients.}}, {Topology} \textbf{39} (2000), no.~2, 401--444.

\bibitem{Par01}
\bysame, \emph{Localization of the {R}iemann-{R}och character}, J. Funct. Anal.
  \textbf{187} (2001), no.~2, 442--509.

\bibitem{Par03}
\bysame, \emph{Spin$^c$-quantization and the {K-multiplicities} of the discrete
  series}, Ann. Sci. {\'E}c. Norm. Sup{\'e}r. \textbf{36} (2003), no.~5,
  805--845.

\bibitem{ramacher13}
P.~Ramacher, \emph{Singular equivariant asymptotics and the moment map.
  {Residue} formulae in equivariant cohomology}, J. Symplectic Geom.
  \textbf{14} (2016), no.~2, 449--539.

\bibitem{SL91}
R.~Sjamaar and E.~Lerman, \emph{Stratified symplectic spaces and reduction}, Ann. Math. \textbf{134} (1991), no.~2, 375--422. 

\bibitem{Sja95}
R.~Sjamaar, \emph{Holomorphic slices, symplectic reduction and multiplicities
  of representations}, Ann. Math. \textbf{141} (1995), no.~1, 87--129.

\bibitem{Tel00}
C.~Teleman, \emph{The quantization conjecture revisited}, Ann. Math.
  \textbf{152} (2000), no.~1, 1--43.

\bibitem{TZ98}
Y.~Tian and W.~Zhang, \emph{An analytic proof of the geometric quantization
  conjecture of {Guillemin-Sternberg}}, Invent. Math. \textbf{132} (1998),
  no.~2, 229--259.

\bibitem{TZ99}
\bysame, \emph{Quantization formula for symplectic manifolds with boundary},
  Geom. Funct. Anal. \textbf{9} (1999), no.~3, 596--640.

\bibitem{Ver96}
M.~Vergne, \emph{Multiplicities formula for geometric quantization, {Part I}},
  Duke Math. J. \textbf{82} (1996), no.~1, 143--179.

\bibitem{witten92}
E.~{Witten}, \emph{{Two dimensional gauge theories revisited}}, {J. Geom.
  Phys.} \textbf{9} (1992), no.~4, 303--368.
  
  \bibitem{Zha99}
W.~Zhang, \emph{Holomorphic quantization formula in singular reduction},
  Commun. Contemp. Math. \textbf{1} (1999), no.~3, 281--293.
  \newline


\end{thebibliography}

\providecommand{\bysame}{\leavevmode\hbox to3em{\hrulefill}\thinspace}
\providecommand{\MR}{\relax\ifhmode\unskip\space\fi MR }
\providecommand{\MRhref}[2]{%
  \href{http://www.ams.org/mathscinet-getitem?mr=#1}{#2}
}
\providecommand{\href}[2]{#2}

\end{document}